\documentclass[a4paper]{article}

\usepackage[all]{xy}\usepackage[latin1]{inputenc}       
\usepackage[dvips]{graphics,graphicx}
\usepackage{amsfonts,amssymb,amsmath,xcolor,mathrsfs, amstext}
\usepackage{amsbsy, amsopn, amscd, amsxtra, amsthm, authblk, enumerate,dsfont}
\usepackage{upref}
\usepackage{geometry}
\usepackage[displaymath]{lineno}
\usepackage{float}
\usepackage{bbm}
\usepackage{tikz, tikz-3dplot}
\usetikzlibrary{calc}
\usepackage{subcaption}
\usepackage{bm}

\allowdisplaybreaks

\usepackage[colorlinks,
            linkcolor=blue,
            anchorcolor=green,
            citecolor=blue
            ]{hyperref}

\numberwithin{equation}{section}
\newtheorem{theorem}{Theorem}[section]
\newtheorem{lemma}[theorem]{Lemma}
\newtheorem{proposition}[theorem]{Proposition}

\newtheorem{definition}[theorem]{Definition}
\newtheorem{remark}[theorem]{Remark}
\newtheorem{prob}[theorem]{\bf Problem}

\renewcommand{\P}{\mathbb{P}}

\newcommand{\response}[1]{\color{black}{#1}}

\newcommand{\TT}{\mathbb{T}}

\title{Dynamical random field Ising model at zero temperature}
\author{Jian Ding\thanks{School of Mathematical Sciences, Peking University}  \qquad \qquad Peng Yang$^*$ \qquad \qquad Zijie Zhuang\thanks{Department of Statistics and Data Science, The Wharton School of Business, University of Pennsylvania}}

\begin{document}

\maketitle







\abstract{In this paper, we study the evolution of the zero-temperature random field Ising model as the mean of the external field $M$ increases from $-\infty$ to $\infty$. We focus on two types of evolutions: the ground state evolution and the Glauber evolution. For the ground state evolution, we investigate the occurrence of global avalanche, a moment where a large fraction of spins flip simultaneously from minus to plus. In two dimensions, no global avalanche occurs, while in three or higher dimensions, there is a phase transition: a global avalanche happens when the noise intensity is small, but not when it is large. Additionally, we study the zero-temperature Glauber evolution, where spins are updated locally to minimize the Hamiltonian. Our results show that for small noise intensity, in dimensions $d =2$ or $3$, most spins flip around a critical time $c_d = \frac{2 \sqrt{d}}{1 + \sqrt{d}}$ (but we cannot decide whether such flipping occurs simultaneously or not). We also connect this process to polluted bootstrap percolation and solve an open problem on it.}




\section{Introduction}
\label{sec:intro}

For a finite graph $G = (V,E)$, let $(h_v)_{v\in V}$ be i.i.d.\ Gaussian random variables with mean 0 and variance 1. Let $\mathbb{P}$ and $\mathbb{E}$ be the probability measure and expectation with respect to $(h_v)$. We define the Hamiltonian of the random field Ising model (RFIM) for a spin configuration $\sigma \in \{-1,1\}^V$ as follows:
\begin{equation}
\label{eq:def-hamiltonian}
    H_{G, M, \epsilon h} (\sigma) = -\sum_{u \sim v} \sigma_u\sigma_v - \sum_{v \in V} (M+\epsilon h_v ) \sigma_v,
\end{equation}
where $M \in \mathbb{R}$ is the mean of the external field, $\epsilon > 0$ is the noise intensity, and $u \sim v$ means $(u,v) \in E$. For a temperature $T \geq 0$, the random field Ising model is a probability measure on $\{ -1, 1\}^V$ whose density is proportional to $\exp( - \frac{1}{T} H_{G, M, \epsilon h}(\sigma))$. In particular, when $T = 0$, it is supported on the spin configurations minimizing the Hamiltonian, and these minimizers are known as ground states.

An interesting object of study is how the RFIM measure evolves as $M$ increases from $-\infty$ to $\infty$, where the mean of the external field $ M $ is also often referred to as the time parameter (in correspondence with ``evolution"). In this paper, we focus on the zero-temperature case. Then the ground state configurations naturally yield an increasing c\`adl\`ag process\footnote{\response{The increasing property follows from the monotonicity of the ground state configuration with respect to $M$. Furthermore, when the ground state configuration is not unique for some $M$, we select the maximum so that the resulting process is c\`adl\`ag. Note that we focus on a finite graph so there are only finitely many configurations.}} indexed by $M$ on $\{-1, 1\}^V$ and it grows from the all-minus to the all-plus configuration as $M$ increases from $-\infty$ to $\infty$. The first question we study in this paper is whether a global avalanche occurs in this evolution process, i.e., whether at some (random) time a positive fraction of vertices flip from minus to plus simultaneously. When there is no noise (namely, $\epsilon = 0$), a global avalanche occurs at $M=0$ where the spin configuration changes from all-minus to all-plus. For RFIM (namely, $\epsilon>0$), it turns out that in two dimensions, there will be no global avalanche. On the contrary, in dimensions at least three, the following phase transition occurs: when $\epsilon$ is sufficiently small, there will be a global avalanche; when $\epsilon$ is sufficiently large, there will be no global avalanche. See Section~\ref{subsec:gs-evolution} for the precise statements.

There has been extensive study in physics on dynamical RFIM, see e.g.~\cite{initiate1_PhysRevLett.70.3347,initiate2_DahmenKKRSS94,PhysRevLett.106.175701,PhysRevResearch.1.033060}. Partly due to the capability of simulation, most of these studies were not on the aforementioned ground state evolution, but on an evolution based on the Glauber dynamics (which we refer to as the Glauber evolution). In the Glauber evolution, we start from the all-minus configuration at $M = -\infty$ and increase the value of $M$. While we still focus on the zero-temperature case, our state at time $ M $ is no longer the ground state but a local minimizer obtained in the following way (via the Glauber dynamics): each vertex is randomly selected and updated from minus to plus if and only if such an update (weakly) decreases the Hamiltonian given in~\eqref{eq:def-hamiltonian}, and the final state (when no update is possible) gives the configuration at time $ M $. This also yields an increasing process indexed by $M$ and grows from the all-minus to the all-plus configuration. Note that this process has a nesting property, as we explain in Lemma~\ref{lem:Glauber-nesting}. Our second result is that for the zero-temperature Glauber evolution, when $\epsilon$ is small and the dimension is two or three, most of the spins flip around a critical time $c_d = \frac{2 \sqrt{d}}{1 + \sqrt{d}}$, which solves the equation $x^2 = d(2-x)^2$. See Section~\ref{subsec:glauber-evolution} for precise statements. The proof is based on connecting the mechanism of this evolution to another model called polluted bootstrap percolation introduced in~\cite{PollutedBP_original}. Along the way, we also solve an open problem in~\cite{Gravner-Holroyd-3D} about polluted bootstrap percolation, as incorporated in Theorem~\ref{thm:1.4}. It is natural to ask whether there will be a global avalanche in the Glauber evolution at zero temperature. We do not know the answer yet, except when $\epsilon$ is sufficiently large, it is easy to show that there will be no global avalanche.

We can also define the Glauber evolution at positive temperatures, although we do not study its properties in this paper. It starts with the all-minus configuration and evolves according to the Glauber dynamics. There are two parameters involved: the inverse temperature and the update rate. The two evolutions discussed above can both be incorporated into this single model by appropriately sending these two parameters to infinity. If we first send the update rate to infinity and then the inverse temperature, we obtain the ground state evolution. Conversely, if we first send the inverse temperature to infinity and then the update rate, the resulting evolution is the zero-temperature Glauber evolution. See Section~\ref{subsec:positive-glauber} for the precise statements.

\subsection{Ground state evolution}
\label{subsec:gs-evolution}

We first study the ground state evolution. To avoid issues about boundary condition and infinite graph, we focus on the toroidal graph $G = \mathbb{T}_N^d$ obtained by identifying vertices on the opposite sides of $[0, N]^d \cap \mathbb{Z}^d$ for an integer $N \geq 1$. Then, $|V| = N^d$. For $M \in \mathbb{R}$, $\epsilon>0$ and a sample of $(h_v)$, let $\tau_{G, M, \epsilon h}$ be the minimizer of $H_{G, M, \epsilon h}$ among $\{-1,1\}^V$. If there are multiple minimizers, we choose the one with the maximal number of plus spins. Then $(\tau_{G, M, \epsilon h} )_{M \in \mathbb{R}}$ is an increasing c\`adl\`ag process on $\{-1, 1\}^V$. Let $\mathcal{M}_G$ be the maximal number of spins flipping at a single time, namely,
$$
\mathcal{M}_G = \max_{M \in \mathbb{R}} |\{v \in V: \tau_{G, M^-, \epsilon h}(v) = -1,  \tau_{G, M, \epsilon h}(v) = 1\}|,
$$
where $\tau_{G, M^-, \epsilon h} := \lim_{t \nearrow M}\tau_{G, t, \epsilon h}$. Note that $\mathcal{M}_G$ is a random variable depending on the sample of $(h_v)$ as well as on $\epsilon$ and $G$.

\begin{theorem}
    \label{thm:2d-ground-state}
    For $d = 2$ and any $\epsilon >0$, there exists a constant $ C=C(\epsilon)>0 $ such that
    \begin{equation}
    \label{eq:thm-2d-ground-state}
        \mathbb{P}[ \mathcal{M}_{\TT_N^2} \leq C \log N ] \geq 1 -  N^{-100} \quad \mbox{for all  $N\geq 1$}.
    \end{equation}
\end{theorem}

\begin{theorem}
    \label{thm:3d-ground-state}
    The following hold for any $ d \geq 3 $.
    \begin{enumerate}[(1)]
        \item \label{claim-3d-ground-state-1} There exists a constant $C > 0$ depending on $d$ such that for all $\epsilon \geq C$ and $N \geq 1$,
        \begin{equation}
        \label{eq:thm-3d-ground-state}
        \mathbb{P}[ \mathcal{M}_{\TT_N^d} \leq C \log N ] \geq 1 -  N^{-100}.
        \end{equation}
        \item For each $\delta > 0$, there exists a constant $c > 0$ depending on $\delta$ and $d$ such that for all $0 < \epsilon < c$,
        \begin{equation}
        \label{eq:thm-3d-ground-state-2}
        \lim_{N\rightarrow \infty} \mathbb{P}[ \mathcal{M}_{\TT_N^d} \geq (1-\delta) N^d] = 1.
        \end{equation}
        Moreover, with probability $1-o_N(1)$, there exists a unique (random) time $ M_{\star} $ at which the maximal number of spins flip, and we have $|M_{\star}|\leq(\log N) N^{-d/2}$.\label{claim-3d-ground-state}
    \end{enumerate}
\end{theorem}

Theorem~\ref{thm:2d-ground-state} states that there is no global avalanche in two dimensions for any $ \epsilon>0 $. In three dimensions and higher, however, Theorem~\ref{thm:3d-ground-state} states that there is a phase transition for global avalanche as $ \epsilon $ changes: when $\epsilon$ is sufficiently large, there is no global avalanche; when $\epsilon$ is sufficiently small, a global avalanche occurs where $ \frac{1}{|V|} \mathcal{M}_{\mathbb{T}^d_N} $ converges to 1 in probability as $ N\to\infty $ and $ \epsilon\to0 $. We also show that the flipping time is around $N^{-d/2+o(1)}$. This power is sharp, and moreover, it is plausible that the sequence of random variables $(M_{\star} N^{d/2})_{N \geq 1}$ converges to a non-degenerate random variable; see Remark~\ref{rmk:tight-threshold}.

Now we review related works and give some high-level ideas of our proofs (with an emphasis on drawing connections to previous works). The distinction between $ d=2 $ and $ d\geq3 $ in Theorems~\ref{thm:2d-ground-state} and \ref{thm:3d-ground-state} is not surprising in light of the Imry-Ma prediction~\cite{ImryMa-75} for spin correlations in the RFIM. For $ d=2 $, this prediction was confirmed by~\cite{AizenmanWehr-89} where the authors showed that the spin correlation decays to 0 with the distance as long as disorder is present (namely $\epsilon>0$). This decay was later proved to be exponential with respect to the distance in~\cite{DingXia, AizenmanHarelPeled}, and the correlation length was characterized in~\cite{DingWirth, DingHuangXia}. Our proof of Theorem~\ref{thm:2d-ground-state} is essentially based on the exponential decay of spin correlations in the RFIM. The high-level idea is that due to the exponential decay of spin correlations, the flip time of a spin depends only on local environments, and thus it is unlikely for two neighboring spins to flip at the same time. Using a percolation argument, we prove Theorem~\ref{thm:2d-ground-state} in Section~\ref{sec:2d}.

For $ d\geq3 $, the behavior of spin correlations in the RFIM depends on both the temperature and the noise intensity. Let $T_c$ be the critical temperature for the Ising model. When there is no disorder ($\epsilon = 0$), the phase transition for the Ising model is sharp~\cite{ABF87} (which holds for all $ d\geq2 $): in the supercritical region ($T > T_c$) the spin correlation decays exponentially to 0 with the distance, while in the subcritical region ($0 \leq T < T_c$) the spin correlation is bounded from below by a positive constant (a phenomenon which we will refer to as long-range order). When $T>T_c$, by~\cite{DingSongSun}, the spin correlation in the RFIM is smaller than that of the Ising model, and thus, also decays exponentially to 0 with the distance. In the low temperature region, the work of~\cite{Imbrie85, BricmontKupiainen88} showed that the RFIM model has long-range order when $T$ and $\epsilon$ are both sufficiently small. A simpler proof was given in~\cite{DingZhuang} based on~\cite{FisherFrolichSpencer}. Recently, \cite{DingLiuXia2022} extended the result to the entire subcritical region: for any given $T< T_c$ and sufficiently small $\epsilon$, the model has long-range order. Claim~\eqref{claim-3d-ground-state} in Theorem~\ref{thm:3d-ground-state} is based on the approach in~\cite{DingZhuang, FisherFrolichSpencer} which controls the change of the ground state Hamiltonian after flipping the disorder. Roughly speaking, the long-range order is ensured by the concentration of the change of the ground-state Hamiltonian and the same effect also implies that most of the spins will flip at the same time. A new difficulty in our case is that we need to provide a uniform bound for the change of the ground-state Hamiltonian over $M$.
To address this, we first control the change of the Hamiltonian when flipping a large set (so that the tail bound allows to afford a union bound), and then we control the change of the Hamiltonian when flipping a small set whereas the Hamiltonian is with respect to a local box (and thus the effect from a slight change in $M$ is small). As for Claim~\eqref{claim-3d-ground-state-1} in Theorem~\ref{thm:3d-ground-state}, it follows from a percolation argument, since when the disorder is large enough, it is unlikely for neighboring spins to flip at the same time.

\subsection{Glauber evolution and polluted bootstrap percolation}
\label{subsec:glauber-evolution}
In this subsection, we first state the result on the Glauber evolution at zero temperature. This model was initiated in physics works \cite{initiate1_PhysRevLett.70.3347,initiate2_DahmenKKRSS94} which used the zero-temperature Glauber evolution to study the Barkhausen effect (e.g., a sudden and discontinuous change in the magnetization of a ferromagnetic material when altering the external field), and we are not aware of any mathematical work on this model. We first define the model. Consider the graph $\mathbb{Z}^d$ for $d \geq 2$. When $M = -\infty$, we set $\sigma_v = -1$ for all $v \in \mathbb{Z}^d$, and then we increase $M$ from $-\infty$ to $\infty$ and perform the following (zero-temperature) Glauber evolution. The rule is that as $M$ increases, we maintain the state until there exists at least one minus spin $\sigma_v$ at $v \in \mathbb{Z}^d$ such that flipping it from minus to plus will (weakly) decrease the Hamiltonian as in~\eqref{eq:def-hamiltonian}; at this moment we define our new state by first flipping these aforementioned spins and then repeating the operation of ``flipping a minus spin as long as it further (weakly) decreases the Hamiltonian" until no more minus spins can be flipped (so there is a chain effect on the flipping, which potentially might lead to a global avalanche). After a spin is flipped to plus, it will remain plus forever. This yields a (random) increasing process on $\{-1, +1\}^{\mathbb{Z}^d}$ indexed by $M \in (-\infty, \infty)$. We call this the (zero-temperature) Glauber evolution, since it is consistent with the evolution process described earlier by taking limits from the Glauber dynamics (see Lemma~\ref{lem:positive-glauber}). Theorem~\ref{thm:1.3} gives the time around which most of the spins flip as the noise intensity $\epsilon\to0$.

\begin{theorem}\label{thm:1.3}
    For dimensions $d \geq 2$, let $c_d = \frac{2 \sqrt{d}}{1 + \sqrt{d}}$.
    \begin{enumerate}[(i)]
        \item For any $d \geq 2$ and fixed $M < c_d$, we have $\mathbb{P}[\sigma_0 = -1 \mbox{ at $M$}] = 1 - o_\epsilon(1)$ as $\epsilon$ tends to 0.\label{thm1.3-claim1}
        \item For $d \in \{2, 3 \}$ and any fixed $M > c_d$, we have $\mathbb{P}[\sigma_0 = 1 \mbox{ at $M$}] = 1 - o_\epsilon(1)$ as $\epsilon$ tends to 0.\label{thm1.3-claim2}
    \end{enumerate}
\end{theorem}

As a consequence, in two or three dimensions, when $\epsilon$ is small, most of the spins will flip around time $\frac{2 \sqrt{d}}{1 + \sqrt{d}}$ (we emphasize that we do not know whether such flipping occurs simultaneously or not). In comparison with Claim~\eqref{claim-3d-ground-state} of Theorem~\ref{thm:3d-ground-state}, we see that the flipping times for the ground state evolution and the zero-temperature Glauber evolution are drastically different. In dimensions $d \geq 4$, we only prove~\eqref{thm1.3-claim1}, but it is plausible that the analog of~\eqref{thm1.3-claim2} holds for all $d \geq 4$.

Next, we discuss a closely related model: polluted bootstrap percolation introduced by Gravner and McDonald~\cite{PollutedBP_original}. In fact, Theorem~\ref{thm:1.3} is derived by connecting the mechanism of the Glauber evolution to this model, \response{and the threshold $ \frac{2\sqrt{d}}{1+\sqrt{d}} $ is determined by the asymptotic threshold of the polluted bootstrap percolation (see the end of this subsection for a more detailed discussion).} We first define the model for any dimension $d \geq 2$. Let $p, q \in [0,1]$ be two parameters with $p + q \leq 1$. In the initial configuration (at time $t = 0$), each vertex on $\mathbb{Z}^d$ is closed with probability $q$, open with probability $p$, and empty with probability $1 - p - q$. The configuration evolves via the following rule: all vertices that are closed or open always remain in the same state, and an empty vertex becomes open at time $t+1$ if and only if it has at least $d$ open neighbors at time $t$. One interpretation of this model is that closed vertices represent pollution or obstacles, while open vertices represent active agents who wish to grow collaboratively. There is a natural partial order on the configurations: a configuration $A$ is said to be larger than (or to dominate) another configuration $B$ if all open vertices in $B$ are also open in $A$ and all closed vertices in $A$ are also closed in $B$. Since the evolution is increasing, as $t$ tends to $\infty$, it converges to a final configuration in the sense that the state of any finite domain will not change after a sufficiently long time. Let $\mathbb{P}_{p, q, \infty}$ denote the law of the final configuration. The law $\mathbb{P}_{p, q, \infty}$ is increasing in $p$ and decreasing in $q$. A natural question is for which values of $p$ and $q$, there will be an infinite open cluster under the law $\mathbb{P}_{p, q, \infty}$. Theorem~\ref{thm:1.4} concerns the correct scaling of $p$ and $q$ to ensure that there is no infinite open cluster as they tend to 0.

\begin{theorem}\label{thm:1.4}
    For any $d \geq 2$, there exists a constant $C$ depending on $d$ such that for all sufficiently small $p$ and $q \geq C p^d$, we have
    $$
    \mathbb{P}_{p, q, \infty}[\mbox{there exists an infinite open cluster}] = 0.
    $$
    Moreover, we have $\lim_{ p \rightarrow 0} \sup_{q \geq Cp^d} \mathbb{P}_{p, q, \infty}[\mbox{$0$ is open}] = 0$.
\end{theorem}

\begin{remark}
    \label{rem:Ivailo-comment}
    After the posting of this work, Ivailo Hartarsky brought to our attention that he, Janko Gravner, Alexander Holroyd, David Sivakoff, and R\'eka Szab\'o are completing a paper that also proves Theorem~\ref{thm:1.4} using the result in~\cite{BBMS24}. 
\end{remark}

The counterpart of Theorem~\ref{thm:1.4} was proved in~\cite{PollutedBP_original, Gravner-Holroyd-3D} in dimensions two or three. We collect their results here.

\begin{theorem}[\cite{PollutedBP_original, Gravner-Holroyd-3D}]\label{thm:1.5}
    When $d = 2$, $p \to 0$, and $q = o(p^2)$, or when $d = 3$, $p \to 0$, and $q = o(p^3(\log p^{-1})^{-3})$, we have
    $$
    \mathbb{P}_{p, q, \infty}[\mbox{there exists an infinite open cluster}] = 1.
    $$
    Moreover, under the same conditions, we have $\lim_{ p \rightarrow 0}\mathbb{P}_{p, q, \infty}[\mbox{$0$ is open}] = 1$.
\end{theorem}

We see that the power $d$ in Theorem~\ref{thm:1.4} is optimal in two or three dimensions. The analog of Theorem~\ref{thm:1.5} is expected to hold for all $d \geq 4$, and thus, the power $d$ in Theorem~\ref{thm:1.4} should be optimal.
Our Theorem~\ref{thm:1.4} solves the first open problem mentioned in~\cite[~Section 11]{Gravner-Holroyd-3D}. The case of $d = 2$ for Theorem~\ref{thm:1.4} was proved in~\cite{PollutedBP_original} based on planarity. In~\cite{Gravner-Holroyd-3D}, the authors considered the case of $ d=3 $ and proved the result in Theorem~\ref{thm:1.4} only when $q \geq C p ^2$ for some large constant $C$. Moreover, the methods in~\cite{PollutedBP_original, Gravner-Holroyd-3D} seem difficult to extend to $d \geq 4$ (and the resulting bound would be far from optimal). The proof of Theorem~\ref{thm:1.4} will be given in Section~\ref{sec:polluted} based on a coarse-graining argument. We refer to that section for a detailed proof outline.

Bootstrap percolation without pollution (i.e., $q = 0$) has a long history and there is extensive research on it. We refer to the review~\cite{Morris-bootstrap} and the references therein for more discussions on this model. The behavior of bootstrap percolation depends on the rule of evolution. For instance, we can consider the evolution where an empty vertex becomes open if and only if it has at least $r$ open neighbors (where $r$ is often referred to as the threshold, and in the above case we have $r = d$). In polluted bootstrap percolation, $r = d$ is in some sense the critical case. For any fixed $q>0$, let $p_c(q)$ be the smallest $p$ such that, with positive probability, there is an infinite open cluster (not necessarily unique) in the final configuration. Different values of $r$ will impact the asymptotic behavior of $p_c(q)$ as $q \to 0$. The conjectured behavior is as follows and we will mention below what has been proved:\footnote{We learned from Ivailo Hartarsky that he, Janko Gravner, Alexander Holroyd, David Sivakoff, and R\'eka Szab\'o are writing a paper proving these conjectures.\label{foot:Ivailo-ongoing}}
\begin{enumerate}[(A)]
    \item For $1 \leq r \leq d-1$, $p_c(q)$ equals 0 for all sufficiently small $q$.\label{conj:pollute-1}
    \item For $r = d$, we have $p_c(q) = q^{1/d+o(1)}$ as $q \to 0$.\label{conj:pollute-2}
    \item For $r > d$, $p_c(q)$ is lower-bounded by a positive constant for all $q > 0$.\label{conj:pollute-3}
\end{enumerate}
Conjecture~\eqref{conj:pollute-1} was proposed in~\cite{Morris-bootstrap}, and the case of $r = 1$ follows from the fact that there is almost surely an infinite non-closed cluster when $q$ is sufficiently small. The case of $r = 2$ and $d \geq 3$ was proved in~\cite{GH-polluted-all} which is highly non-trivial. The other cases are still open. For Conjecture~\eqref{conj:pollute-2}, the cases of $d = 2,3$ follow from combining Theorems~\ref{thm:1.4} and \ref{thm:1.5}. The cases for $d \geq 4$ are still open and the missing part is the analog of Theorem~\ref{thm:1.5} for $d \geq 4$. It is also plausible that for $r = d$, $p_c(q)/q^{1/d}$ will converge to a constant as $q \to 0$. Claim~\eqref{conj:pollute-3} follows from the fact that in bootstrap percolation ($q = 0$) with threshold $r = d + 1$, there is almost surely no infinite open cluster when $p$ is sufficiently small. This result can be proved using the approach in~\cite[Theorem (ii)]{schonmann} (the result there was stated for $d=2$, but the same proof applies to any $d \geq 3$). 

There are other variants of bootstrap percolation that one can consider. A notable variant is the so-called modified bootstrap percolation: in this evolution, an empty vertex becomes open if and only if it has open neighbors in at least $r$ directions, where $1 \leq r \leq d$. 
Conjectures~\eqref{conj:pollute-1} and \eqref{conj:pollute-2} are also conjectured to hold for this model. Since modified bootstrap percolation is stochastically dominated by bootstrap percolation with the same threshold, Theorem~\ref{thm:1.4}, which is proved for the non-modified version, readily applies to the modified one. For $d = 3$, this result was first established in~\cite{Gravner-Holroyd-3D}, where the authors proved the analog of Theorem~\ref{thm:1.4} for modified polluted bootstrap percolation and polluted bootstrap percolation with large obstacles. Another remark is that Conjecture~\eqref{conj:pollute-1} for modified polluted bootstrap percolation (in particular, the case $r = d - 1$) would allow us to derive the analog of Theorem~\ref{thm:1.5} for all $d \geq 4$, thereby establishing Conjecture~\eqref{conj:pollute-2} for polluted bootstrap percolation for all $d \geq 4$ (see discussions below Problem~\ref{prob:threshold}).

To connect the mechanism of the zero-temperature Glauber evolution to polluted bootstrap percolation, we will call a vertex $v$ in the Glauber evolution open if $\epsilon h_v + M -2 \geq 0$, and call it closed if $\epsilon h_v + M < 0$. Let $p$ and $q$ be the probabilities of a vertex being open and closed, respectively. An open vertex needs only $ d-1 $ plus neighbors to flip from minus to plus which serves an active agent, while a closed vertex requires at least $ d+1 $ plus neighbors to flip which serves as an obstacle.
The threshold $M = c_d$ corresponds exactly to the case $q \approx p^d$ as $ \epsilon $ tends to 0\response{, in the sense that for any fixed $ M<c_d $ (resp., $ M>c_d $), we have $ q\gg p^d $ (resp., $ q\ll p^d $) for sufficiently small $ \epsilon $}. By comparing the Glauber evolution with polluted bootstrap percolation (they are very similar but still slightly different), we will show in Section~\ref{sec:glauber} that Claim~\eqref{thm1.3-claim1} in Theorem~\ref{thm:1.3} follows from proof ideas similar to Theorem~\ref{thm:1.4}, and that Claim~\eqref{thm1.3-claim2} in Theorem~\ref{thm:1.3} follows from proof ideas similar to Theorem~\ref{thm:1.5} (which is based on results from~\cite{GH-polluted-all}).

\response{
\begin{remark}[Short-range interactions]
    One may wonder what happens if we add short-range interactions to the model, say adding edges between points with bounded Euclidean lengths. Let us first consider the ground state evolution. Indeed, Theorem~\ref{thm:3d-ground-state} can be proved using the same argument. However, there are technical difficulties in proving Theorem~\ref{thm:2d-ground-state}, since we rely on the result from~\cite{DingXia, AizenmanHarelPeled} which in turn relies on planar duality and does not apply to short-range models although we believe the same result should hold. The Glauber evolution is even more sensitive to the lattice. In fact, bootstrap percolation is very sensitive to the lattice structure (see~\cite{Morris-bootstrap} for a discussion) and the conjectured behavior below Theorem~\ref{thm:1.5} is specific to the integer lattice and may be quite different if additional edges are added (for instance, the critical threshold may change). So, even if Theorem~\ref{thm:1.3} holds for other lattices, the threshold will generally be different.
\end{remark}
}

\subsection{Basic notations}
\label{subsec:notations}

For a set $A$, we use $|A|$ to denote its cardinality.
For a graph $G = (V,E)$ and a subset $A$ of $V$, let $\partial_i A := \{ x \in A: \exists y \in V \setminus A$ such that $x\sim y\}$, $\partial_o A := \partial_i (V \setminus A)$, and $\partial_e A := \{ (x,y): x \in \partial_i A, y \in \partial_o A, x\sim y\}$.
For a spin configuration $\sigma \in \{-1,1\}^V$ and a vertex $v \in V$, we use $\sigma_v$ or $\sigma(v)$ to denote the spin at vertex $v$. There is a natural partial order for two spin configurations $\sigma, \sigma^{\prime} \in \{-1,1\}^V$: we write $\sigma \geq \sigma^{\prime}$ if $\sigma(v) \geq \sigma^{\prime}(v)$ for all $v \in V$. In this paper, we will consider the graph $G$ to be torus $\TT_N^d$ or $\mathbb{Z}^d$ for $d \geq 2$. We can define the $|\cdot|_\infty$-norm on $G$ and let $B_r(u) =B(u,r) := \{v \in G : |v-u|_\infty \leq r\}$ be the $|\cdot|_\infty$-box centered at $u$ with radius $r$.

Constants like $c,c',C,C'$ may change from place to place, while constants with subscripts like $c_1,C_1$ remain fixed throughout the article. All constants may implicitly rely on the dimension $d$. The dependence on additional variables will be indicated at the first occurrence of each constant.

\medskip

\noindent\textbf{Organization of the paper.} 
In Section~\ref{sec:2d}, we study the ground state evolution for $ d=2 $ and prove Theorem~\ref{thm:2d-ground-state}. In Section~\ref{sec:3d}, we study the ground state evolution for $d \geq 3$ and prove Theorem~\ref{thm:3d-ground-state}. In Section~\ref{sec:polluted}, we study polluted bootstrap percolation and prove Theorem~\ref{thm:1.4}. Finally, in Section~\ref{sec:glauber} we study the Glauber evolution and prove Theorem~\ref{thm:1.3}.

\medskip
\noindent\textbf{Acknowledgements.} 
We thank Rongfeng Sun for suggesting the problem to us and thank Ron Peled for helpful discussions. We are grateful to Ivailo Hartarsky for bringing~\cite{schonmann} to our attention and for informing us of his ongoing work with Janko Gravner, Alexander Holroyd, David Sivakoff, and R\'eka Szab\'o (see Remark~\ref{rem:Ivailo-comment} and Footnote~\ref{foot:Ivailo-ongoing}).
We warmly thank two anonymous reviewers for their careful and insightful comments on an earlier version of the manuscript.
J.D. is partially supported by NSFC Tianyuan Key Program Project No.12226001, NSFC Key Program Project No.12231002 and an Explorer Prize. Z.Z.\ is partially supported by NSF grant DMS-1953848.


\section{{Ground state evolution in two dimensions}}
\label{sec:2d}
	
In this section, we consider the ground state evolution of the RFIM on $\TT_N^2$ with the mean of the external field $ M $ increasing from $ -\infty $ to $ \infty$ and prove Theorem~\ref{thm:2d-ground-state}. Recall from Section~\ref{subsec:gs-evolution} the definition of $\tau_{\TT_N^2, M ,\epsilon h}$.
In this section, we may consider the RFIM on subdomains of $\TT_N^2$ whereas the disorder field will always be induced by $\{h_v:v\in\TT^2_N\}$. For simplicity, we will omit $\epsilon h$ in the subscripts. For an integer $N \geq 1$ and $M \in \mathbb{R}$, let the flipping set at time $M$ be   
$$ {\rm Flip}(M) := \{ v \in \TT_N^2: \tau_{\TT_N^2, M}(v) \neq \tau_{\TT_N^2, M^-}(v) \} \quad \mbox{with}\quad  \tau_{\TT_N^2, M^-}(v) := \lim_{t \nearrow M}\tau_{\TT_N^2,t}(v).$$
Then we have $\mathcal{M}_{\TT_N^2} = \max_{M \in \mathbb{R}} |{\rm Flip}(M)|$. For $\epsilon>0$, let $$M_1(\epsilon) = 4 + 100 \epsilon.$$ The proof of Theorem~\ref{thm:2d-ground-state} will be divided into two cases. In the first case $|M| \geq M_1$, we will compare the plus or minus spin clusters in $\tau_{\TT_N^2,M}$ with a subcritical Bernoulli percolation (Proposition~\ref{prop:2d-large-m}). In the second case $|M| < M_1$, we will use the exponential decay of spin correlations in the two-dimensional RFIM~\cite{DingXia, AizenmanHarelPeled}; see Proposition~\ref{prop:2d-small-m}. We first deal with the case of $|M| \geq M_1$.

Before the proposition, we state the following lemma which occurs many times in the following sections.
\begin{lemma}
    \label{lem:count_surrounding_contour}
    Consider the torus $ \TT^d_N $ and let $ k\geq 1 $ be an integer. Then the number of connected subsets in $ \TT^d_N $ with size $ k $ is at most \response{$ |\TT^d_N|\times(2d)^{2k-2} $.}
\end{lemma}
\begin{proof}
    Let $ A\subset\TT^d_N $ be a connected subset with size $ k $. Then we can arbitrarily choose a spanning tree of $ A $ with an arbitrarily chosen root. Consider a depth-first search of this spanning tree from the root, we know that this search process yields a nearest-neighbor path in $ A $ that \response{contains $(2k-2)$ edges} and traverses each edge of the aforementioned spanning tree twice. Therefore, we can map $ A $ injectively to this path, whose enumeration is at most \response{$ |\TT^d_N|\times(2d)^{2k-2} $}.
\end{proof}

\response{

We will use the connectivity of ${\rm Flip}(M)$ in the following proof.

\begin{lemma}\label{lem:connect-flip}
    Almost surely, the flipping set ${\rm Flip}(M)$ is connected for all $M \in \mathbb{R}$.
\end{lemma}

\begin{proof}
    The underlying intuition of the lemma is the following: if ${\rm Flip}(M)$ had multiple connected components, it would imply linear equations of independent Gaussian variables, which can only hold with probability 0. We now carry out the proof carefully. Let us suppose ${\rm Flip}(M)$ has (without loss of generality) two connected components $ A_1 $ and $A_2$, then $\tau_{\TT^d_N, M}(v) = - \tau_{\TT^d_N, M^-}(v)\cdot\mathbf{1}_{v\in A_1\cup A_2} + \tau_{\TT^d_N, M^-}(v)\cdot\mathbf{1}_{v\notin A_1\cup A_2}$.
    Let
    \begin{equation*}
        \tau^{A_i}(v):= - \tau_{\TT^d_N, M^-}(v)\cdot\mathbf{1}_{v\in A_i} + \tau_{\TT^d_N, M^-}(v)\cdot\mathbf{1}_{v\notin A_i}
    \end{equation*}
    for $i\in\{1,2\}$.
    Since our Hamiltonian is defined on a finite graph, there exists $M_1<M$ such that $\tau_{\TT^d_N, M^{\prime}}(v)=\tau_{\TT^d_N, M^-}(v)$ for all $ M^{\prime}\in[M_1, M) $ and all $v\in\TT^d_N$. Then by definition we have
    \begin{equation*}
        \begin{aligned}
            &H_{\TT^d_N, M, \epsilon h}(\tau_{\TT^d_N, M}) \leq H_{\TT^d_N, M, \epsilon h}(\tau^{A_i}), \quad H_{\TT^d_N, M, \epsilon h}(\tau_{\TT^d_N, M})\leq H_{\TT^d_N, M, \epsilon h}(\tau_{\TT^d_N, M^-}),\\
            &H_{\TT^d_N, M^{\prime}, \epsilon h}(\tau_{\TT^d_N, M^-}) \leq H_{\TT^d_N, M^{\prime}, \epsilon h}(\tau^{A_i}),\quad H_{\TT^d_N, M^{\prime}, \epsilon h}(\tau_{\TT^d_N, M^-}) \leq H_{\TT^d_N, M^{\prime}, \epsilon h}(\tau_{\TT^d_N, M})
        \end{aligned}
    \end{equation*}
    for $i\in\{1,2\}$.
    Combining the preceding displays yields
    \begin{equation*}
        \sum_{u\sim v,u\in A_i,v\notin A_i}\tau_{\TT^d_N, M}(u)\tau_{\TT^d_N, M}(v) + \sum_{u\in A_i}(M+\epsilon h_u)\cdot\tau_{\TT^d_N, M}(u)=0
    \end{equation*}
    for $i\in\{1,2\} $, which further implies that
    \begin{equation}
        \label{eq:flipping_set_connected}
        \begin{aligned}
            &\frac{1}{\sum_{u\in A_1}\tau_{\TT^d_N, M}(u)}\cdot\Big(\sum_{u\sim v,u\in A_1,v\notin A_1}\tau_{\TT^d_N, M}(u)\tau_{\TT^d_N, M}(v) + \sum_{u\in A_1}\epsilon h_u\cdot\tau_{\TT^d_N, M}(u)\Big)\\
            =&\frac{1}{\sum_{u\in A_2}\tau_{\TT^d_N, M}(u)}\cdot\Big(\sum_{u\sim v,u\in A_2,v\notin A_2}\tau_{\TT^d_N, M}(u)\tau_{\TT^d_N, M}(v) + \sum_{u\in A_2}\epsilon h_u\cdot\tau_{\TT^d_N, M}(u)\Big).
        \end{aligned}
    \end{equation}
    (Note that by monotonicity, $\tau_{\TT^d_N, M}(u) = 1$ for $u \in A_i$ and thus $\sum_{u\in A_i}\tau_{\TT^d_N, M}(u) > 0$). Since the choices for $ A_1 $, $ A_2 $, and $ \{\tau_{\TT^d_N, M}(v)\}_{v\in\TT^d_N} $ are finite, the probability that \eqref{eq:flipping_set_connected} holds is zero. Therefore, $ {\rm Flip}(M) $ is almost surely connected simultaneously for all $M$. \qedhere
    
\end{proof}
}

\begin{proposition}
    \label{prop:2d-large-m}
    For $d = 2$ and any $\epsilon > 0$, there exists a constant $C = C(\epsilon)>0$ such that
    $$
    \mathbb{P}\big[|\mathrm{Flip}(M)| \leq C \log N \mbox{ for all }|M| > M_1\big] \geq 1 - N^{-200} \quad \forall N \geq 1.
    $$
\end{proposition}

\begin{proof}
    Let us first consider the case of $M > M_1$. \response{Note that any vertex $u \in \TT_N^2$ has at most $4$ minus neighbors. Therefore by flipping the spin at $u$ and the definition of the ground state, we know that if $-4 + M_1 + \epsilon h_u \geq 0$, then $\tau_{\TT_N^2,M}(u)=1$ for all $M > M_1$.} So, $\cup_{M^{\prime} > M_1} {\rm Flip}(M^{\prime})$ is stochastically dominated by a Bernoulli site percolation on $\TT_N^2$ with open probability
    \begin{equation*}
        \P\left[ -4 + M_1 + \epsilon h_u \leq 0 \right] = \P\left[ h_u\leq -100 \right] \leq 100^{-1}.
    \end{equation*}
    \response{By Lemma~\ref{lem:connect-flip}}, almost surely, ${\rm Flip}(M)$ is connected for all $M \in \mathbb{R}$. Therefore, for all $M > M_1$, Flip$(M)$ must be contained in a connected subset of $\cup_{M^{\prime} > M_1} {\rm Flip}(M^{\prime})$. This together with an elementary counting argument implies the lemma, as we now elaborate. Fix a constant $A$ to be chosen.
    For each fixed cluster with size $k$, the probability that it is open in a Bernoulli site percolation with open probability $100^{-1}$ is $100^{-k}$. Combining these arguments with Lemma~\ref{lem:count_surrounding_contour}, we obtain
    $$
    \mathbb{P}\big[ \max_{M > M_1} \left| \mathrm{Flip}(M) \right| \geq A \log N \big] \leq \sum_{k \geq A \log N}|\TT_N^2| \times 4^{2k - 2} \times 100^{-k}. 
    $$ 
    Taking a large constant $A$ yields the lemma in the case of $M > M_1$. The case of $M < -M_1$ can be treated similarly by observing that if $4 - M_1 + \epsilon h_u \leq 0$, then $\tau_{\TT_N^2,M}(u)=-1$ for all $M < -M_1$. \qedhere
\end{proof}

Now we deal with the case of $-M_1 \leq M \leq M_1 $. We begin with a lemma which states that with probability close to 1, the ground state configuration will not change in a fixed small box in a small time interval of length $N^{-1}$. For an integer $L \geq 1$, $u \in \TT_N^2$, and any spin configuration $\xi \in \{-1,1\}^{\partial_o B_L(u)}$, let $\tau_{B_L(u), M}^\xi$ be the ground state with respect to the RFIM Hamiltonian on $B_L(u)$ with boundary condition $\xi$ and disorder $\{h_v\}$. More precisely, the Hamiltonian is given by
    $$
    H_{B_L(u), M}^\xi(\sigma) := - \sum_{\substack{v \sim w\\ v,w \in B_L(u)}} \sigma_v \sigma_w - \sum_{v \in B_L(u)} (M + \epsilon h_v) \sigma_v - \sum_{\substack{v \sim w\\ v \in \partial_i B_L(u), w \in \partial_o B_L(u)}} \sigma_v \xi_w\,,
    $$
where $\sigma \in \{-1,1\}^{B_L(u)}$. Note that $\{h_v\}$ is the same as in~\eqref{eq:def-hamiltonian}.

\begin{lemma}
\label{lem:2d-m-large}
    Let $d = 2$ and fix any $\epsilon, \rho > 0$. There exists an integer $L>0$ that depends on $\rho$ and $\epsilon$ such that for all sufficiently large $N$, any vertex $u \in \TT_N^2$, and $|t| \leq M_1$:
    \begin{equation}
        \label{eq:lem-2d-m-large}
        \mathbb{P}\Big[ \tau_{B_{2L}(u), M}^{\xi} = \tau_{B_{2L}(u), M'}^{\xi'} \mbox{ in }B_L(u) \quad \forall \xi,\xi' \in \{-1,1\}^{\partial_o B_{2L}(u)}, t \leq M,M' \leq t + N^{-1} \Big] \geq 1- \rho.
    \end{equation}
\end{lemma}
\begin{proof}
    When $\xi$ is all plus (resp. minus), we write $\tau_{B_L(u), M}^\xi$ as $\tau_{B_L(u), M}^+$ (resp. $\tau_{B_L(u), M}^-$). By monotonicity, for all $\xi \in \{-1,1\}^{\partial_o B_{2L}(u)}$ and $t \leq M \leq t + N^{-1}$, we have
    $$
    \tau_{B_{2L}(u), t}^- \leq \tau_{B_{2L}(u), M}^\xi \leq \tau_{B_{2L}(u), t + N^{-1}}^+.
    $$
    Therefore, it suffices to show that $\mathbb{P}[\tau_{B_{2L}(u), t}^-(v) = \tau_{B_{2L}(u), t + N^{-1}}^+(v) \quad \forall v \in B_L(u)] \geq 1 - \rho$ which will follow from the following two claims:
    \begin{equation}
        \label{eq:lem2.3-1}
        \mathbb{P}\Big[ \tau_{B_{2L}(u),t}^-(v) =  \tau_{B_{2L}(u),t}^+(v) \quad \forall v \in B_L(u) \Big] \geq 1 - \frac{\rho}{2};
    \end{equation}
    \begin{equation}
        \label{eq:lem2.3-2}
        \mathbb{P}\Big[ \tau_{B_{2L}(u),t}^+(v) =  \tau_{B_{2L}(u), t + N^{-1}}^+(v)\quad \forall v \in B_L(u) \Big] \geq 1 - \frac{\rho}{2}.
    \end{equation}
    Next, we prove~\eqref{eq:lem2.3-1} and \eqref{eq:lem2.3-2} separately. Claim~\eqref{eq:lem2.3-1} follows from the exponential decay of spin correlations in the RFIM~\cite{DingXia, AizenmanHarelPeled}. Using \cite[Theorem 1.1]{AizenmanHarelPeled}, there exists a constant $C>0$ that depends only on $\epsilon$ such that
    $$
    \mathbb{P}\Big[ \tau_{B_{2L}(u),t}^-(v) =  \tau_{B_{2L}(u),t}^+(v)\Big] \geq 1 - C e^{-L/C} \quad \mbox{for all } v \in B_L(u), t \in \mathbb{R}\,.
    $$
    Taking a union bound over $v \in B_L(u)$ and then selecting a large constant $L$ yields~\eqref{eq:lem2.3-1}.

    Now we prove~\eqref{eq:lem2.3-2} for any fixed $L$ and sufficiently large $N$. The idea is that each spin can flip at most once, however, the spin has almost equal probability to flip at many $N^{-1}$ time intervals near $t$. Therefore, the probability of any spin in $B_L(u)$ flipping in the time interval $[t, t + N^{-1}]$ will be arbitrarily close to 0 when $N$ is sufficiently large. 
    
    We next implement the proof idea above. Fix $L \geq 1$ and let $K =  \max \{100L^2/\rho, 100\}$. Then there exists an integer $N_1$ that depends on $L, \rho$ such that for all $N \geq N_1$, with probability at least $1 - K^{-1}$, \response{the Radon-Nikodym derivative of $\{\epsilon h_v\}_{v \in B_{2L}(u)}$ and $\{\epsilon h_v + s\}_{v \in B_{2L}(u)}$ is between $[1/2,2]$ for all $- KN^{-1} \leq s \leq KN^{-1}$.} Denote this event by $\mathcal{E}$. Fix $-M_1 \leq t \leq M_1$. Then for all $v \in B_L(u)$ and $t - K N^{-1} \leq s \leq t + K N^{-1}$, \response{by regarding $ \{t+\epsilon h_v\}_{v\in B_{2L}(u)} $ as $ \{s + \epsilon h_v\}_{v\in B_{2L}(u)} $,} we have:
    $$
    \mathbb{P}\Big[ \{ \tau_{B_{2L}(u),t}^+(v) \neq  \tau_{B_{2L}(u), t + N^{-1}}^+(v)\} \cap \mathcal{E} \Big] \leq 2 \mathbb{P}\Big[ \{  \tau_{B_{2L}(u), s}^+(v) \neq  \tau_{B_{2L}(u), s + N^{-1}}^+(v) \} \cap \mathcal{E} \Big].
    $$
    Since each spin can flip at most once, we have
    $$
    \sum_{s \in t + N^{-1}(\mathbb{Z} \cap [-K,K])} \mathbb{P}\Big[ \{  \tau_{B_{2L}(u), s}^+(v) \neq  \tau_{B_{2L}(u), s + N^{-1}}^+(v) \} \cap \mathcal{E} \Big] \leq 1\,.
    $$
    Therefore, 
    \begin{equation*}
        \begin{aligned}
            & \quad\ \mathbb{P}\Big[  \tau_{B_{2L}(u),t}^+(v) \neq  \tau_{B_{2L}(u), t + N^{-1}}^+(v)\ \Big] \\&\leq \mathbb{P}\Big[ \{ \tau_{B_{2L}(u),t}^+(v) \neq  \tau_{B_{2L}(u), t + N^{-1}}^+(v)\} \cap \mathcal{E} \Big] + \mathbb{P}[\mathcal{E}^{c}] \\& \leq \frac{2}{K} + \frac{1}{K}.
        \end{aligned}
    \end{equation*}
    Taking a union bound over $v \in B_L(u)$ yields~\eqref{eq:lem2.3-2}. This concludes the desired lemma.
\end{proof}

We are now ready to control $ {\rm Flip}(M) $ for $ |M|\leq M_1 $, as incorporated in the next proposition. 

\begin{proposition}
    \label{prop:2d-small-m}
    For $d = 2$ and any $\epsilon > 0$, there exists a constant $C = C(\epsilon)>0$ such that
    $$
    \mathbb{P}\left[ \max_{\left|M\right| \leq M_1}\left| \mathrm{Flip}(M) \right| \leq C \log N \right] \geq 1 - N^{-200} \quad \forall N \geq 1.
    $$
\end{proposition}

\begin{proof}
    Fix $\epsilon>0$. It suffices to show that there exists a constant $C' > 0$ that depends on $\epsilon$ such that for all sufficiently large $N$
    \begin{equation}
        \label{eq:prop-2d-small-1}
        \mathbb{P}\Big[ \max_{t \leq M \leq t + N^{-1}} |{\rm Flip}(M)| \geq C' \log N \Big] \leq N^{-300} \quad \forall -M_1 \leq t \leq M_1\,.
    \end{equation}
    Indeed, provided with~\eqref{eq:prop-2d-small-1}, we can take a union bound over $t \in (N^{-1} \mathbb{Z} \cap [-M_1, M_1]) \cup \{-M_1\}$, and derive Proposition~\ref{prop:2d-small-m} for large $N$. We can then enlarge the constant $C$ in Proposition~\ref{prop:2d-small-m} such that it holds for all small $N$.
    
    We next prove~\eqref{eq:prop-2d-small-1}. Let $\rho>0$ be a small constant to be chosen and let $L$ be the integer from Lemma~\ref{lem:2d-m-large} with the same $\epsilon$ and $\rho$. Fix $t \in [-M_1, M_1]$. Consider the set $\mathcal{L} := L\mathbb{Z}^2 \cap \TT_N^2$, where $L\mathbb{Z}^2 \cap \TT_N^2$ denotes $0 + (L\mathbb{Z} \cap [0,N-1])^2$, and here $0$ is a fixed vertex on $\TT_N^2$.
    The $B_L$-boxes centered at the vertices in $\mathcal{L}$ can cover $\TT_N^2$, and each vertex in $\mathcal{L}$ has at most 8 neighbors, where we consider two vertices in $\mathcal{L}$ connected if their $|\cdot|_\infty$-distance on $\TT_N^2$ is at most $L$.
    For $u \in \mathcal{L}$, we call the box $B_L(u)$ \textbf{$L$-good} if it satisfies the condition in~\eqref{eq:lem-2d-m-large} with the same $L, u, t$. Therefore, every spin in an $L$-good box $B_L(u)$ will not flip in the ground state evolution during the time interval $[t, t + N^{-1}]$, and thus, $\cup_{t \leq M \leq t+N^{-1}} {\rm Flip}(M)$ cannot intersect any $L$-good box. \response{See Figure~\ref{fig:2.4} for an illustration of the preceding argument.}

    \begin{figure}[h]
        \centering
        \includegraphics[width=0.5\textwidth]{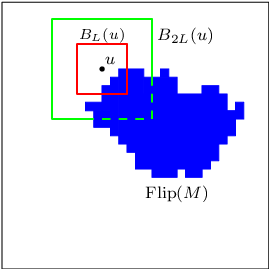}
        \caption{\response{An illustration that $ {\rm Flip}(M) $ does not intersect any $L$-good box. The subset $ {\rm Flip}(M) $ is colored in blue, the boundary of $ B_{2L}(u) $ is colored in green and dashed green, and the boundary of $B_{L}(u)$ is colored in red. The change of boundary condition of $ B_{2L}(u) $ (the dashed green) affects the configuration in $ B_L(u) $, and thus $B_L(u)$ is not $L$-good.}}
        \label{fig:2.4}
    \end{figure}

    By Lemma~\ref{lem:2d-m-large}, the probability for a $B_L(u)$ box being $L$-good is at least $1- \rho $. Moreover, this event only depends on the disorder $\{h_v\}$ restricted to $B_{2L}(u)$. Using \cite[Theorem 0.0]{Liggett1997DominationBP}, we can choose a sufficiently small constant $\rho > 0$ such that the measure of being $L$-good boxes on $\mathcal{L}$ stochastically dominates a Bernoulli site percolation with open probability $1 - 100^{-1}$. Next, we will show that \eqref{eq:prop-2d-small-1} follows from this fact and a counting argument. 
    
    We claim that if $|{\rm Flip}(M)| \geq A \log N$ for some $M \in [t, t + n^{-1}]$ and $A>0$, then there exists a connected subset of $ \mathcal{L} $ with size at least $ A\log N/(2L+1)^2 $ such that none of the $B_L$-boxes centered at the vertices in this subset is $L$-good. Suppose that $|{\rm Flip}(M)| \geq A \log N$ for some $M \in [t, t + N^{-1}]$ and a constant $A>0$. Consider all the vertices in $\mathcal{L}$ such that the $B_L$-boxes centered at these vertices intersect ${\rm Flip}(M)$. Then, none of these $B_L$-boxes is $L$-good. Since ${\rm Flip}(M)$ is connected (Lemma~\ref{lem:connect-flip}), we see that these vertices form a connected subset of $\mathcal{L}$ (recall that two vertices in $\mathcal{L}$ are connected if their $|\cdot|_\infty$-distance on $\TT_N^2$ is at most $L$). Furthermore, the size of this subset is at least $A \log N / (2L+1)^2$ because the union of the $B_L$-boxes centered at these vertices must contain ${\rm Flip}(M)$. This proves the claim. Using a similar counting argument as in the proof of Proposition~\ref{prop:2d-large-m}, we can show that the probability of this event is smaller than $N^{-300}$ for a sufficiently large constant $A$. This concludes \eqref{eq:prop-2d-small-1}. \qedhere

\end{proof}

\begin{proof}[Proof of Theorem~\ref{thm:2d-ground-state}]
Combine Propositions~\ref{prop:2d-large-m} and \ref{prop:2d-small-m}.
\end{proof}


\section{Ground state evolution in {higher dimensions}}
\label{sec:3d}

In this section, we study the ground state evolution on $ \mathbb{T}^d_N $ for dimensions $d \geq 3$. In contrast to the two-dimensional case, there is a phase transition as the disorder strength $ \epsilon $ increases, as incorporated in Theorem \ref{thm:3d-ground-state}: on the one hand, for large $ \epsilon $ the maximal size of the flipping set is at most logarithmic with high probability; on the other hand, for small $ \epsilon $ the maximal size of the flipping set is proportional to the size of the torus with high probability. Thus, a global avalanche occurs for small disorder strength. In Section \ref{subsec:3d-large}, we will prove the case of strong disorder for Theorem \ref{thm:3d-ground-state} (Claim~\eqref{claim-3d-ground-state-1}), and in Sections \ref{subsec:3d-proof-strategy}--\ref{subsec:proof-avalanche} we will prove the case of weak disorder for Theorem \ref{thm:3d-ground-state} (Claim~\eqref{claim-3d-ground-state}).

Throughout Section~\ref{sec:3d}, we will use with probability close to 1 to express with probability at least $ 1-\delta $ for some prefixed $ \delta $, and use with high probability to express with probability at least $ 1-o(1) $ as the torus size $ N $ tends to infinity.

\subsection{{No avalanche with strong disorder}}
\label{subsec:3d-large}

The proof follows from a comparison to a Bernoulli site percolation similar to that of Proposition~\ref{prop:2d-large-m}. \response{Recall from the beginning of Section~\ref{subsec:gs-evolution} that $\tau_{\TT_N^d, M, \epsilon h}$ is the ground state with external field $M$ and disorder $\{ h_v \}$, i.e., the minimizer of the Hamiltonian~\eqref{eq:def-hamiltonian} defined on $\TT_N^d$.}

\begin{proof}[Proof of \eqref{eq:thm-3d-ground-state}]
    Observe that for any $M \in \mathbb{R}$, the spin at a vertex $v$ must have flipped to plus if $-2d + M + \epsilon h_v > 0$ and must stay minus if $2d + M + \epsilon h_v < 0$. Here $2d$ is the number of neighbors of a vertex in $\TT_N^d$. So, the spin at each vertex $v$ flips in the time interval $[-2d-\epsilon h_v, 2d-\epsilon h_v]$. In particular two neighboring vertices $u, v$ cannot flip at the same time if $|\epsilon h_u - \epsilon h_v| > 4d$.
    
    We call a vertex $v \in \TT_N^d$ \textbf{bad} if $|\epsilon h_u - \epsilon h_v| > 4d$ for all its neighboring vertices $u$. It is easy to see that the probability of a vertex being bad is close to 1 when $\epsilon$ is large. Since the event of a vertex $v$ being bad only depends on {$ \{ h_u:|u-v|_1\leq 1 \} $}, using~\cite[Theorem 0.0]{Liggett1997DominationBP} we can select a sufficiently large constant $C$ such that for all $\epsilon \geq C$, the measure of bad vertices stochastically dominates a Bernoulli site percolation with open probability at least $1 - (100d)^{-1}$. By the property of bad vertices and the connectivity of flipping sets, a flipping set with more than one vertex cannot contain any bad vertex.
    Moreover, the number of connected subsets on $\TT_N^d$ with size $k$ is at most $|\TT_N^d| \times (2d)^{2k-2}$ by Lemma~\ref{lem:count_surrounding_contour}.
    Therefore, from the aforementioned stochastic domination by Bernoulli site percolation, we have for all $\epsilon \geq C$
    $$
    \mathbb{P}\big[ \mathcal{M}_{\TT_N^d} \geq A \log N \big] \leq \sum_{k \geq A \log N} |\TT_N^d| \times (2d)^{2k-2} \times (100 d)^{-k}
    $$
    for any $A>0$. Taking $A = 100d$, we see that the right-hand side above can be upper-bounded by $N^{-100}$, thereby completing the proof.
\end{proof}

\subsection{{Proof strategy for avalanche with weak disorder}}
\label{subsec:3d-proof-strategy}

We now introduce some notation and sketch the proof strategy for avalanche with weak disorder as in \eqref{eq:thm-3d-ground-state-2} of Theorem \ref{thm:3d-ground-state}. We refer to Section~\ref{subsec:proof-avalanche} for details. We assume that $N \geq 100$ is sufficiently large which may depend on the disorder strength $\epsilon$. A set $A \subset \TT_N^d$ is called simply connected if both $A$ and $\TT_N^d \setminus A$ are connected. Let \begin{equation}
\mathcal{A}_N:= \Big{\{} \mbox{all simply connected sets in } \TT_N^d \mbox{ with cardinality at most } {\frac{1}{2}}|\TT_N^d|\Big{\}}.\label{eq:def-A-N}\end{equation}
We will frequently use the following isoperimetric inequality taken from~\cite{Iso-torus}: there exists a constant $C_2>0$ such that for all $A \in \mathcal{A}_N$, 
\begin{equation}
    \label{eq:isoperimetric}
    |\partial_e A| \geq C_2 |A|^{\frac{d-1}{d}}.
\end{equation}
For a set $A \in \mathcal{A}_N$, we will call its edge boundary $\partial_e A$ a (spin) interface for a spin configuration $\sigma \in \{-1,1\}^{\TT_N^d}$ if $\sigma|_{\partial_i A} = -1$ and $\sigma|_{\partial_o A} = 1$, or $\sigma|_{\partial_i A} = 1$ and $\sigma|_{\partial_o A} = -1$. A vertex $v$ is said to be enclosed by an interface if $v$ belongs to such {an} $A$. In this section, we will write $\tau_M$ and $H_M$ short for $\tau_{\TT_N^d, M, \epsilon h}$ and $H_{\TT_N^d, M, \epsilon h}$, respectively.

To prove that a global avalanche occurs in the time interval $[- (\log N) N^{-d/2}, (\log N) N^{-d/2}]$, it suffices to show that most of the vertices are not enclosed by any spin interface in this time interval. By~\cite{DingZhuang}, a sufficient condition for the vertex 0 not to be enclosed by any spin interface is (see Lemma~\ref{lem:ground-state-flip} and see Step 1 in the proof of Theorem~\ref{thm:3d-ground-state} Claim~\eqref{claim-3d-ground-state} for more explanation on this):
\begin{equation}
\label{eq:strategy}
\sup_{0 \in A, A \in \mathcal{A}_N} \frac{|\mathcal{H}_M^A - \mathcal{H}_M|}{|\partial_e A|} \leq \frac{3}{2} \quad \mbox{for all }|M| \leq (\log N) N^{-d/2},
\end{equation}
where $\tau_M^A$ and $\mathcal{H}_M^A$ are the ground state and the corresponding Hamiltonian with respect to the disorder field $h^A$ obtained from flipping $\{h_v\}$ on the set $A$ (see \eqref{eq:def-flip-disorder} and~\eqref{eq:def-hA} below). 

As shown in~\cite{DingZhuang}, it follows from~\cite{FisherFrolichSpencer} that when $\epsilon$ is small, \eqref{eq:strategy} holds for a fixed $M$ with probability close to 1. However, it cannot be straightforwardly proved that such an inequality holds simultaneously for all $|M| \leq (\log N) N^{-d/2}$ since it seems to be difficult to apply a union bound over $ M $. A natural attempt is to decompose the interval $ \left[ -(\log N) N^{-d/2}, (\log N)N^{-d/2} \right] $ into smaller intervals and aim to prove that $ \frac{\mathcal{H}^A_M-\mathcal{H}_M}{|\partial_e  A|} $ is essentially the same as $ M $ varies within each small interval. Indeed, this can be implemented when $|\partial_e A| \geq (\log N)^{C_1}$ (where $C_1$ is a large constant defined in Lemma~\ref{lem:control-change} below): we can divide $[-(\log N) N^{-d/2}, (\log N) N^{-d/2}]$ into time intervals of length $N^{-d}$, and prove \eqref{eq:strategy} by applying tail estimates in \cite{FisherFrolichSpencer} with a union bound; see Lemma~\ref{lem:A-large} below. However, when $ |\partial_e A| $ is small, the tail bound from \cite{FisherFrolichSpencer} is only moderate, and thus we cannot afford the cost of a union bound. Instead of proving \eqref{eq:strategy} simultaneously for all $M$ (which is possible but would require considerable effort), we will use another strategy.

As mentioned above, we can prove that when $\epsilon$ is sufficiently small, with high probability,
\begin{equation}
\label{eq:strategy-2}
\sup_{A \in \mathcal{A}_N, |\partial_e A| \geq (\log N)^{C_1}} \frac{|\mathcal{H}_M^A - \mathcal{H}_M|}{|\partial_e A|} \leq \frac{3}{2} \quad \mbox{for all } |M| \leq (\log N) N^{-d/2}.
\end{equation}
From this, we can deduce that there exists a strongly percolated spin cluster in $\tau_M$ which we call a global spin cluster (see Definition~\ref{def:global-spin-config}). However, it is not necessarily the case that this spin cluster contains most of the vertices. To overcome this, we introduce the notion of good vertices (see~\eqref{eq:strategy-3} below) which will imply that a good vertex belongs to the global spin cluster in $\tau_M$ for all $|M| \leq (\log N) N^{-d/2}$. Moreover, we will show that with high probability, most of the vertices are good. 

To be precise, we call a vertex $u \in \TT_N^d$ \textbf{good} if for all $A \in \mathcal{A}_N$ with $u \in A$ and $A \subset B(u, 2 (\log N)^{C_1})$,
\begin{equation}\label{eq:strategy-3}
|\mathcal{H}_{B(u,2 (\log N)^{C_1}), M, \epsilon h}^\pm - \mathcal{H}_{B(u,2 (\log N)^{C_1}), M, \epsilon h^A}^{\pm} | \leq |\partial_e A| \quad \mbox{for all $|M| \leq (\log N) N^{-d/2}$},
\end{equation}
where $\mathcal{H}_{B(u,2 (\log N)^{C_1}), M, \epsilon h}^\pm$ (resp.\ $\mathcal{H}_{B(u,2 (\log N)^{C_1}), M, \epsilon h^A}^\pm$) denotes the ground state Hamiltonian of the RFIM on $B(u, 2(\log N)^{C_1})$ with plus or minus boundary condition and disorder field $h$ (resp.\ $h^A$). Once again, we can derive from~\cite{FisherFrolichSpencer} that when $\epsilon$ is small, a fixed vertex is good with probability close to 1. In this case, we can apply a union bound over $M$ as the Hamiltonian in the small box $B(u, 2 (\log N)^{C_1})$ is essentially the same in the time interval $[-(\log N) N^{-d/2}, (\log N) N^{-d/2}]$. Using the first and second moment estimates and the independence of good vertices, we will show in Lemma~\ref{lem:good-vertex} that with high probability, most of the vertices are good. Furthermore, we will use monotonicity to show that if $\tau_M$ has a global spin cluster and $u$ is a good vertex, then $u$ belongs to the global spin cluster (see Step 1 in the proof of Theorem~\ref{thm:3d-ground-state} Claim~\eqref{claim-3d-ground-state}). Combining these arguments, we obtain that with high probability, in the time interval $[- (\log N) N^{-d/2}, (\log N) N^{-d/2}]$, not only $\tau_M$ has a global spin cluster, but also most of the vertices are always in this global spin cluster. In addition, it is easy to see that with high probability, at least one quarter of the spins are minus (resp.\ plus) at $M = - (\log N) N^{-d/2}$ (resp.\ $M = (\log N) N^{-d/2}$). This implies that a global avalanche occurs at some random time $M_{\star} \in [- (\log N) N^{-d/2}, (\log N) N^{-d/2}]$, where most of the spins flip.

In Section~\ref{subsec:prelim-3d}, we present some preliminary results. In Section~\ref{subsec:proof-avalanche}, we prove the global avalanche with weak disorder.

\subsection{Preliminary results}\label{subsec:prelim-3d}

We first prove several basic lemmas about the RFIM that will be used in Section~\ref{subsec:proof-avalanche}. For $M \in \mathbb{R}$, let the Hamiltonian of the ground state be
\begin{equation}
\label{eq:def-flip-disorder}
\mathcal{H}_M = \mathcal{H}_M(\epsilon, \{ h_v \}, N):= H_M (\tau_M)\,.
\end{equation}
For any set $A \subset \TT_N^d$, let $ h^A $ be obtained from $ h $ by flipping the sign of disorder on $ A $. That is,
\begin{equation}\label{eq:def-hA}
h^A(v) := h(v) \mathbbm{1}_{v \notin A} - h(v) \mathbbm{1}_{v \in A} \quad \mbox{for all } v \in \TT_N^d.
\end{equation}Let $\tau^A_M$ and $\mathcal{H}^A_M$ be the ground state and the associated Hamiltonian with respect to the disorder $ h^A $.
\begin{lemma}
    \label{lem:ground-state-flip}
    The following holds for any $ A\subset\mathbb{T}^d_N $, $ M\in\mathbb{R} $, $ \epsilon>0 $ and any external field $ \{h_v\} $. If $\tau_M|_{\partial_i A} = -1$ and $\tau_M|_{\partial_o A} = 1$, or $\tau_M|_{\partial_i A} = 1$ and $\tau_M|_{\partial_o A} = -1$, then we have
    $$
    \mathcal{H}^A_M \leq \mathcal{H}_M - 2|\partial_e A| + 2|M| |A|\,.
    $$
\end{lemma}
\begin{proof}
    Consider $\tau' (v):= \tau_M(v) \mathbbm{1}_{v \notin A} - \tau_M(v) \mathbbm{1}_{v \in A}$ for $v \in \TT_N^d$. By definition, we have
    \begin{align*}
        \mathcal{H}^A_M &\leq H_{\TT_N^d, M ,\epsilon h^A}(\tau') = -\sum_{u \sim v} \tau' (u) \tau' (v) - \sum_{v \in \TT_N^d} (M+\epsilon h_v^A ) \tau' (v) \\& \leq \mathcal{H}_M - 2|\partial_e A| + 2|M| |A|\,.\qedhere
    \end{align*}
\end{proof}

The following lemma is essentially taken from \cite{FisherFrolichSpencer}.

\begin{lemma}
\label{lem:control-change}
    Let $d \geq 3$. Recall from \eqref{eq:def-A-N} the definition of $\mathcal{A}_N$. We have the following.
    \begin{enumerate}
    \item For any $\alpha > 0$, there exists a constant $c > 0$ that depends on $\alpha$ such that for all $0 < \epsilon < c$:
    $$
    \mathbb{P}\Big[\sup_{0 \in A, A \in \mathcal{A}_N} \frac{|\mathcal{H}_M^A - \mathcal{H}_M|}{|\partial_e A|} \leq 1 \Big] \geq 1 - \alpha \quad \forall N \geq 1, M \in \mathbb{R}\,.
    $$\label{lem3.3-claim1}
    \item There exists a constant $C_1 > 0$ such that for all $0 < \epsilon < C_1^{-1}$
    $$
    \mathbb{P} \Big[\sup_{A \in \mathcal{A}_N, |\partial_e A| \geq (\log N)^{C_1}} \frac{|\mathcal{H}_M^A - \mathcal{H}_M|}{|\partial_e A|} \leq 1 \Big] \geq 1 - N^{-2d} \quad \forall N \geq 1, M \in \mathbb{R}\,.
    $$\label{lem3.3-claim2}
    \end{enumerate}
\end{lemma}

\begin{proof}
    Fix $ N\geq 1 $ and $ M\in\mathbb{R} $.
    For $A \subset \TT_N^d$, let $\Delta_A := \mathcal{H}_M^A - \mathcal{H}_M$.
    Note that given $h|_{A^c}$ (i.e., the restriction of $ h $ to $ A^c $), we have that $\Delta_A$ is a $2 \epsilon$-Lipschitz function with respect to $h_v$ for each $v \in A$. Applying \cite[Theorem~3.25]{vanHandel} (see also \cite{Gaussian_concerntration_Sudakov} and \cite{Gaussian_concerntration_Borell}), we obtain that for all $A \subset \TT_N^d$: 
    \begin{equation}
    \label{eq:lem-gaussian-est}
    \mathbb{P}\big[ \left|\Delta_A \right| \geq t \big] \leq 2 \exp \big( -\frac{t^2}{8 \epsilon^2 |A|} \big) \quad \forall t>0\,.
    \end{equation}
    For all $A, B \subset \TT_N^d$, let $ A\oplus B:=(A\cup B)\setminus (A\cap B) $ be the symmetric difference between $ A $ and $ B $, and let $ \widetilde{\Delta}_{A\oplus B} $ be defined as $ \Delta_{A\oplus B} $ but with respect to the external field $ h^{B} $ (so here we used the tilde symbol to emphasize that the underlying field is $ h^B $).
    Then by
    \begin{equation*}
        h^A(v) = h(v)\mathbbm{1}_{v\notin A}-h(v)\mathbbm{1}_{v\in A} = h^B(v)\mathbbm{1}_{v\notin A\oplus B}-h^B(v)\mathbbm{1}_{v\in A\oplus B},
    \end{equation*}
    we see that
    \begin{equation*}
        \Delta_{A}-\Delta_{B}=\widetilde{\Delta}_{A\oplus B}.
    \end{equation*}
    Therefore, we have
    \begin{equation}\label{eq:lem-gaussian-est-10}
    \mathbb{P}\big[ \left|\Delta_{A} - \Delta_{B} \right| \geq t \big] \leq 2 \exp \big( -\frac{t^2}{8 \epsilon^2 |A \oplus B|} \big) \quad \forall t>0\,.
    \end{equation}
    These two inequalities allow us to adapt the multi-scale analysis in \cite{FisherFrolichSpencer} \response{(which is based on a chaining method for simply connected sets, approximating them by unions of dyadic boxes of varying sizes)} and show that for some constant $C>0$,
    \begin{equation}\label{eq:lem-gaussian-est-result}
    \mathbb{P}\Big[ \sup_{0 \in A, A \in \mathcal{A}_N} \frac{|\Delta_A|}{|\partial_e A|} \leq 1 \Big] \geq 1- C \exp \big( -C/ \epsilon^2 \big) \quad \forall N \geq 1, M \in \mathbb{R}\,.
    \end{equation}
    \response{Note that although the definition of $\Delta_A$ is different from $F_A(h)$ in \cite{FisherFrolichSpencer}, inequalities~\eqref{eq:lem-gaussian-est} and~\eqref{eq:lem-gaussian-est-10} show that $\Delta_A$ satisfies the condition \cite[Eq.(10)]{FisherFrolichSpencer} which is the only property of $F_A(h)$ needed to prove \eqref{eq:lem-gaussian-est-result} for small $\epsilon$.} Despite the fact that \cite{FisherFrolichSpencer} addresses simply connected sets in $\mathbb{Z}^d$, their argument can be adapted to the torus case.
    The slight subtlety arises from the possibility that $\partial_e A$ may have multiple connected components.
    As such, a straightforward entropy bound might not be enough. However, in this scenario, these components (in fact at most two) must be non-contractable and thus have a length of at least $N$ (see more details on such topological result in \cite[Section~3]{Topology_of_torus_1} and \cite[Section~3]{Topology_of_torus_2}).
    Therefore by applying the same method as in \cite{FisherFrolichSpencer} to those $ \partial_e A $'s with only one connected component and applying a union bound to those $ \partial_e A $'s with multiple connected components, the torus case can be proved as well.
    Taking $\epsilon$ small in~\eqref{eq:lem-gaussian-est-result} yields the first claim. Following the method in \cite{FisherFrolichSpencer}, we can also show that there exists a constant $C_1 > 0$ such that 
    $$
    \mathbb{P}\Big[ \sup_{0 \in A, A \in \mathcal{A}_N, |\partial_e A| \geq (\log N)^{C_1}} \frac{|\Delta_A|}{|\partial_e A|} \leq 1 \Big] \geq 1- N^{-4d} \quad \forall 0 < \epsilon <C_1^{-1}, N \geq 1, M \in \mathbb{R}\,.
    $$
    \response{(As mentioned before, the argument in~\cite{FisherFrolichSpencer} applies verbatim to $\Delta_A$. To obtain the tail estimate, we sum~\cite[Eq.(15)]{FisherFrolichSpencer} over surface area at least $(\log N)^{C_1}$ and $k$. The relevant tail estimate appears in the second indented inequality after~\cite[Eq.(16)]{FisherFrolichSpencer}. By choosing $C_1$ sufficiently large, we obtain the desired bound.)} This implies the second claim by taking a union bound.
\end{proof}

The following two lemmas are derived from geometrical arguments. They will be used to prove the existence of a global spin cluster.

\begin{lemma}
\label{lem:geometric-connect}
    For any spin configuration $\sigma \in \{-1,1\}^{\TT_N^d}$ and two vertices $u,v$, if neither $u$ nor $v$ is enclosed by an interface, then we have $\sigma(u) = \sigma(v)$.
\end{lemma}
\begin{proof}
    Suppose that $\sigma(u) \neq \sigma(v)$. Let $ \mathsf C_u $ be the set of vertices that can be connected to $u$ with a nearest-neighbor path of the same sign as $\sigma(u)$, and let $ \mathbf C_u $ be the union of $ \mathsf C_u $ and all the connected components of $\mathbb{T}_N^d \setminus\mathsf{C}_u $ which do not contain $ v $.
    Define $\mathsf C_v $ and $ \mathbf C_v $ similarly. Then $ \mathbf C_u $ and $ \mathbf C_v $ are two simply connected sets, and we will now show that they are disjoint.
    By $ \sigma(u)\neq\sigma(v) $, we have $ \mathbf{C}_u\cap\mathsf{C}_v=\varnothing $. If a vertex $ w\in\mathbf{C}_v $, then any nearest neighbor path that connects $ w $ and $ u $ will contain at least one edge in $ \partial_e \mathsf{C}_v $. This fact combined with $ \mathbf{C}_u\cap\mathsf{C}_v=\varnothing $ yields $ w\notin\mathbf{C}_u $.
    Then $ \mathbf C_u $ and $ \mathbf C_v $ are disjoint simply connected sets, and thus at least one of them belongs to $\mathcal{A}_N$. This contradicts with the assumption that neither $u$ nor $v$ is enclosed by an interface. Therefore, we have $\sigma(u) = \sigma(v)$.
\end{proof}

\begin{lemma}
    \label{lem:geometric-2}
    For any two integers $N, F \geq 1$ with $F \leq \frac{N}{8}$\response{, let} $\sigma \in \{-1,1\}^{\TT_N^d}$ be a spin configuration such that any vertex in $\TT_N^d$ is not enclosed by any interface with an edge boundary of size at least $F$. Then there exists a spin cluster $A \subset \TT_N^d$ such that each connected component of $\TT_N^d \setminus A$ has $|\cdot|_\infty$-diameter at most $F$.
\end{lemma}

\begin{proof}
    
    Under the assumption of the lemma, we first find a spin cluster $ A \subset\TT^d_N $ with $ |\cdot|_{\infty} $-diameter at least $ \frac{N}{4} $, by investigating a series of interfaces that separate two distant vertices.
    To this end, consider two fixed vertices $ u, v \in\TT^d_N $ with $ |u-v|_{\infty}\geq \frac{N}{2} $.
    Recall from Lemma~\ref{lem:geometric-connect} that $\mathsf C_u $ is the set of vertices that can be connected to $u$ with a nearest-neighbor path of the same sign as $\sigma(u)$, and that $ \mathbf C_u $ is the union of $ \mathsf C_u $ and all the connected components of $ \mathbb{T}_N^d \setminus\mathsf{C}_u $ that do not contain $v$.
    Let $ \mathbf{C}^1 = \mathbf{C}_u $.
    For an integer $j \geq 1$, given $\mathbf{C}^j$ with $ v \not \in \mathbf{C}^j $, let $\mathsf{C}^{j+1}$ be the set of vertices that can be connected to $\partial_o \mathbf{C}^j$ with a path of the same sign as the spins on $\partial_o \mathbf{C}^j$ (this is valid since spins on $ \partial_o \mathbf{C}^j $ are indeed the same by induction).
    Then we let $ \mathbf{C}^{j+1} $ be the union of $ \mathbf{C}^{j} $ and $ \mathsf{C}^{j+1} $ as well as all the connected components of $\TT_N^d \setminus \mathsf{C}^{j+1}$ that do no contain $v$.
    This iteration process stops until $\mathbf{C}^m = \TT_N^d$ for some integer $m \geq 1$.
    See Figure~\ref{fig:3.11} for an illustration.
    \begin{figure}[h]
        \centering
        \includegraphics[width=0.5\textwidth]{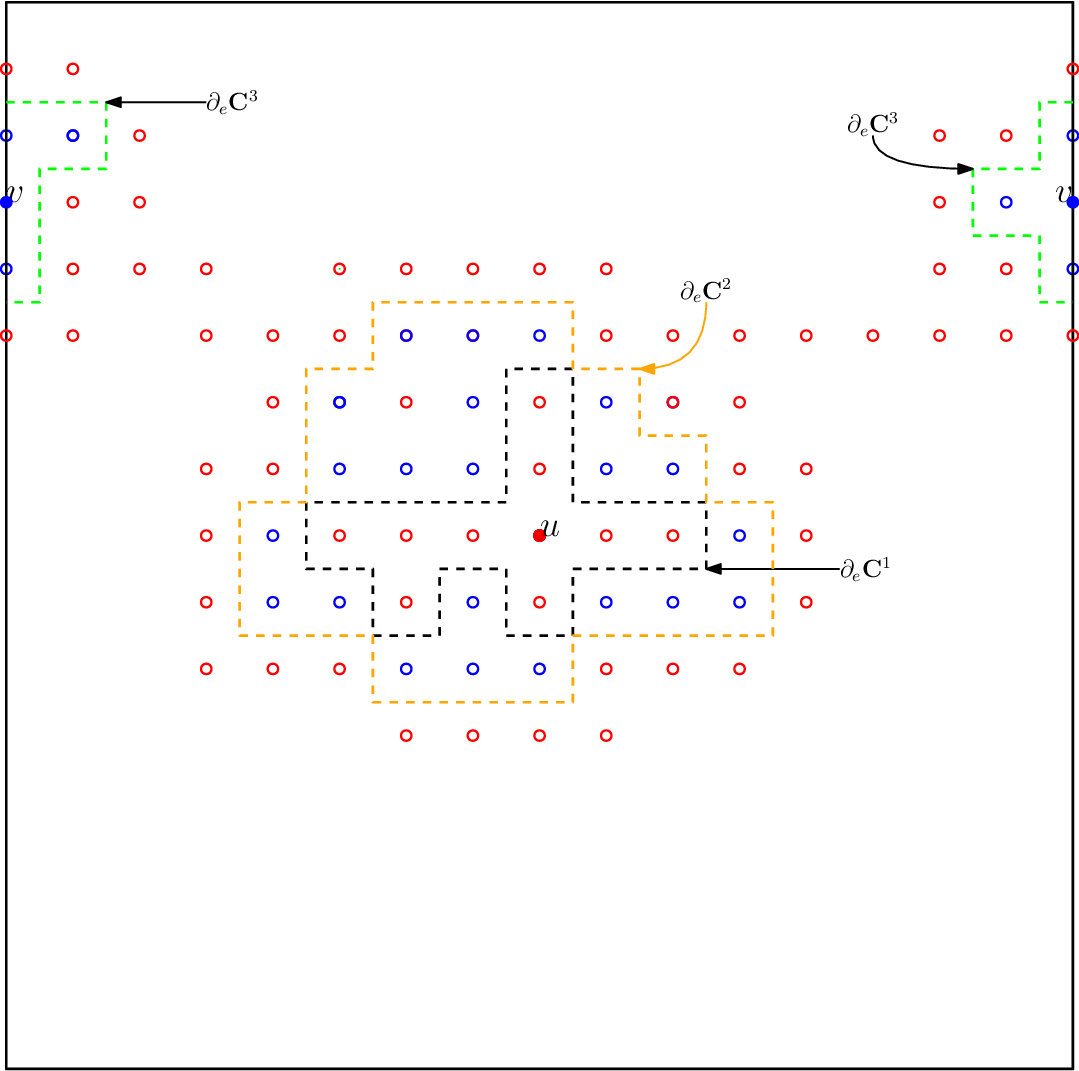}
        \caption{An illustration for $ \mathbf{C}^{j} $. Plus spins are colored in red and minus spins are colored in blue. \response{The set $\mathbf{C}^1$ is enclosed by the dashed black contour, the set $\mathbf{C}^2$ is enclosed by the dashed orange contour, and the sets $\mathbf{C}^3$ is the part outside the dashed green contour in the figure. In this figure, we have $\mathbf{C}^1 = \mathsf{C}^1$.}}
        \label{fig:3.11}
    \end{figure}
    
    Since $\mathbf{C}^j$ is a simply connected subset of $\TT_N^d$ and its edge boundary has plus and minus spins on the two sides respectively, by the assumption in the lemma-statement we have either $|\mathbf{C}^j| > \frac{N^d}{2}$ or $|\partial_e \mathbf{C}^j| \leq F$. Similarly, since $\TT_N^d \setminus \mathbf{C}^j$ is also a simply connected set, we have either $ |\TT_N^d \setminus \mathbf{C}^j| > \frac{N^d}{2} $ or $|\partial_e(\TT_N^d \setminus \mathbf{C}^j)| = |\partial_e \mathbf{C}^j| \leq F$.
    Therefore, we always have $|\partial_e \mathbf{C}^j| \leq F $ for any $1 \leq j \leq m$.
    This, together with the isoperimetric inequality~\eqref{eq:isoperimetric} and $ F\leq \frac{N}{8} $, implies that either $ {\rm Diam}(\mathbf{C}^j)\leq F $ (if $ \mathbf{C}^j\in\mathcal{A}_N $) or $ {\rm Diam}(\TT^d_N\setminus\mathbf{C}^j)\leq F $ (if $ \mathbf{C}^j\notin\mathcal{A}_N $), where ${\rm Diam}(\cdot)$ is the diameter defined with respect to the $|\cdot|_\infty$-distance.
    Let $ i_0\in [1,m] \cap \mathbb{Z}$ be the minimal $ i\geq 1 $ such that $ {\rm Diam}(\mathbf{C}^i)> F $.
    Then we have ${\rm Diam}(\mathbf{C}^{i_0-1}) \leq F$ and ${\rm Diam}(\TT_N^d \setminus \mathbf{C}^{i_0}) \leq F$, and thus $ \mathsf{C}^{i_0} $ contains a spin cluster with $ |\cdot|_{\infty} $-diameter at least $ |u - v|_{\infty}-2F\geq \frac{N}{4} $. Denote this spin cluster by $A$.

    Next, we prove that $A$ satisfies the assertion in the lemma (which together with $ F\leq\frac{N}{8} $ implies the uniqueness of $ A $).
    For any $ w \notin A $, let $ D_w $ be the connected component of $ \TT^d_N\setminus A $ containing $ w $.
    Then $\TT^d_N\setminus D_w = A\cup (\bigcup_{x\notin D_w} D_x)$ is also connected and thus $D_w$ is a simply connected set.
    Moreover, since $A $ is a spin cluster, edges in $ \partial_e D_w $ have endpoints with opposite spins.
    So $ \partial_e D_w $ is an interface (note the simple fact that either $ D_w $ or $ \TT^d_N\setminus D_w $ has at most half of the vertices).
    By our assumption on $ \sigma $ in the lemma-statement we have $ |\partial_e D_w|\leq F $ for any $ w\notin A $, thereby proving the lemma.
    
\end{proof}

\subsection{Proof of avalanche with weak disorder}\label{subsec:proof-avalanche}

In this subsection, we follow the strategy sketched in Section~\ref{subsec:3d-proof-strategy} to prove the avalanche with weak disorder. We first control the change of the RFIM Hamiltonian for large sets. Recall the constant $C_1$ from Lemma~\ref{lem:control-change}.

\begin{lemma}
    \label{lem:A-large}
    Let $\mathcal{G}_1$ be the event that $|\mathcal{H}_M^A - \mathcal{H}_M| \leq \frac{3}{2} |\partial_e A|$ for all $A \in \mathcal{A}_N$ with $|\partial_e A| \geq (\log N)^{C_1}$ and all $|M| \leq (\log N ) N^{-d/2}$. For all $0 < \epsilon < C_1^{-1}$ and all sufficiently large $N$, we have
    \begin{equation}
    \label{eq:lem3.6}
    \mathbb{P}[\mathcal{G}_1] \geq 1 - o_N(1).
    \end{equation}
\end{lemma}
\begin{proof}
    Let $ \mathfrak{M}= N^{-d} \mathbb{Z} \cap [-(\log N) N^{-d/2}, (\log N) N^{-d/2}]$. By Claim~\ref{lem3.3-claim2} in Lemma~\ref{lem:control-change}, for all $0 < \epsilon < C_1^{-1}$
    \begin{equation}
    \label{eq:lem3.6-1}
    \mathbb{P} \Big[\sup_{A \in \mathcal{A}_N, |\partial_e A| \geq (\log N)^{C_1}} \frac{|\mathcal{H}_M^A - \mathcal{H}_M|}{|\partial_e A|} \leq 1 \Big] \geq 1 - N^{-2d} \quad \forall N \geq 1, M \in \mathfrak{M}\,.
    \end{equation}
    For any $M \in [-(\log N) N^{-d/2}, (\log N) N^{-d/2}]$, there exists an $ M' \in \mathfrak{M} $ such that $|M - M'| \leq N^{-d}$. Moreover, by the definition of $\mathcal{H}_M$ in \eqref{eq:def-flip-disorder} and~\eqref{eq:def-hamiltonian}, we have 
    \begin{equation}
    \label{eq:lem3.6-2}
    \mathcal{H}_M -\mathcal{H}_{M'} \leq H_M(\tau_{M'}) - H_{M'}(\tau_{M'}) \leq |M - M'| \cdot N^d \leq 1.
    \end{equation}
    Similarly, we also have $\mathcal{H}_{M'} -\mathcal{H}_M \leq 1$ and thus $|\mathcal{H}_{M'} -\mathcal{H}_M| \leq 1$. This, together with~\eqref{eq:lem3.6-1} and a union bound over $M \in \mathfrak{M} $, implies that for all $0 < \epsilon < C_1^{-1}$ and all sufficiently large $N$:
    \begin{align*}
        &\quad \mathbb{P} \Big[\sup_{M \in [-(\log N) N^{-d/2}, (\log N) N^{-d/2}]} \sup_{A \in \mathcal{A}_N, |\partial_e A| \geq (\log N)^{C_1}} \frac{|\mathcal{H}_M^A - \mathcal{H}_M|}{|\partial_e A|} \geq \frac{3}{2} \Big] \\
        &\leq \sum_{M \in \mathfrak{M}} \mathbb{P} \Big[\sup_{A \in \mathcal{A}_N, |\partial_e A| \geq (\log N)^{C_1}} \frac{|\mathcal{H}_M^A - \mathcal{H}_M|}{|\partial_e A|} \geq 1 \Big] = o_N(1)\,.\qedhere
    \end{align*}
\end{proof}

In light of Lemma~\ref{lem:A-large}, we introduce the definition of global spin cluster (as announced in Section~\ref{subsec:3d-proof-strategy}). Let
\begin{equation}
\label{eq:def-tilde-N}
    \widetilde N = (\log N)^{C_1}.
\end{equation}

\begin{definition}
    \label{def:global-spin-config}
    For a spin configuration $\sigma \in \{-1,1\}^{\TT_N^d}$, we define a spin cluster of $\sigma$ as a global spin cluster if after removing this cluster, all remaining connected components of $\TT_N^d$ have $|\cdot|_\infty$-diameters at most $\widetilde N$; see Figure~\ref{fig:global_spin}. It is easy to see that the global spin cluster (if it exists) is unique for all sufficiently large $N$.
\end{definition}

\begin{figure}[h]
    \centering
    \includegraphics[width=0.5\textwidth]{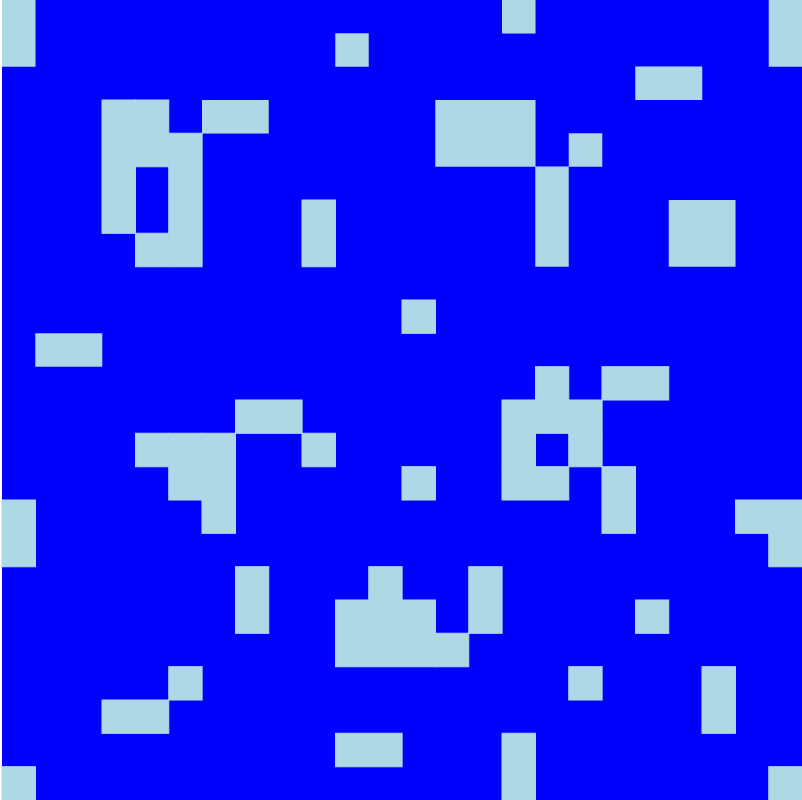}
    \caption{A spin configuration with a plus global spin cluster. Here we color the plus spins in blue and minus spins in lightblue. Note that there exists a blue cluster such that all the other connected components after removing this cluster have $ |\cdot|_{\infty} $-diameters at most $ \widetilde{N} $.}
    \label{fig:global_spin}
\end{figure}

We can derive the following lemma from Lemmas~\ref{lem:ground-state-flip} and~\ref{lem:geometric-2}.
\begin{lemma}\label{lem:global-spin}
    On the event $\mathcal{G}_1$ defined in Lemma~\ref{lem:A-large}, $\tau_M$ has a global spin cluster for all $|M| \leq (\log N) N^{-d/2}$.
\end{lemma}

\begin{proof}
    On the event $\mathcal{G}_1$, by the isoperimetric inequality~\eqref{eq:isoperimetric}, we have that for all $A \in \mathcal{A}_N$ with $|\partial_e A| \geq (\log N)^{C_1}$ and $|M| \leq (\log N) N^{-d/2}$,
    $$
    \mathcal{H}_M^A \geq \mathcal{H}_M - \frac{3}{2} |\partial_e A| > \mathcal{H}_M - 2|\partial_e A| + 2 (\log N) N^{-d/2} |A|.
    $$
    Thus, we can deduce from Lemma~\ref{lem:ground-state-flip} that there is no spin interface in $\tau_M$ with an edge boundary of size at least $(\log N)^{C_1}$. Combining this with Lemma~\ref{lem:geometric-2} yields the lemma.
\end{proof}

Recall the definition of good vertices from~\eqref{eq:strategy-3}. Next, we show that when $\epsilon$ is small, with high probability, most of the vertices in $\TT_N^d$ are good.
\begin{lemma}\label{lem:good-vertex}
    Fix $\delta>0$. Let $\mathcal{G}_2$ be the event that at least $(1 - \frac{\delta}{2})N^d$ vertices in $\TT_N^d$ are good. The following inequality holds for all sufficiently small $\epsilon$:
    \begin{equation}
    \label{eq:vertex-being-good}
    \mathbb{P}[\mathcal{G}_2] \geq 1-o_N(1) \quad \mbox{as }N \rightarrow \infty.
\end{equation}
\end{lemma}
\begin{proof}
    First, we show that the probability of a vertex being good can be arbitrarily close to 1 by choosing $ \epsilon $ sufficiently small. Similar to Claim~\ref{lem3.3-claim1} in Lemma~\ref{lem:control-change}, we have that the following event happens with probability arbitrarily close to 1 by choosing $ \epsilon $ sufficiently small: For all $A \in \mathcal{A}_N$ with $u \in A$ and $A \subset B(u, 2 (\log N)^{C_1})$, we have
    $$
    |\mathcal{H}_{B(u,2 (\log N)^{C_1}), M, \epsilon h}^\pm - \mathcal{H}_{B(u,2 (\log N)^{C_1}), M, \epsilon h^A}^{\pm} | \leq \frac{1}{2} |\partial_e A| \quad \mbox{for } M = -(\log N) N^{-d/2}.
    $$
    This condition already implies that the vertex $u$ is good since 
    $$
    |\mathcal{H}_{B(u,2 (\log N)^{C_1}), M', \epsilon h}^\pm - \mathcal{H}_{B(u,2 (\log N)^{C_1}), M, \epsilon h}^{\pm} | \leq |M - M'| \times |B(u,2 (\log N)^{C_1})|,
    $$
    which can be proved similarly to~\eqref{eq:lem3.6-2}.

    \response{Now we have shown that there exists $ c >0 $ such that for all $ 0< \epsilon < c $ and every vertex $ u\in\TT^d_N $, $\mathbb{P}[u{\rm\ is\ not\ good}]\leq \delta/4$.
    In addition, note that for any two vertices $u,v$ with $|u - v|_\infty > 4 \widetilde N$, the events ``$u$ is good'' and ``$v$ is good'' are independent.
    Therefore, by applying Chebyshev's inequality, we have
    \begin{equation*}
        \begin{aligned}
            \mathbb{P}[\mathcal{G}^c_2]
            &\leq \mathbb{P}\Bigg[\sum_{u\in\TT^d_N}\mathbf{1}_{u{\rm\ is\ not\ good}} - \mathbb{E}[\mathbf{1}_{u{\rm\ is\ not\ good}}] \geq\frac{\delta}{4}N^d\Bigg]\\
            &\leq \frac{1}{(\delta/4)^2N^{2d}}{\rm Var}\Bigg(\sum_{u\in\TT^d_N}\mathbf{1}_{u{\rm\ is\ not\ good}}\Bigg)\\
            &= \frac{1}{(\delta/4)^2N^{2d}} \sum_{u, v\in\TT^d_N, |u-v|_{\infty}\leq4\widetilde{N}} {\rm Cov}\left(\mathbf{1}_{u{\rm\ is\ not\ good}}, \mathbf{1}_{v{\rm\ is\ not\ good}}\right) = o_N(1). \qedhere
        \end{aligned}
    \end{equation*}}
\end{proof}

Now we complete the proof of Claim~\eqref{claim-3d-ground-state} in Theorem~\ref{thm:3d-ground-state}.

\begin{proof}[Proof of Theorem~\ref{thm:3d-ground-state} Claim~\eqref{claim-3d-ground-state}]
    Fix $\delta >0$. We assume that $\delta < \frac{1}{10}$. The proof consists of three steps, following the strategy in the second last paragraph of Section~\ref{subsec:3d-proof-strategy}.

    \textbf{Step 1.} Recall from Lemma~\ref{lem:global-spin} that on the event $\mathcal{G}_1$, $\tau_M$ has a global spin cluster for all $|M| \leq (\log N) N^{-d/2}$. In this step, we show that on the event $\mathcal{G}_1$, a good vertex belongs to the global spin cluster in $\tau_M$ for all $|M| \leq (\log N) N^{-d/2}$. 

    Let $u$ be a good vertex, and let $|M| \leq (\log N) N^{-d/2}$. We first consider the case that the global spin cluster in $\tau_M$ is minus. Let $\widetilde \tau$ be the RFIM ground state configuration on $B(u, 2 \widetilde N)$ with respect to the minus boundary condition. Using~\eqref{eq:strategy-3}, together with \eqref{eq:isoperimetric} and Lemma~\ref{lem:ground-state-flip}, we have that 
    \begin{equation}\label{eq:thm1.2-claim2-0}
        \mbox{$u$ is not enclosed by any spin interface in $\widetilde \tau$}.
    \end{equation}\response{In fact, if $u$ is enclosed by a spin interface, then by a similar argument to Lemma~\ref{lem:ground-state-flip}, there exists $A \in \mathcal{A}_N$ with $u \in A$ and $A \subset B(u, 2 (\log N)^{C_1})$ such that
    \begin{equation}\label{eq:thm1.2-claim2-1}
    |\mathcal{H}_{B(u,2 (\log N)^{C_1}), M, \epsilon h}^- - \mathcal{H}_{B(u,2 (\log N)^{C_1}), M, \epsilon h^A}^- | \geq 2|\partial_e A| - 2|M| \cdot |A|.
    \end{equation}
    However, \eqref{eq:strategy-3} implies that the left-hand side of~\eqref{eq:thm1.2-claim2-1} is at most $|\partial_e A|$. Moreover, by~\eqref{eq:isoperimetric}, we have $|\partial_e A| \geq C_2 |A|^{1-1/d} \geq 2|M| \cdot |A|$, which leads to a contradiction. This proves~\eqref{eq:thm1.2-claim2-0}.}
    
    Then by duality (or by a similar argument in Lemma~\ref{lem:geometric-connect}), we deduce that $u$ is connected to $ \partial_o B(u, 2 \widetilde N) $ by a minus path in $\widetilde \tau$.
    By Definition~\ref{def:global-spin-config} and the fact that the global spin cluster in $\tau_M$ is minus, there exists a minus contour \response{(i.e., a $*$-connected path with the same start and end vertex)} $\mathcal{C}$ in $\tau_M$ that lies in $B(u, 2 \widetilde N)$ and surrounds $u$, such that any path from $u$ to $ \partial_o B(u, 2\tilde{N}) $ intersects $ \mathcal{C}$. 
    Since all the spins in $\tau_M$ are minus on this contour, by monotonicity, we deduce that the spin configuration of $\tau_M$ inside $\mathcal{C}$ is dominated by that of $\widetilde \tau$ (in fact, they are the same). Therefore, $u$ is connected to $\mathcal{C}$ by a minus path in $\tau_M$, which implies that $u$ belongs to the minus global spin cluster. The case that the global spin cluster is plus can be treated similarly using the plus boundary condition case in~\eqref{eq:strategy-3}. This proves the claim.

    \textbf{Step 2.} In this step, we show that with high probability, at least one quarter of the spins are minus (resp.\ plus) at $M = - (\log N) N^{-d/2}$ (resp.\ $M = (\log N) N^{-d/2}$). This is equivalent to show that
    \begin{equation}
    \label{eq:prop3.4-2}
    \mathbb{P}\Big[\sum_{v \in \TT_N^d} \tau_M(v) \geq \frac12 N^d \Big] \leq o_N(1) \quad \mbox{for } M = - (\log N) N^{-d/2}.
    \end{equation}
    Let $\mathcal{H}_{M, -\epsilon h}$ be the Hamiltonian of the ground state with the external field $\{-h_v\}_{v \in \TT_N^d}$. On the event that $\sum_{v \in \TT_N^d} \tau_M(v) \geq \frac12 N^d$, we have
    \begin{align*}
    \mathcal{H}_{M, -\epsilon h} &\leq H_{\TT_N^d, M, -\epsilon h} ( - \tau_M)=  H_{\TT_N^d, M, \epsilon h} ( \tau_M) + 2 M \sum_{v \in \TT_N^d} \tau_M(v) \leq \mathcal{H}_M - (\log N) N^{d/2}.
    \end{align*}
    However, by~\eqref{eq:lem-gaussian-est}, we have 
    $$
    \mathbb{P}\Big[|\mathcal{H}_{M, -\epsilon h} - \mathcal{H}_M| \geq t\Big] \leq 2 \exp \big(-\frac{t^2}{8\epsilon^2 | \TT_N^d|} \big) \quad \forall t>0\,.
    $$
    Combining the preceding two inequalities yields~\eqref{eq:prop3.4-2}.

    \textbf{Step 3.} By Step 1, we see that on the event $\mathcal{G}_1 \cap \mathcal{G}_2$, $\tau_M$ has a global spin cluster which contains at least $(1 - \frac{\delta}{2}) N^d$ vertices for all $|M| \leq (\log N) N^{-d/2}$. Combining this with Step 2 and the fact that $\delta < \frac{1}{10}$, we obtain that, with high probability, this global spin cluster is minus at $M = -(\log N) N^{-d/2}$ and plus at $M = (\log N) N^{-d/2}$. Therefore, the global spin cluster changes from minus to plus at some (random) time $M_{\star} \in [-(\log N) N^{-d/2}, (\log N) N^{-d/2}]$, where the flipping size is at least $(1 - \delta) N^d$. By Lemmas~\ref{lem:A-large} and~\ref{lem:good-vertex}, both the events $\mathcal{G}_1$ and $\mathcal{G}_2$ happen with high probability when $\epsilon$ is sufficiently small. This concludes Claim~\eqref{claim-3d-ground-state} in Theorem~\ref{thm:3d-ground-state}.\qedhere
\end{proof}

\begin{remark}\label{rmk:tight-threshold}
    The coefficient $\log N$ in $(\log N) N^{-d/2}$ was chosen to be an arbitrarily large number that tends to infinity as $N$ tends to infinity, and it was only used in the proof of~\eqref{eq:prop3.4-2}. By slightly modifying the proof, our method can show that for all sufficiently small $\epsilon$, the sequence of random variables $(N^{d/2} M_\star(N))_{N \geq 1}$ is tight. Moreover, any subsequential limit of these random variables is non-degenerate. This follows from the fact that for any fixed $\lambda>0$, the Radon-Nikodym derivative between $(h_v)$ and $(h_v + \lambda N^{-d/2})$ is bounded with probability close to 1, and if we change $(h_v)$ to $(h_v + \lambda N^{-d/2})$, $M_\star$ would become $M_\star + \epsilon \lambda N^{-d/2}$ (if they exist).
\end{remark}


\section{Polluted bootstrap percolation}
\label{sec:polluted}

In this section, we study polluted bootstrap percolation in dimensions $d \geq 2$. The main goal is to prove Theorem~\ref{thm:1.4}. 

Let $p, q \in [0,1]$ be two parameters with $p + q \leq 1$. In the initial configuration, each vertex on $\mathbb{Z}^d$ is independently chosen to be closed with probability $q$, open with probability $p$, and empty with probability $1 - p - q$. The configuration evolves according to the (polluted) bootstrap percolation with threshold $r = d$, and let $\mathbb{P}_{p,q,\infty}$ denote the law of the final configuration. Recall that a configuration $A$ is said to be larger than (or to dominate) another configuration $B$ if all open vertices in $B$ are also open in $A$, and all closed vertices in $A$ are also closed in $B$.

Our proof of Theorem~\ref{thm:1.4} is based on a coarse-graining argument. Assume that $p < 10^{-10^{10d}}$. Let $K>10^{10^{10d}}$ be a large integer to be chosen. We also require that $ K $ is the square of an integer and is divisible by $(1000d)!$, so that $ \sqrt{K}/t\in\mathbb{Z} $ for all $ 1\leq t\leq 100d $. Let the coarse-graining scale
\begin{equation}\label{eq:def-Ln}
    L_n = K^n \lfloor p^{-1} \rfloor \quad \mbox{for integers }n \geq 0.
\end{equation}
Recall that $B(x,r) = \{ y \in \mathbb{Z}^d: |x-y|_\infty \leq r \}$ for $x \in \mathbb{Z}^d$ and an integer $r \geq 0$. We call a box $B(x,r)$ an $L_n$-box if $x \in L_n \mathbb{Z}^d$ and $r = L_n$. 

In what follows, we often consider the evolution of vertices restricted to a finite domain $A$, while the rule of the evolution remains the same. Namely, closed and open vertices remain in the same state, and whenever an empty vertex in $A$ has at least $d$ neighbors in $A$ that are open, this vertex becomes open. Throughout Section~\ref{sec:polluted}, we will abbreviately use the term ``final configuration of $A$'' to refer to the final configuration in the evolution of vertices restricted to $A$, which is different from the final configuration in $A$ in the evolution of $\mathbb{Z}^d$. Since the evolution is increasing, we have two other equivalent ways to determine the final configuration in the evolution of $\mathbb{Z}^d$, as described below.

\begin{enumerate}[(A)]
    \item\label{evolve-1} We arrange the vertices in an arbitrary order. At each time $t \geq 1$, we check if the status of the $t$-th vertex will change from empty to open. If it changes, we continue to check its neighboring vertices, until no neighboring vertices can change the status.

    \item\label{evolve-2} We can also check the evolution subsequentially among boxes of different scales. As before, the closed or open vertices always remain in the same state. At time $t = 0$, we consider the evolution of vertices restricted to the $L_0$-boxes, and let an empty vertex be open if it is open in the final configuration of some $L_0$-box containing it. Note that the resulting configuration may continue to evolve, since for instance the update in one $ L_0 $-box may change the evolution for an intersecting $ L_0 $-box. However, we do not perform such evolution that is triggered by intersecting boxes, and instead we proceed to time $t = 1$. At time $t = 1$, we consider the evolution of vertices restricted to $L_1$-boxes, and let an empty vertex be open if it is open in the final configuration of some $L_1$-box containing it. We see that the open vertices at time $t = 0$ will still be open at time $t = 1$ since any $L_0$-box is contained in some $L_1$-box. Similarly, for time $t \geq 2$, we perform the evolution in $L_t$-boxes.

\end{enumerate}

\begin{lemma}\label{lem:infinite-evolve}
    For any initial configuration, the two evolutions described above and the one described before Theorem~\ref{thm:1.4} yield the same final configuration.
\end{lemma}

\begin{proof}
    We see that all three evolutions are increasing, as the closed or open vertices remain in the same state, and an empty vertex remains open once it becomes open. Therefore, all three evolutions converge to a final configuration in any finite domain in a finite time. \response{Furthermore, we can verify that all these final configuration do not evolve at any vertex. Using this property and induction, we can show that the final configuration of any of these evolutions dominates the finite time configuration of any other evolution. Specifically, if the final configuration of one evolution dominates the configuration of another at time $k$, then by monotonicity, it also dominates at time $k+1$, which implies that it dominates any finite time configuration.} Therefore, their final configurations are the same. \qedhere
\end{proof}

We will use the evolution as in~\eqref{evolve-2} to prove Theorem~\ref{thm:1.4}. To this end, we consider the final configuration of an $L_n$-box and define good $L_n$-boxes. We will see from the definition below whether an $L_n$-box is good or not is measurable with respect to the initial configuration restricted to this box. Hence, the statuses of different $L_n$-boxes are independent if the pairwise $|\cdot|_\infty$-distances between their centers are at least $3L_n$. 

\begin{definition}\label{def:good-box}
    \begin{enumerate}[(1)]
        \item For an $L_0$-box $B(x,L_0)$, we will call it good if in the final configuration of $B(x,L_0)$, all open clusters have $|\cdot|_\infty$-diameters at most $100^d$.

        \item Given the definition of good $L_{n-1}$-boxes, we call an $L_n$-box $B(x, L_n)$ good if one of the following two conditions hold:
        \begin{enumerate}[(i)]
            \item There is no bad $L_{n-1}$-box contained in $B(x, L_n)$ (a box that is not good is called bad).
            \item \label{def:good-box-case-b} There are bad $L_{n-1}$-boxes contained in $B(x,L_n)$, but all of them are contained in $B(y, 3L_{n-1})$ for some $y \in L_{n-1}\mathbb{Z}^d$. Moreover, all the open clusters in the final configuration of $B(x,L_n)$ that intersect $B(y, 3 L_{n-1})$ are contained in $B(y, \sqrt{K} L_{n-1})$; \response{see Figure~\ref{fig:def-good-L_n-boxes} for an illustration.}
        \end{enumerate}
    \end{enumerate}
\end{definition}

\begin{figure}[h]
    \centering
    \includegraphics[width=0.5\textwidth]{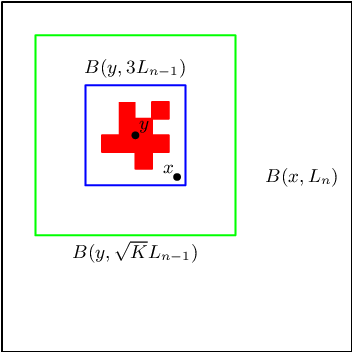}
    \caption{\response{An illustration of case~\eqref{def:good-box-case-b} for a good $L_n$-box $B(x, L_n) $. The boundary of $ B(y, \sqrt{K}L_{n-1}) $ is colored in green, and the boundary of $ B(y, 3L_{n-1}) $ is colored in blue. The bad $ L_{n-1} $-boxes, colored in red, are all contained in $ B(y, 3L_{n-1})$. Note that we allow bad $ L_{n-2} $-boxes to exist in $ B(x, L_n)\setminus B(y, 3L_{n-1}) $, but they are not illustrated. In the final configuration of $B(x,L_n)$, all the open clusters that intersect $ B(y, 3L_{n-1}) $ are contained in $ B(y, \sqrt{K}L_{n-1}) $.}}
    \label{fig:def-good-L_n-boxes}
\end{figure}

Theorem~\ref{thm:1.4} essentially follows from the following lemma.

\begin{lemma}\label{lem:bad-probability}
    For an integer $n \geq 0$, let $\mathbb{P}_{p,q, L_n}$ denote the law of the initial configuration restricted to $B(0, L_n)$. There exist a constant $C>0$ and a function $\mathscr{L}: \mathbb{N} \rightarrow (0,\infty)$ such that for all $K \geq C$, $p < \mathscr{L}(K)^{-1}$, and $q = \mathscr{L}(K) p^d$, we have
    \begin{equation}\label{eq:bound-bad-probability}
    p_n = p_n (p, q, K) := \mathbb{P}_{p,q, L_n}[\mbox{$B(0,L_n)$ is bad}] \leq p^{2d}e^{-K\cdot2^{n}} \quad \mbox{for all integers }n \geq 0.
    \end{equation}
\end{lemma}

In fact, by combining Lemmas~\ref{lem:bad-probability} and~\ref{lem:domination} below, we can show that open clusters in the final configuration of a good $L_n$-box are very sparse. More precisely, they can be stochastically dominated by a random environment containing independently sampled open vertices and open cubes at each vertex. The density of open vertices is proportional to $p$, and the diameters of the open cubes have a decaying right tail given by $p_k$ (see \response{Definition~\ref{def:Qn} and} Lemma~\ref{lem:domination} for the precise statement). It is easy to see that in this environment, the origin is contained in a bounded open cluster with high probability, and the probability of the origin being open tends to 0 as $p$ tends to 0, thus proving Theorem~\ref{thm:1.4} (see the proof of Theorem~\ref{thm:1.4}).

Next, we sketch the proof strategy of Lemma~\ref{lem:bad-probability} based on an induction argument. The case $n = 0$ will be proved in Lemma~\ref{lem:n=0} by noting that with high probability, there will not be too many open vertices getting close to each other in an $L_0$-box. Suppose that \eqref{eq:bound-bad-probability} is proved for $n \leq k$. For $n = k+1$, by independence, the probability that there are bad $L_{n-1}$-boxes contained in $B(0, L_n)$ but cannot be covered by $B(y, 3L_{n-1})$ for any $y \in L_{n-1} \mathbb{Z}^d$ is at most $(2K - 1)^{2d} \times p_{n-1}^2$, which is negligible. Thus, we may assume that all the bad $L_{n-1}$-boxes can be covered by $B(y, 3L_{n-1})$ for some $y \in L_{n-1} \mathbb{Z}^d$. We will first show that the final configuration of $B(0, L_n) \setminus B(y, 3L_{n-1})$ can be stochastically dominated by a random environment containing open cubes and closed vertices independently sampled at each vertex (see Lemma~\ref{lem:domination} for the precise statement). In this environment, each vertex is open (resp.\ closed) with probability up-to-constants equivalent to $p$ (resp.\ $q$), and in addition the diameters of the open cubes at each vertex have right tails that decay very fast. Then we will use the assumption $q \geq Cp^d$ to show that the initially open vertices in $B(y, 3L_{n-1})$ cannot grow to distance $\sqrt{K} L_{n-1}$ in this random environment with very high probability (Lemma~\ref{lem:control-growth}). This proves the claim for $n = k + 1$, and by induction, proves Lemma~\ref{lem:bad-probability}. 

In Section~\ref{subsec:n=0}, we prove the case $n = 0$ in Lemma~\ref{lem:bad-probability}. In Section~\ref{subsec:4.2}, we complete the proof of Lemma~\ref{lem:bad-probability} for all $n \geq 1$ and Theorem~\ref{thm:1.4}, assuming Lemma~\ref{lem:control-growth}. In Section~\ref{subsec:lem4.10}, we prove Lemma~\ref{lem:control-growth} using delicate geometrical arguments.

\subsection{The base case}\label{subsec:n=0}

The following lemma provides a lower bound on the number of initially open vertices in any open cluster during the evolution of bootstrap percolation. Recall that the $|\cdot|_\infty$-diameter of a set is the maximum $|\cdot|_\infty$-distance between any pair of points in the set. In particular, the $|\cdot|_\infty$-diameter of a single-point set is 0.

\begin{lemma}
    \label{lem:initial-open-number-lower-bound}
    Let $ r\geq2 $. During the evolution of (polluted) bootstrap percolation with threshold $r$ (recall its definition from Section~\ref{subsec:glauber-evolution}) on vertices restricted to any finite domain $A$, if open vertices form a connect component with $|\cdot|_\infty$-diameter $N$, then this component must contain at least \response{$\frac{N+2}{2}$} initially open vertices.
\end{lemma}
\begin{proof}
    Consider the evolution restricted to $A$. Define $U$ to be the sum of the $|\cdot|_\infty$-diameters of open components plus two times the number of open components. We claim that $U$ is non-increasing during the evolution. Consider each time an open vertex $v$ is produced. If all of its open neighbors come from the same open component, then the number of open components will not change and in addition the diameter of this open component will not increase, implying that $U$ will not increase. If the open neighbors of $v$ come from different open components, then the number of open components decreases by at least 1 and the sum of diameters increases by at most 2. (Note that the sum of diameters can increase by 2, for instance, when two isolated open vertices with $|\cdot|_\infty$-distance 2 grow into an open component with three vertices.) Hence, $U$ will still not increase. In addition, in the initial configuration, $U$ is at most two times the number of open vertices. Therefore, an open component with $|\cdot|_\infty$-diameter $N$ includes at least \response{$\frac{N + 2}{2}$} initially open vertices. 
\end{proof}

Now we prove the case $n=0$ in Lemma~\ref{lem:bad-probability}. We also control the open clusters in the final configuration of a good $L_0$-box using a stochastic dominance argument.

\begin{lemma}\label{lem:n=0}
    Fix $K \geq 1$. For all sufficiently small $p$ (depending on $d, K$) and $q>0$, we have $ p_0 \leq p^{2d}e^{-K} $. Moreover, conditioned on $B(0,L_0)$ being good, the open clusters in the final configuration of $B(0, L_0)$ are stochastically dominated by a set $S$ sampled as follows. For each $x \in B(0, L_0)$ and $0 \leq j \leq 100^d$, we independently sample an open cube $B(x, j) \cap B(0, L_0)$ with probability $p^{1 + j/3}$, and let $S$ be the union of these open cubes.
\end{lemma}
\begin{proof}
    First we prove that $p_0 \leq p^{2d}e^{-K}$ for all sufficiently small $p$. If there exists an open cluster with $|\cdot|_\infty$-diameter at least $100^d$ in the final configuration of $B(0,L_0)$, consider the first time during the evolution that an open component with $|\cdot|_\infty$-diameter at least $100^d$ is formed. At this time, the $|\cdot|_\infty$-diameter of this open component is at most $2 \cdot 100^d$. Therefore, by Lemma~\ref{lem:initial-open-number-lower-bound}, there exists $x \in B(0,L_0)$ and $100^d \leq k \leq 2 \cdot 100^d$ such that there are at least \response{$(k+2)/2$} initially open vertices in $B(x, k)$. For a fixed $x$, this probability is upper-bounded by $\mathbb{P}[X_1 + \ldots + X_{(2k+1)^d} \geq (k+2)/2]$ where $(X_i)_{i \geq 1}$ are i.i.d.\ Bernoulli random variables with $\mathbb{P}[X_i = 1] = p$. By Markov's inequality, for any $t>0$, we have
    \begin{equation}\label{eq:n=0-markov}
        \begin{aligned}
            \mathbb{P}[X_1 + \ldots + X_{(2k+1)^d} \geq (k+2)/2]
            &\leq e^{-(k+2) t/2} \mathbb{E}[\exp(t (X_1 + \ldots + X_{(2k+1)^d}))] \\
            &= e^{-(k+2) t/2} (1 + p(e^t-1))^{(2k+1)^d}.
        \end{aligned}
    \end{equation}
    By taking $e^t = p^{-1}$, we obtain that the above probability is upper-bounded by $C p^{(100^d+1)/2}$ for some constant $C$ depending only on $d$ and for all $100^d \leq k \leq 2 \cdot 100^d$. Since $|B(0,L_0)| = (2 \lfloor p^{-1} \rfloor +1)^d$ and the number of choices for $k$ is at most $100^d + 1$, we obtain that $p_0 \leq |B(0,L_0)| \times (100^d + 1) \times C p^{(100^d+1)/2} \leq p^{2d}e^{-K}$ for all sufficiently small $p$.
    
    Next, we prove the stochastic dominance argument. Fix any vertex $v \in B(0,L_0)$. Consider the open cluster $O_v$ in the final configuration of $B(0,L_0)$ containing $v$ and denote by $D_v$ the $|\cdot|_\infty$-diameter of $O_v$. We define $\widetilde O_v = O_v$ if $D_v\leq100^d$, and define $\widetilde O_v = B(v, 100^d) \cap B(0,L_0)$ if $D_v>100^d$. We first show that the law of $\widetilde O_v$ is stochastically dominated by the set $S_v$, where $S_v$ is obtained by independently sampling an open cube $B(v,j) \cap B(0,L_0)$ with probability $p^{1 + j/3}$ for each $0 \leq j \leq 100^d$ and then taking the union of these open cubes.
    It suffices to show that $\mathbb{P}[D_v \geq j] \leq p^{1+j/3}$ for all $1 \leq j \leq 100^d$. By Lemma~\ref{lem:initial-open-number-lower-bound}, on $\{D_v = j\}$, there would be at least \response{$\frac{j+2}{2}$} initially open vertices in $B(v,j)$. 
    Applying similar arguments to~\eqref{eq:n=0-markov}, we obtain that $\mathbb{P}[D_v = j]\leq \frac12 p^{1+j/3}$ for all sufficiently small $p$. In addition, we have $\mathbb{P}[D_v > 100^d] \leq p_0 \leq |B(0,L_0)| \times (100^d + 1) \times C p^{(100^d+1)/2}$. Therefore, for all sufficiently small $p$ and $1 \leq j \leq 100^d$, we have
    \response{
    \begin{equation*}
        \mathbb{P}[D_v \geq j] \leq \sum_{k=j}^{100^d} \mathbb{P}[D_v = k] + \mathbb{P}[D_v > 100^d] \leq p^{1+j/3}.
    \end{equation*}
    }
    This proves the claim.
    
    Let $(x_1,\ldots, x_{(2L_0+1)^d})$ be an arbitrary ordering of the vertices in $B(0,L_0)$. We say that an open cluster in the final configuration of $B(0,L_0)$ is generated from $x_i$ if it contains $x_i$ but contains none of the vertices in $x_1,\ldots, x_{i-1}$. We observe that these open clusters are negatively correlated. In particular, conditioned on the open clusters generated from $x_1, \ldots, x_{i-1}$ for any $i \geq 1$, and on the event that $B(0,L_0)$ is good, we claim that the open cluster generated from $x_i$ (denoting by $ O' $ a sample from this conditional law) is stochastically dominated by the law of $\widetilde O_v$. By definition, the open cluster generated from $x_i$ is not allowed to use the open vertices in the open clusters generated from $x_1, \ldots, x_{i-1}$. In addition, the event that $B(0,L_0)$ is good is decreasing with respect to the initial configuration. Therefore, the law of $O'$ is stochastically dominated by the law of $O_v$. Since $O'$ has $|\cdot|_\infty$-diameter at most $100^d$, we have $ O'=O'\cap B(x_i, 100^d)\cap B(0,L_0) $, which is dominated by $ O_v\cap B(x_i, 100^d)\cap B(0,L_0) \subset \widetilde O_v$. Thus, the law of $O'$ can be further dominated by the law of $\widetilde O_v$. This proves the claim. Combining this with the above stochastic dominance argument, we conclude the proof of the lemma. \qedhere
\end{proof}
    
\subsection{Inductive proof of Lemma~\ref{lem:bad-probability}}
\label{subsec:4.2}

In this subsection, we prove Lemma~\ref{lem:bad-probability} assuming Lemma~\ref{lem:control-growth} whose proof is postponed to Section~\ref{subsec:lem4.10}. We also complete the proof of Theorem~\ref{thm:1.4}.

For any integer $n \geq 0$, we call a finite subset of $\mathbb{Z}^d$ an $L_n$-domain if it can be written as the union of $L_n$-boxes. Next, we present the following deterministic lemma: for an $L_n$-domain $A$, if for each $L_n$-box $ B $ in this $L_n$-domain, none of the open clusters in the final configuration of $B$ has large size \response{(say, $|\cdot|_{\infty}$-diameter at most $\frac{1}{10}L_n$)}, then the final configuration of this $L_n$-domain will not have large open clusters either. Moreover, under this condition, the final configuration of this $L_n$-domain can be obtained by taking a ``careful concatenation'' of the final configuration of each $L_n$-box contained in it (here a careful concatenation means the following: for each vertex $ x $ in the domain, its state in the final configuration of $A$ is the same as its state in the final configuration of any $L_n$-box $B$ as long as $x$ has $|\cdot|_{\infty}$-distance at least $ L_n/4 $ from $ A\setminus B $).
Note that the aforementioned condition is essential since otherwise the open clusters obtained as the final configuration of an $ L_n $-box may continue to evolve when considering the evolution in the $L_n$-domain because of new open vertices produced by the evolution in the intersecting $L_n$-boxes.

\begin{lemma}\label{lem:geometric-growth}
    Let $n \geq 0$ be an integer and let $A$ be an $L_n$-domain. Suppose that the open clusters in the final configurations of all the $L_n$-boxes contained in $A$ have $|\cdot|_\infty$-diameters at most $\frac{1}{10} L_n$. Then, we have
    \begin{enumerate}[(i)]
        \item all open clusters in the final configuration of $A$ have $|\cdot|_\infty$-diameters at most $\frac{1}{10} L_n$;\label{lem4.7-claim-1}
        \item for all $x \in A$ and each $L_n$-box $B(z, L_n)$ contained in $A$ such that $x$ has $|\cdot|_\infty$-distance at least $\frac{1}{4} L_n$ from $A \setminus B(z, L_n)$, the open cluster in the final configuration of $A$ containing $x$ is the same as the open cluster in the final configuration of $B(z,L_n)$ containing $x$. \label{lem4.7-claim-2} 
    \end{enumerate}
\end{lemma}

\begin{proof}
    We will prove Claim~\eqref{lem4.7-claim-2}, and Claim~\eqref{lem4.7-claim-1} follows as a consequence of Claim~\eqref{lem4.7-claim-2} and the assumption. To this end, we consider the evolution as in~\eqref{evolve-1}. Specifically, we perform the evolution by checking the status of each vertex in an arbitrary but prefixed order. For any pair $(x, B(z,L_n))$ satisfying the condition in Claim~\eqref{lem4.7-claim-2}, by monotonicity, the open cluster $\mathcal{O}_x$ in the final configuration of $A$ containing $x$ must contain the open cluster $\mathcal{O}_x'$ in the final configuration of $B(z,L_n)$ containing $x$. If Claim~\eqref{lem4.7-claim-2} does not hold, let us consider the first time during the evolution of $A$ when the open cluster containing $x$ (which might be only a subset of $\mathcal{O}_x$) is not contained in $\mathcal{O}_x'$ for some pair $(x, B(z,L_n))$.
    (The time here does not refer to the time in the evolution~\eqref{evolve-1} where multiple changes can occur in a single time step. Instead, it refers to a moment when an empty vertex becomes open. In addition, we are not necessarily checking $x$ at this time.)
    By our assumption on ``first'', before this time, all open clusters in the evolution of $A$ have $|\cdot|_\infty$-diameters at most $\frac{1}{10}L_n$. Thus, we see that this open cluster containing $x$ has $|\cdot|_\infty$-diameter at most $2 \times \frac{1}{10} L_n + 2 \leq \frac{1}{4} L_n$, which should be contained in $B(z,L_n)$. Moreover, recalling~\eqref{evolve-1}, this open cluster also appears in the final configuration of $B(z,L_n)$, which leads to a contradiction. Therefore, we obtain Claim~\eqref{lem4.7-claim-2}.
\end{proof}

Using the above lemma, we prove a deterministic bound on the diameters of open clusters in the final configuration of a good $L_n$-box.

\begin{lemma}\label{lem:bound-diameter-good}
    The following holds for all $K > 10^{10^{10d}}$ and $p < 10^{-10^{10d}}$. Suppose that $B(0, L_n)$ is a good $L_n$-box for an integer $n \geq 0$. Then all the open clusters in the final configuration of $B(0, L_n)$ have $|\cdot|_\infty$-diameters at most $2 \sqrt{K} L_{n-1}$ if $n \geq 1$, and at most $100^d$ if $n = 0$.
\end{lemma}
\begin{proof}
    We prove this lemma by induction. The case of $n =0$ follows from Definition~\ref{def:good-box}. Suppose that the statement holds for $ n=k $ and now we consider the case $n = k+1$. First, we deal with the case where all the $L_{n-1}$-boxes contained in $B(0, L_n)$ are good. By the induction hypothesis, the final configurations of all these $L_{n-1}$-boxes do not have any open cluster with $|\cdot|_\infty$-diameter larger than $2 \sqrt{K} L_{n-2}$ if $n \geq 2$, or larger than $100^d$ if $n = 1$, both of which are smaller than $\frac{1}{10} L_{n-1}$. Therefore, by \response{applying the first claim of Lemma~\ref{lem:geometric-growth} with $n-1$,  we see that all open clusters in the final configuration of $B(0,L_n)$ have $|\cdot|_\infty$-diameters at most $\frac{1}{10} L_{n-1} \leq 2 \sqrt{K} L_{n-1}$}. This yields the argument for $n = k+1$.
    

    Next, we deal with the case where there are bad $L_{n-1}$-boxes contained in $B(0,L_n)$, but all of them are contained in $B(y, 3L_{n-1})$ for some $y \in L_{n-1} \mathbb{Z}^d$. As in the previous case, by applying Lemma~\ref{lem:geometric-growth} with $A = B(0,L_n) \setminus B(y, 3L_{n-1})$, we derive that all open clusters in the final configuration of $A$ have $|\cdot|_\infty$-diameters at most $2 \sqrt{K} L_{n-2}$ if $n \geq 2$, and at most $100^d$ if $n = 1$. We next need to control the influence from the box $ B(y,3L_{n-1}) $. \response{To this end, we first perform the evolution of $A$ and then check the evolution in $B(0,L_n)$ induced by the initially open vertices in $B(y, 3L_{n-1})$. Then, the final configuration is the same as the evolution of $B(0,L_n)$. This implies that the open clusters in the final configuration of $B(0,L_n)$ are contained in the union of the open clusters in the final configuration of $A$ and those open clusters in the final configuration of $B(0,L_n)$ that intersect $B(y, 3L_{n-1})$.} By Definition~\ref{def:good-box}, we know that the open clusters in the final configuration of $B(0,L_n)$ that intersect $B(y, 3L_{n-1})$ are contained in $B(y, \sqrt{K} L_{n-1} )$, and thus, have $|\cdot|_\infty$-diameters at most $2 \sqrt{K} L_{n-1}$. This concludes the claim for $n = k+1$. By an induction argument, we obtain the lemma. \qedhere
\end{proof}

In the following lemma, we prove an analog of the stochastic dominance argument in Lemma~\ref{lem:n=0} for any $n \geq 1$. Specifically, for an $L_n$-domain $A$, conditioned on all the $L_n$-boxes contained in $A$ being good, we will stochastically dominate the final configuration of $A$ by a random environment \response{defined as follows.
\begin{definition}\label{def:Qn}
    Let $n \geq 0$ be an integer, and let $A$ be an $L_n$-domain. Define the law $\mathbb{Q}_n$ on open cubes and closed vertices in $A$ as follows. For each $0 \leq k \leq n-1$ and $z \in L_k \mathbb{Z}^d \cap A$, we sample an open cube $B(z, 3 \sqrt{K} L_k) \cap A$ independently with probability $p_k$ defined in~\eqref{eq:bound-bad-probability}. For $0 \leq j \leq 100^d$, we also sample an open cube $B(z, j) \cap A$ independently with probability $3^d p^{1 + j/3}$. Moreover, we sample closed vertices independently with probability $q$ for those vertices not covered by any open cube. Let $\mathbb{Q}_n$ denote the law of these open cubes and closed vertices.
\end{definition}
}

\begin{lemma}\label{lem:domination}
    Let $n \geq 0$ be an integer, and let $A$ be an $L_n$-domain. Then, conditioned on all the $L_n$-boxes contained in $A$ being good, the final configuration of $A$ can be stochastically dominated by $\mathbb{Q}_n$. Namely, there exists a coupling such that the open clusters in $A$ are contained in the open clusters in $\mathbb{Q}_n$ and the closed vertices in $A$ contain the closed vertices in $\mathbb{Q}_n$.
\end{lemma}

\begin{proof}
    We will use an induction argument on $n$ to prove that for any $L_n$-domain $A$ and any decreasing event $\mathcal{D}$ with respect to the initial configuration (possibly in a domain larger than $A$), conditioned on all the $L_n$-boxes contained in $A$ being good and the event $\mathcal{D}$,
    \begin{equation}\label{eq:lem4.8-claim1}
        \mbox{the final configuration of $A$ can be stochastically dominated by $\mathbb{Q}_n$.}
    \end{equation}
    Note that this result is stronger than that in Lemma~\ref{lem:domination} because, for two decreasing events $E_1$ and $E_2$, the law $\mathbb{P}[\cdot|E_1]$ does not necessarily dominate $\mathbb{P}[\cdot|E_1, E_2].$

    The case of $n = 0$ follows from Lemma~\ref{lem:n=0}, as we now elaborate. Let $\widehat{\mathbb{P}}_0$ denote the conditional law of the initial configuration in $A$, conditioned on all the $L_0$-boxes contained in $A$ being good and on some arbitrary decreasing event with respect to the initial configuration. By Lemma~\ref{lem:n=0}, under the law $\widehat{\mathbb{P}}_0$, the open clusters in the final configuration of each $L_0$-box can be stochastically dominated by a set of independently sampled open cubes as defined in Lemma~\ref{lem:n=0}.\footnote{In Lemma~\ref{lem:n=0}, we only condition on the event that $B(0,L_0)$ is good, but it is easy to see that the proof of Lemma~\ref{lem:n=0} extends to the case where we condition on $B(0,L_0)$ being good, together with some arbitrary decreasing event with respect to the initial configuration. Note that the event that all the $L_0$-boxes contained in $A$ are good is decreasing with respect to the initial configuration.}
    Therefore, by Claim~\eqref{lem4.7-claim-2} in Lemma~\ref{lem:geometric-growth} (\response{note that all $L_0$-boxes are good, so the open cluster in each of them has $|\cdot|_\infty$-diameter at most $100^d \leq \frac{1}{10} L_0$}),
    the $\widehat{\mathbb{P}}_0$-law of the open clusters in the final configuration of $A$ can be stochastically dominated by a set of independently sampled open cubes as defined in Lemma~\ref{lem:n=0} (this is consistent with our sampling in the lemma-statement, except that we multiplied by a factor of $3^d$ because each vertex is contained in at most $3^d$ $L_0$-boxes). Conditioned on the same event and the open clusters in the final configuration of $A$, the closed vertices in the remaining domain stochastically dominate a Bernoulli percolation with probability $q$. This yields Claim~\eqref{eq:lem4.8-claim1} for $n = 0$. 

    Next, we prove the case of $n \geq 1$ by induction. Suppose that Claim~\eqref{eq:lem4.8-claim1} holds for $n = k$ and for all $L_k$-domains and decreasing events. We next consider the case for $ n=k+1 $. Let $A$ be an $L_n$-domain. Let $\mathcal{A}_n$ be the event that all the $L_n$-boxes contained in $A$ are good, and let $\mathcal{D}_n$ be any fixed decreasing event with respect to the initial configuration. Let $\widehat{\mathbb{P}}_n$ denote the conditional law of the initial configuration in $A$, conditioned on $\mathcal{A}_n \cap \mathcal{D}_n$. We only need to show that the $\widehat{\mathbb{P}}_n$-law of the open clusters in the final configuration of $A$ is stochastically dominated by the open clusters in $\mathbb{Q}_n$, since conditioned on $\mathcal{A}_n \cap \mathcal{D}_n$ and the open clusters in the final configuration of $A$, the closed vertices in the remaining domain always stochastically dominate a Bernoulli percolation with probability $q$. 
    
    Let $U$ be the union of all the bad $L_{n-1}$-boxes contained in $A$. \response{We divide the open clusters in the final configuration of $A$ into two parts. First, consider the evolution of $A$, and let $O_U$ be the union of open clusters that intersect $U$. Then, consider the evolution of $A \setminus U$, and let $O_{A \setminus U}$ be the union of open clusters that do not intersect $O_U$. Then, the open clusters in the final configuration of $A$ are contained in $O_U$ and $O_{A \setminus U}$ (in fact, these open clusters are the same as the union of $ O_U$ and $O_{A\setminus U} $).} Now we dominate these open clusters in two steps.
    
    \textbf{Step 1.} We first stochastically dominate $O_U$ under the law $\widehat{\mathbb{P}}_n$. 
    By the event $\mathcal{A}_n$ and Definition~\ref{def:good-box}, the set $U$ can be covered by a collection of boxes $(B(y_i, 3L_{n-1}))_{i \geq 1}$, where each $ y_i $ corresponds to a different $ L_n $-box, and the centers $y_i$ have pairwise $|\cdot|_\infty$-distances at least $\frac{1}{3}L_n$ (since no two of them can be contained in the same $L_{n}$-box contained in $A$).
    By \response{Definition~\ref{def:good-box} and }Lemma~\ref{lem:geometric-growth}, all of the open clusters in $O_U$ have $|\cdot|_\infty$-diameters at most $2 \sqrt{K} L_{n-1}$. \response{Therefore, open clusters in $O_U$ that intersect different $B(y_i, 3L_{n-1})$ boxes are pairwise disjoint. 
    
    This allows us to consider a different exploration of $O_U$.} We examine each vertex $y \in L_{n-1} \mathbb{Z}^d \cap A$ in a prefixed order, and check whether $B(y, L_{n-1})$ is a bad box. If yes, then we place an open cube $B(y, 3 \sqrt{K} L_{n-1}) \cap A$ which can cover the open clusters in $O_U$ intersecting $B(y, L_{n-1})$ and its neighboring bad $L_{n-1}$-boxes (\response{here we enlarge $2 \sqrt{K} L_{n-1}$ to $3 \sqrt{K} L_{n-1}$ because $y$ may not belong to $\{y_i: i\geq 1\}$ but is only $3L_{n-1}$ close to some of them}). We claim that under the law $\widehat{\mathbb{P}}_n$, conditioned on the previously placed open cubes, 
    \begin{equation}\label{eq:lem4.9-claim-OU}
        \mbox{each step has probability at most $p_{n-1}$ of placing an open cube.}
    \end{equation}Assuming this claim, we can use independently sampled $3 \sqrt{K} L_{n-1} $ cubes centered at $L_{n-1} \mathbb{Z}^d \cap A$ to cover $O_U$, each with probability $p_{n-1}$. 
    
    Next, we prove Claim~\eqref{eq:lem4.9-claim-OU}. It suffices to check that the conditioned events are all decreasing with respect to the initial configuration. Recall that the events $\mathcal{A}_n$ and $\mathcal{D}_n$ in $\widehat{\mathbb{P}}_n$ are both decreasing. The conditioned events arising from the previously placed open cubes contain two parts: (1). If we do not place an open cube at $x_i$, then $B(x_i,L_{n-1})$ is good, which is a decreasing event; (2). If we place an open cube at $x_i$, then $B(x_i,L_{n-1})$ is bad, which is a priori an increasing event. However, in this case, we can further condition on the realization of $O_U$ that intersects $B(x_i, L_{n-1})$ and its neighboring bad boxes, which turns the event into a decreasing event because the open clusters outside are not allowed to use these open vertices during the evolution. Integrating over the realization yields~\eqref{eq:lem4.9-claim-OU}.
    
    \textbf{Step 2.} Having dominated $O_U$, conditioned on $O_U$ and the events $\mathcal{A}_n, \mathcal{D}_n$, we next continue to stochastically dominate \response{$O_{A \setminus U}$}. We can apply the induction hypothesis to $A \setminus U$ by noting that the conditioned event is decreasing and is contained in the event that all the $L_{n-1}$-boxes in $A \setminus U$ are good. This allows us to stochastically dominate \response{$O_{A \setminus U}$} using $\mathbb{Q}_{n-1}$, conditioned on $O_U$ and the events $\mathcal{A}_n, \mathcal{D}_n$. Combining the previous arguments, we obtain the claim for $n = k + 1$. By induction, we conclude~\eqref{eq:lem4.8-claim1}, hence proving the lemma. \qedhere
\end{proof}

Now we control the evolution triggered by a fixed open cube in the random environment introduced in the previous lemma. The proof will be postponed to Section~\ref{subsec:lem4.10}.

\begin{lemma}\label{lem:control-growth}
There exist a constant $C>0$ and a function $\mathscr{L}: \mathbb{N} \rightarrow (0,\infty)$ such that the following holds for all $K \geq C$, $p \leq \mathscr{L}(K)^{-1}$, $q \geq \mathscr{L}(K) p^d$, and all integers $n \geq 1$. Suppose that $p_k \leq p^{2d}e^{-K\cdot2^k}$ for all $k \in [0, n-1] \cap \mathbb{Z}$. Sample open cubes and closed vertices in $\mathbb{Z}^d$ according to the law $\mathbb{Q}_{n-1}$ as defined in \response{Definition~\ref{def:Qn}}. We additionally place an open cube $B(0, 3L_{n-1})$, and consider the evolution starting from it, i.e., at each time, newly produced open vertices must be neighboring to the open cluster containing $B(0, 3L_{n-1})$. Then, with probability at least $1 - e^{-K\cdot2^{n}}/(100K)^d$, the open cluster containing $B(0, 3 L_{n-1})$ in the final configuration is contained in $B(0, \sqrt{K} L_{n-1})$.
\end{lemma}

\response{We remark that the bound $1 - e^{-K\cdot2^{n}}/(100K)^d$ is not optimal but is sufficient for our purpose. Indeed, we expect $K^{2-o(1)} 2^n$ in the exponent.}

Now we complete the proof of Lemma~\ref{lem:bad-probability}.

\begin{proof}[Proof of Lemma~\ref{lem:bad-probability}]
    The case of $n = 0$ was proved in Lemma~\ref{lem:n=0}. Suppose that the result holds for $n \leq k$ and we next consider $n = k+1$.
    Recalling Definition~\ref{def:good-box}, we see that there are two scenarios for $ B(0,L_n) $ to be bad:
    \begin{enumerate}[(a)]
        \item There are bad $ L_{n-1} $-boxes that cannot be covered by $ B(y,3L_{n-1}) $ for any $y\in L_{n-1}\mathbb{Z}^d$. \label{lem4.3-sce-a}
        \item There exist bad boxes which are all covered by $ B(y,3L_{n-1}) $ for some $ y\in L_{n-1}\mathbb{Z}^d$. In addition, there exists an open cluster in the final configuration of $ B(0,L_n) $ intersecting $ B(y,3L_{n-1})$ such that this open cluster is not contained in $ B(y,\sqrt{K} L_{n-1})$.\label{lem4.3-sce-b}
    \end{enumerate}
    \response{If Scenario~\eqref{lem4.3-sce-a} occurs, then there exist two bad $ L_{n-1} $-boxes contained in $B(0,L_n)$ whose $|\cdot|_\infty$-distance is at least $3L_{n-1}$. By taking a union bound over all such pairs (there are at most $(2K - 1)^{2d}$ of them) and using independence, we see that} the probability of Scenario~\eqref{lem4.3-sce-a} is at most
    \begin{equation*}
        (2K - 1)^{2d} \times p_{n-1}^2 \leq (2K - 1)^{2d} \times (p^{2d} e^{-K \cdot 2^{n-1}})^2 \leq \frac12 p^{2d}e^{-K\cdot2^n},
    \end{equation*}
    for all sufficiently small $p$ (which may depend on $K$). Next, we consider Scenario~\eqref{lem4.3-sce-b} where all the bad $L_{n-1}$-boxes can be covered by $B(y, 3L_{n-1})$ for some $y \in L_{n-1} \mathbb{Z}^d$. We can apply Lemma~\ref{lem:domination} with the $ L_{n-1} $-domain $A := B(0, L_n) \setminus B(y, 3L_{n-1})$ and apply Lemma~\ref{lem:control-growth} to show that with probability at least $1 - e^{-K\cdot2^{n}}/(100K)^d$, the open clusters in the final configuration of $B(0, L_n)$ that intersect $B(y, 3L_{n-1})$ are contained in $B(y, \sqrt{K} L_{n-1})$ where we require $K, p, q$ to satisfy the conditions in Lemma~\ref{lem:control-growth} (note that this is consistent with the assumptions on $ K,p,q $ as in the lemma-statement).
    Also note that by the assumption of Lemma~\ref{lem:control-growth}, the preceding event is independent of the initial configuration within $ B(y,3L_{n-1}) $, where there exists a bad $ L_{n-1} $-box.
    Summing over the choices of $y$, we obtain that $$p_n \leq \frac12 p^{2d}e^{-K\cdot2^{n}} + (2K + 7)^{d} \times 5^d p_{n-1} \times e^{-K\cdot2^{n}}/(100K)^d \leq p^{2d}e^{-K\cdot2^{n}},$$
    where $(2K+7)^d$ is the number of $y \in L_{n-1} \mathbb{Z}^d$ such that $B(y, 3L_{n-1}) \cap B(0, L_n) \neq \emptyset$, and the factor $5^d$ is the number of $L_{n-1}$-boxes contained in $B(y, 3L_{n-1})$.
    This yields the claim for $n$. By induction, we conclude the lemma.
\end{proof}

Now we complete the proof of Theorem~\ref{thm:1.4}.

\begin{proof}[Proof of Theorem~\ref{thm:1.4}]
    Choose $K, p, q$ satisfying the conditions in Lemma~\ref{lem:bad-probability}. Then for all integers $n \geq 0$ we have $p_n \leq p^{2d}e^{-K\cdot2^{n}}$. Fix an integer $M \geq 0$, and consider the evolution of $\mathbb{Z}^d$ as in~\eqref{evolve-2}. By Lemma~\ref{lem:infinite-evolve}, we have
    \begin{equation}\label{eq:thm1.4-proof-0}
        \begin{aligned}
            &\quad \mathbb{P}_{p,q,\infty}[{\rm the\ open\ cluster\ containing\ }0{\rm\ has\ diameter\ at\ least\ }M] \\
            &\leq \lim_{n \rightarrow \infty}\mathbb{P}_{p,q,L_n}[\mbox{the open cluster containing 0 has diameter at least } M],
        \end{aligned}
    \end{equation}
    where in the measure $ \mathbb{P}_{p,q,L_n}$ we refer to the open cluster in the final configuration of $B(0,L_n)$.
    Next, we upper-bound the right-hand side of~\eqref{eq:thm1.4-proof-0} for any fixed $n$. The first case is that there are bad $L_{n-1}$-boxes contained in $B(0,L_n)$, and this happens with probability at most $(2K-1)^d \times p_{n-1}$. The second case is that all the $L_{n-1}$-boxes contained in $B(0,L_n)$ are good. Then we can apply Lemma~\ref{lem:domination} to show that the open clusters in the final configuration of $B(0,L_n)$ are stochastically dominated by $\mathbb{Q}_n$. We observe that under $\mathbb{Q}_n$, the cardinality of the open cluster containing 0 can be dominated by the total population of a Galton-Watson tree, whose offspring distribution is given by the total number of vertices in the open cubes sampled from $\mathbb{Q}_n$ that are neighboring to (or contain) a fixed vertex. The expectation of this offspring distribution is at most
    \begin{equation}\label{eq:density-GW}
    \sum_{j=0}^{100^d} (2j+3)^d \times 3^d p^{1+j/3} \times |B(0,j)| + \sum_{k=0}^\infty (6 \sqrt{K} + 3 )^d \times p_k \times |B(0, 3 \sqrt{K} L_k)| \leq \frac{1}{2}
    \end{equation}
    for all sufficiently small $p$. (Here, $(2j+3)^d$ counts the number of $y \in \mathbb{Z}^d$ such that $B(y,j)$ is neighboring to a fixed vertex, and $(6 \sqrt{K} + 3 )^d$ counts the number of $y \in L_k \mathbb{Z}^d$ such that $B(y, 3 \sqrt{K} L_k)$ is neighboring to a fixed vertex.) Therefore, this Galton-Watson tree is subcritical. Combining this with the first case, we obtain
    $$
    \mathbb{P}_{p,q,L_n}[\mbox{the open cluster containing 0 has diameter at least }M] \leq (2K-1)^d \times p_{n-1} + o_M(1).
    $$
    By first taking $n$ to $\infty$ and then $M$ to $\infty$, we obtain that there is no infinite open cluster a.s.-$\mathbb{P}_{p,q,\infty}$. By noting that the left hand side of~\eqref{eq:density-GW} tends to 0 as $p$ tends to 0, we can obtain $\lim_{ p \rightarrow 0}\mathbb{P}_{p, q, \infty}[\mbox{$0$ is open}] = 0$.
\end{proof}

\begin{remark}\label{rmk:polluted-bootstrap}
In Section~\ref{sec:glauber}, we will consider a variant of polluted bootstrap percolation. In this variant, a closed vertex will become open if and only if it has at least $(d+1)$ open neighbors. Theorem~\ref{thm:1.4} also holds for this variant (albeit with a larger constant $C$). In fact, all the arguments in Sections~\ref{subsec:n=0} and \ref{subsec:4.2} apply directly to this variant, as we rely only on the monotonicity properties of polluted bootstrap percolation, rather than its specific setting. In Section~\ref{subsec:lem4.10}, we will prove Lemma~\ref{lem:control-growth} by constructing a contour to block the open cluster triggered by $B(0, 3L_{n-1})$ in polluted bootstrap percolation. At each corner of this surface, there is a closed vertex. As we will see in Section~\ref{subsec:lem4.10}, particularly in Definition~\ref{def:nice-box}, this surface can also block this variant of polluted bootstrap percolation, since all closed vertices at the corners have at most $d$ open neighbors and thus remain closed. Therefore, Lemma~\ref{lem:control-growth} also applies to this variant, and so does Theorem~\ref{thm:1.4}.
\end{remark}

\subsection{Proof of Lemma~\ref{lem:control-growth}}
\label{subsec:lem4.10}
In this subsection, we will write the origin of $ \mathbb{Z}^d $ as $ \mathbf{0} $ for clarity.
We will prove Lemma~\ref{lem:control-growth} by constructing a contour $ \mathbf{S}_n $ (note that for $ d\geq3 $, a contour looks more like a surface) enclosing $B(\mathbf{0}, 3 L_{n-1})$ such that the evolution starting from $B(\mathbf{0}, 3L_{n-1})$ will stop when (if not before) reaching the contour $\mathbf{S}_n$.
We will construct $ \mathbf{S}_n $ by concatenating some specific local structures called oriented shields (as defined below in~\eqref{eq:def_of_shield}). Oriented shields can be found in nice boxes (Definition~\ref{def:nice-box}), illustrated in Figure~\ref{fig:shield}, and a fixed box is nice with probability close to 1. Lemma~\ref{lem:shield-local} then states that the oriented shields can locally prevent the evolution. To ensure that the concatenation of oriented shields stops the evolution, we have to further restrict the positions of the oriented shields (or equivalently, nice boxes) to a sublattice of $ \mathbb{Z}^d $ (see Figure~\ref{fig:oriented_sublattice}) defined by the linear map in~\eqref{eq:def-pi_k}. Definition~\ref{def:oriented_surface} then gives the specific concatenation of nice boxes (also referred to as nice oriented surfaces) that we need in the construction of $ \mathbf{S}_n $. In Lemmas~\ref{lem:oriented_surface_disconnect_origin_from_infinity} and~\ref{lem:oriented_surface_prevents_growth}, we will show that combining the nice oriented surfaces in all $2^d$ orientations can indeed form a contour $\mathbf{S}_n$ to stop the evolution. It remains to prove that a nice oriented surface exists with very high probability. To this end, we apply a coarse-graining argument and investigate the positions of nice boxes. Definition~\ref{def:overwhelming} and Lemma~\ref{lem:overwhelming_dominance} provide a more manageable description of the distribution of nice boxes, and we eventually finish the proof of Lemma~\ref{lem:control-growth} by lower-bounding the probability of the existence of nice oriented surfaces through a duality argument.
\response{Constants $ \{a_j\}_{-1\leq j\leq n-2} $ and $ \{\lambda_{j}\}_{-1\leq j\leq n-2} $ (where $n\geq1$) in Definition~\ref{def:overwhelming} and the proof of Lemma~\ref{lem:control-growth} are carefully chosen to satisfy the stochastic domination in Lemma~\ref{lem:overwhelming_dominance} and to obtain the proper estimates~\eqref{eq:count-path_-2},~\eqref{eq:count-path-0},~\eqref{eq:count_path_j=-1} and~\eqref{eq:count_path_j>=0} for the duality argument.}

We first introduce some notations that will be frequently used in this subsection. Consider a $d$-dimensional discrete hypercube $ \{-1,1\}^d $. In this subsection, we will call its elements as orientations.
For $1 \leq i \leq d$, let $\boldsymbol{e}_i$ be the $i$-th standard basis vector. For $x \in \mathbb{Z}^d$, we write $x_i$ for its $i$-th coordinate for all $1 \leq i \leq d$.
For an index set $ I\subset\{1,\dots,d\} $, we will write $ \{1,\dots,d\}\setminus I $ as $ I^c $ for brevity.
For any $ \boldsymbol{k}\in\{-1,1\}^d $ and any $ x,y\in\mathbb{Z}^d $, we write $ x\preceq_{\boldsymbol{k}}y $ (resp.\ $ x\prec_{\boldsymbol{k}}y $) if and only if
\begin{equation*}
    (x-y)_i\cdot\boldsymbol{k}_i\leq0 \quad\mbox{(resp. } (x-y)_i\cdot\boldsymbol{k}_i<0\mbox{)}\quad\mbox{for all }1\leq i\leq d,
\end{equation*}
where $ (x-y)_i$ and $ \boldsymbol{k}_i$ are the $i$-th coordinates of $(x-y)$ and $\boldsymbol{k}$, respectively. In addition, for $ x\in\mathbb{Z}^d $ and $ A\subset\mathbb{Z}^d $, we write $ x\preceq_{\boldsymbol{k}}A $ (resp.\ $x\prec_{\boldsymbol{k}}A $) if and only if $ x\preceq_{\boldsymbol{k}}y $ (resp.\ $x\prec_{\boldsymbol{k}} y$) for all $ y\in A $. Recall the constants $K$ and $p$ defined before~\eqref{eq:def-Ln}. Let the scale
\begin{equation}\label{eq:def-L}
    L=\Big\lfloor\frac{1}{pK}\Big\rfloor .
\end{equation}
For a vertex $ y\in\mathbb{Z}^d $, define
\begin{equation*}
    \mathsf{R}(y):=y+\bigcup_{1\leq i\leq d}\big\{ z\in\mathbb{Z}^d:|z_i|\leq10^{10d}L \mbox{, and $|z_j| \leq 100d$ for all $j \in \{i \}^c$} \big\}.
\end{equation*}

Recall the setting of Lemma~\ref{lem:control-growth}. Fix an integer $n \geq 1$ and assume that $p_k \leq p^{2d}e^{-K\cdot2^k}$ for all $k \in [0,n-1] \cap \mathbb{Z}$. We sample open cubes and closed vertices on $\mathbb{Z}^d$ according to the law $\mathbb{Q}_{n-1}$, which depends on $K, p, q$.
We refer to this configuration as the initial configuration, and the open vertices contained in the open cubes (sampled from $ \mathbb{Q}_{n-1} $) are called initially open vertices.
Then, we additionally place an open cube $B(\mathbf{0},3L_{n-1})$, and consider the evolution starting from it, namely, at any time the newly produced open vertices must be neighboring to the open cluster containing $B(\mathbf{0},3L_{n-1})$.
Our goal is to control the open cluster containing $B(\mathbf{0},3L_{n-1})$ in the final configuration by constructing a contour $ \mathbf{S}_n $.

We will construct $ \mathbf{S}_n $ by concatenating local structures, as indicated by the following definitions of nice boxes and $ \boldsymbol{k} $-oriented shields for $ \boldsymbol{k}\in\{-1,1\}^d $.
\begin{definition}
    \label{def:nice-box}
    Fix an initial configuration. For $x \in \mathbb{Z}^d$, we say a box $ B(x,L) $ is marked if we mark a point $ y\in B(x,L) $ as a distinct point. In addition, we say $ (B(x,L),y) $ is a nice marked box, if the following hold: \response{\rm (1)} $y$ is closed; {\rm (2)} every $ z\in\mathsf{R}(y) $ is not initially open; {\rm (3)} for any $ z\in B(x,10^{10d}L) $, there are at most $ 4d $ initially open vertices in $ B(z,10d) $.
    We say $ B(x,L) $ is nice if $ (B(x,L),y) $ is a nice marked box for some $ y\in B(x,L) $.
    If $ B(x,L) $ is nice, then we choose an arbitrary $ y\in B(x,L) $ such that $ (B(x,L),y) $ is nice and denote it by $ y(x) $. Note that the niceness of a box $ B(x,L) $ is not measurable with respect to the initial configuration in $ B(x,L) $, but instead with respect to the initial configuration in $ B(x,20^{10d}L) $.
\end{definition}
For an orientation $ \boldsymbol{k}\in\{-1,1\}^d $, the $ \boldsymbol{k} $-oriented shield $ {\rm Sh}(x, \boldsymbol{k}) $ for a marked box $ (B(x,L),y) $ is defined as follows\footnote{\response{Note that the shield $ {\rm Sh}(x,\boldsymbol{k}) $ is unique with respect to the marked point $ y\in B(x,L) $ and the orientation $ \boldsymbol{k}\in\{-1,1\}^d $. However we only care about the location of the shields up to a coarse-grained level, so we label the shield of a marked box $ (B(x, L), y) $ without the marked point $y$ occasionally. In fact, what we are going to deal with in the rest of this subsection are the shields for the nice marked boxes, as defined below.}}:
\begin{equation}
    \label{eq:def_of_shield}
    \begin{aligned}
        &\quad\ {\rm Sh}(x, \boldsymbol{k})={\rm Sh}(x, \boldsymbol{k};y)\\&:=\bigcup_{1 \leq i \leq d}\Big( y+\big\{ z\in\mathbb{Z}^d:z_i=0, \mbox{ and }  0\leq \boldsymbol{k}_j\cdot z_j \leq 10^{10d}L\ \forall j\in \{i\}^c\big\} \Big).
    \end{aligned}
\end{equation}
We will call $ y $ the corner of $ {\rm Sh}(x,\boldsymbol{k}) $.
Note that the above definition naturally applies to a nice box $ B(x,L) $, where $ y=y(x) $ is chosen as in Definition~\ref{def:nice-box}. Namely, for a nice box $B(x,L)$, we can first choose $y(x) \in B(x,L)$ such that $(B(x,L), y(x))$ is a nice marked box as defined in Definition~\ref{def:nice-box}, and then define ${\rm Sh}(x,\boldsymbol{k})$ with $y(x)$ as the corner. See Figure~\ref{fig:shield} for an illustration.
\begin{figure}[h]
    \centering
    \includegraphics[width=0.5\textwidth]{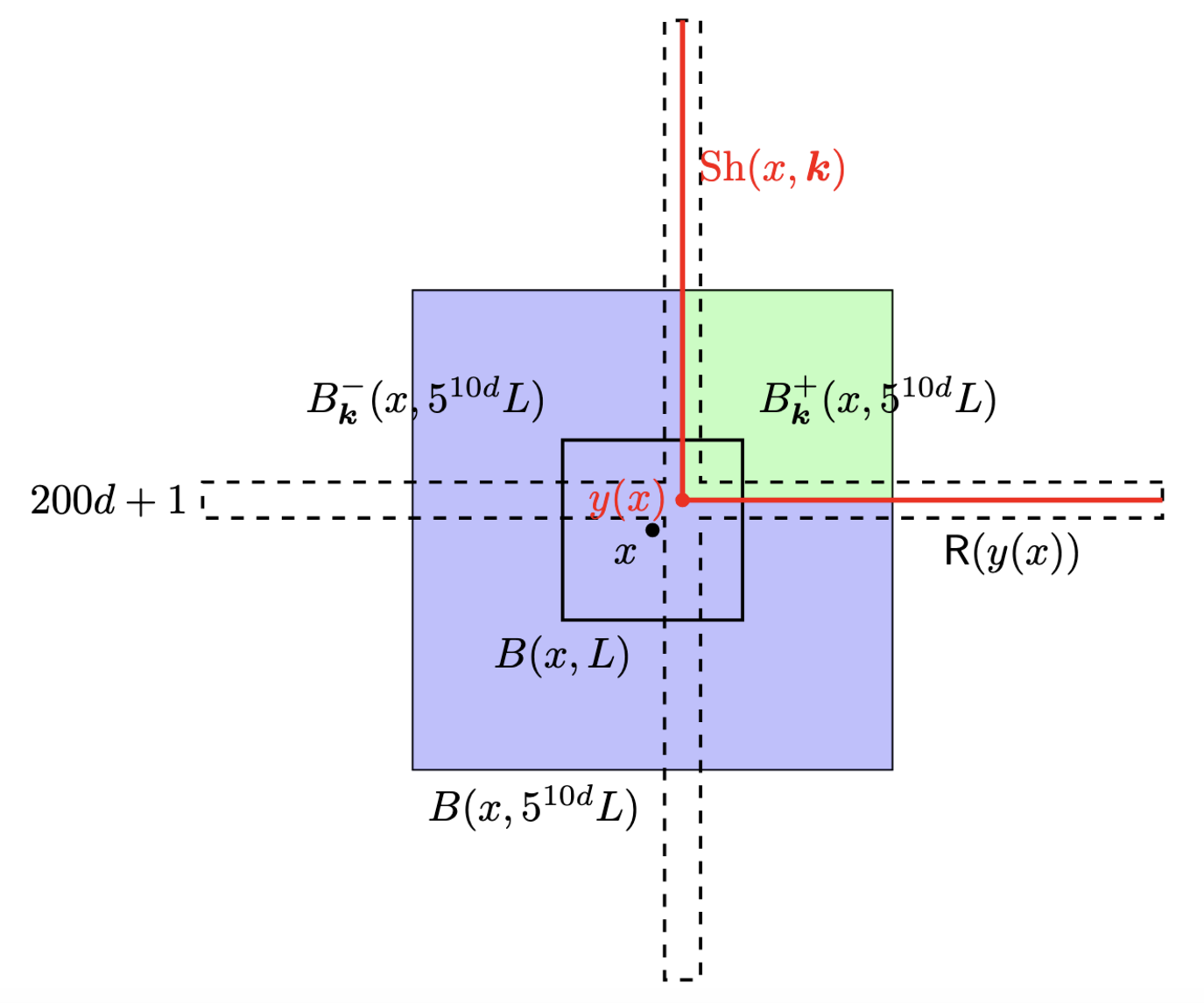}
    \caption{An illustration for $ d=2 $ and $ \boldsymbol{k}=(1,1) $. For the (nice) marked box $(B(x,L),y(x))$, the chosen vertex $y(x)$ is colored in red. The oriented shield $ {\rm Sh}(x,\boldsymbol{k}) $ is also colored in red with $y(x)$ as the corner. The domain $ B^+_{\boldsymbol{k}}(x,5^{10d}L) $ is colored in green, and the domain $ B^-_{\boldsymbol{k}}(x,5^{10d}L) $ is colored in light purple. The boundary of $ \mathsf{R}(y(x)) $ is shown with dashed black lines. Although in the two-dimensional case the set $ \mathsf{R}(y(x)) $ contains $ {\rm Sh}(x,\boldsymbol{k}) $, this is not true in three dimensions or higher.}
    \label{fig:shield}
\end{figure}

The following Lemma shows that the shield of a nice box can locally prevent the evolution. Observe that for any nice box $ B(x,L) $, the shield $ {\rm Sh}(x,\boldsymbol{k}) $ separates $ B(x,5^{10d}L) $ into two disjoint parts:
\begin{align*}
    & B^{+}_{\boldsymbol{k}}(x,5^{10d}L):=\{z\in B(x,5^{10d}L):y(x)\preceq_{\boldsymbol{k}}z\},\quad\mbox{and}\\
    & B^{-}_{\boldsymbol{k}}(x,5^{10d}L):=\{z\in B(x,5^{10d}L):y(x)\not\preceq_{\boldsymbol{k}}z\}.
\end{align*}
\begin{lemma}
    \label{lem:shield-local}
    Suppose that $ B(x,L) $ is nice for an initial configuration.
    We consider the polluted bootstrap percolation restricted to $ B(x,5^{10d}L)$ initiated from the following configuration: each vertex in $ B^{+}_{\boldsymbol{k}}(x,5^{10d}L) $ has the same state as in the initial configuration (as defined before Definition~\ref{def:nice-box}), and each vertex in $ B^{-}_{\boldsymbol{k}}(x,5^{10d}L) $ is open. Then in the final configuration of $ B(x,5^{10d}L) $, the open cluster containing $ B^{-}_{\boldsymbol{k}}(x,5^{10d}L) $ is within $ |\cdot|_{\infty} $-distance $ 10d $ from $ B^{-}_{\boldsymbol{k}}(x,5^{10d}L) $.
\end{lemma}
\begin{proof}
    We see from Definition~\ref{def:nice-box} that the initial configuration in the lemma-statement satisfies the following: \response{(1)} the corner $y(x)$ is closed; (2) each $ z\in \mathsf{R}(y(x))\cap B^{+}_{\boldsymbol{k}}(x,5^{10d}L) $ is not initially open; (3) for any $ w\in B^{+}_{\boldsymbol{k}}(x,5^{10d}L) $, there are at most $ 4d $ initially open vertices in $ B(w,10d)\cap B^{+}_{\boldsymbol{k}}(x,5^{10d}L)$.

    We prove by contradiction. Consider the open vertices in $ B^{+}_{\boldsymbol{k}}(x,5^{10d}L) $ and the open clusters formed by these vertices. We say that a subset is an open cluster in $ B^{+}_{\boldsymbol{k}}(x,5^{10d}L) $ if it is connected and only consists of the open vertices in $ B^{+}_{\boldsymbol{k}}(x,5^{10d}L) $.
    Suppose that the lemma does not hold, then there exists an open cluster in $B^{+}_{\boldsymbol{k}}(x,5^{10d}L)$ with $|\cdot|_\infty$-diameter at least $10d$. We consider the first time during the evolution that there exists an open cluster $\mathcal{O}$ in $B^{+}_{\boldsymbol{k}}(x,5^{10d}L)$ with $|\cdot|_\infty$-diameter at least $10d$. At this time, all the open clusters in $B^{+}_{\boldsymbol{k}}(x,5^{10d}L)$, including $\mathcal{O}$, have $|\cdot|_\infty$-diameters at most $20d$. We claim that $\mathcal{O}$ will also appear if we consider the polluted bootstrap percolation with threshold 2 restricted to $B^{+}_{\boldsymbol{k}}(x,5^{10d}L)$ with the same initial configuration. In order to prove this claim, we first argue that $\mathcal{O}$ should not intersect the edges of $B^{+}_{\boldsymbol{k}}(x,5^{10d}L)$ in $B_{\boldsymbol{k}}(x,5^{10d}L)$, that is,
    \begin{equation*}
        (y+\cup_{1\leq i\leq d}\{ z\in\mathbb{Z}^d: 0 \leq \boldsymbol{k}_i \cdot z_i \leq 10^{10d}L , z_j = 0 \mbox{ for all } j \in \{i \}^c \}) \cap B_{\boldsymbol{k}}(x,5^{10d}L).
    \end{equation*}
    Because otherwise, we have $\mathcal{O} \subset \mathsf{R}(y(x))$ (since the $|\cdot|_{\infty}$-distance between the edges and $ B^{+}_{\boldsymbol{k}}(x,5^{10d}L)\setminus\mathsf{R}(y(x)) $ is at least $ 100d $), but (2) implies that there is no initially open vertex contained in $\mathcal{O}$ which is impossible. Therefore, each vertex in $\mathcal{O}$ is neighboring to at most $d-2$ vertices in $B^{-}_{\boldsymbol{k}}(x,5^{10d}L)$. This proves the claim. By Lemma~\ref{lem:initial-open-number-lower-bound}, we obtain that there are at least $\frac{10d+2}{2}$ initially open vertices contained in $\mathcal{O}$, and thus there exists $y \in B^{+}_{\boldsymbol{k}}(x,5^{10d}L)$ such that $B(y, 10d) \supset \mathcal{O}$ contains at least $\frac{10d+2}{2}$ initially open vertices, which contradicts (3). This concludes the proof of the lemma. 
    \qedhere
\end{proof}

Despite that a shield can locally prevent the evolution in the sense of Lemma~\ref{lem:shield-local}, there is still some subtlety when concatenating different shields. One concern is that for two nice boxes $ B(x,L) $ and $ B(x^{\prime},L) $ and for an orientation $ \boldsymbol{k} $, shields $ {\rm Sh}(x,\boldsymbol{k}) $ and $ {\rm Sh}(x^{\prime},\boldsymbol{k}) $ do not intersect each other if $ y(x)\prec_{\boldsymbol{k}}y(x^{\prime}) $. See Figure~\ref{fig:concatenating-subtlety} for an illustration with $ d=2 $ and $ \boldsymbol{k}=(1,1) $.
\begin{figure}[h]
    \centering
    \includegraphics[width=0.5\textwidth]{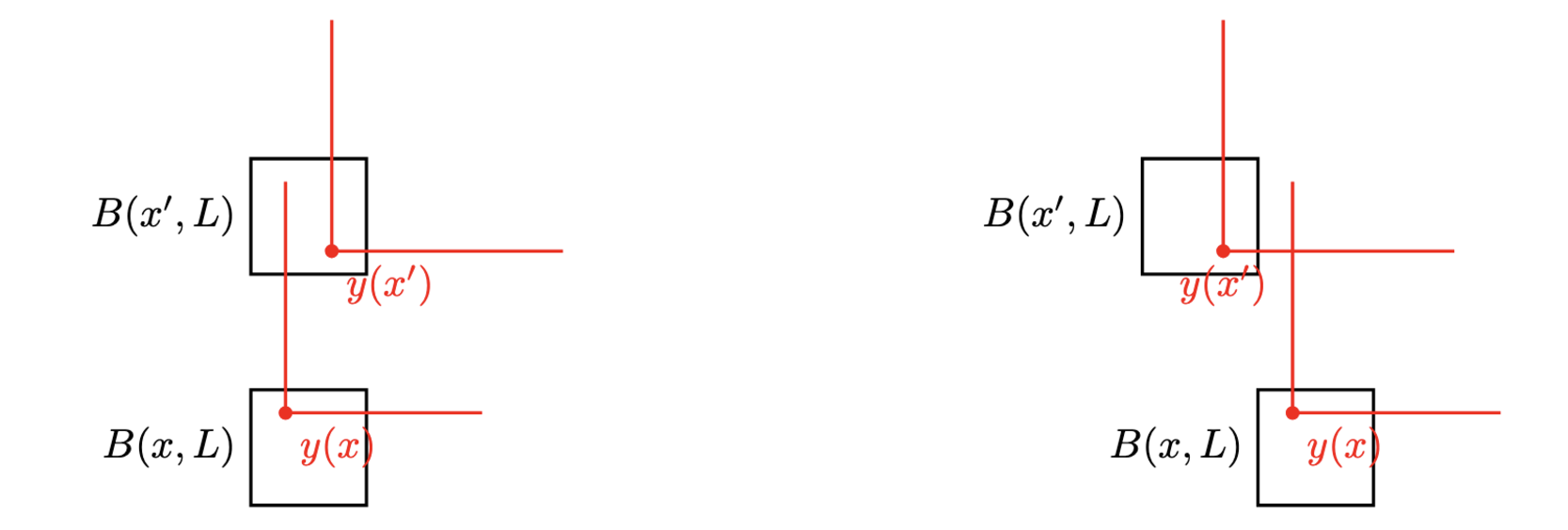}
    \caption{Some subtlety in concatenating the shields. $ B(x,L) $ and $ B(x^{\prime},L) $ are nice boxes, and oriented shields are colored in red. \emph{Left:} For two axis-aligned nice boxes, their shields might not intersect and therefore unable to be concatenated. \emph{Right:} If we pick the positions of nice boxes carefully, their shields will intersect despite of the exact location of $ y(x) $ (resp., $ y(x^{\prime}) $) in $ B(x,L) $ (resp., $ B(x^{\prime},L) $).}
    \label{fig:concatenating-subtlety}
\end{figure}

To this end, we restrict the positions of nice boxes to a carefully selected sublattice. For an orientation $ \boldsymbol{k}\in\{-1,1\}^d $, let $ \boldsymbol{l}_i(\boldsymbol{k}):=9d\boldsymbol{k}_i\boldsymbol{e}_i-3\boldsymbol{k}\in\mathbb{Z}^d $ and define a linear endomorphism $ \pi_{\boldsymbol{k}}:\mathbb{Z}^d\to\mathbb{Z}^d $ by
\begin{equation}
    \label{eq:def-pi_k}
    \pi_{\boldsymbol{k}}:(x_1,\dots,x_d)\mapsto x_1\boldsymbol{l}_1(\boldsymbol{k})+\dots+x_d\boldsymbol{l}_d(\boldsymbol{k})=\sum^d_{i=1}\boldsymbol{e}_i\cdot\boldsymbol{k}_i(9dx_i-3\sum^d_{j=1}x_j)
\end{equation}
with the inverse map
\begin{equation}
    \label{eq:pi_k_inverse}
    \pi^{-1}_{\boldsymbol{k}}:y\mapsto \sum^d_{i=1}\boldsymbol{e}_i\cdot\frac{1}{9d}\boldsymbol{k}_i\big( y_i + \frac{1}{2d}\sum^d_{j=1}y_j \big).
\end{equation}
Note that
\begin{equation}
    \label{eq:L_1-of-pi^-1}
    |\pi^{-1}_{\boldsymbol{k}}(y)|_1=\frac{1}{9d}\sum^d_{i=1}|y_i+\frac{1}{2d}\sum^d_{j=1}y_j|\geq\frac{1}{9d}\sum^d_{i=1}\big(|y_i|-\frac{1}{2d}|y|_1\big)=\frac{1}{18d}|y|_1
\end{equation}
for all $ y\in\pi_{\boldsymbol{k}}(\mathbb{Z}^d) $.
Let $ \Pi_{\boldsymbol{k}}:=\pi_{\boldsymbol{k}}(L\mathbb{Z}^d) $ denote a sublattice of $ \mathbb{Z}^d $ where two vertices $ x,y $ are adjacent if and only if $ |\pi^{-1}_{\boldsymbol{k}}(x)-\pi^{-1}_{\boldsymbol{k}}(y)|_1=L $. Also define the metric box in $ \Pi_{\boldsymbol{k}} $ as
\begin{equation}
    \label{eq:Pi_metric_ball}
    B_{\Pi_{\boldsymbol{k}}}(x,r):= \{ x+\pi_{\boldsymbol{k}}(Ly): y\in B(\mathbf{0},r) \}
\end{equation}
for $ x\in\Pi_{\boldsymbol{k}} $ and $ r\in\mathbb{N} $. In other words, we can regard $ \Pi_{\boldsymbol{k}} $ as a metric space isomorphic to $ \mathbb{Z}^d $.
For $ x,y\in\Pi_{\boldsymbol{k}} $, we write $ x\prec^{\prime}_{\boldsymbol{k}}y $ ($ x\preceq^{\prime}_{\boldsymbol{k}}y $) if there exist positive (non-negative, respectively) integers $ b_1,\dots,b_d $ such that \response{$ y=x+L\sum^{d}_{i=1}b_i\cdot\boldsymbol{l}_i(\boldsymbol{k}) $}. One can check that
\begin{equation*}
    \mbox{$x\prec^{\prime}_{\boldsymbol{k}}y $ (resp.\ $ x\preceq^{\prime}_{\boldsymbol{k}}y $)} \quad \mbox{if and only if} \quad   \pi^{-1}_{\boldsymbol{k}}(x)\prec_{\boldsymbol{k}}\pi^{-1}_{\boldsymbol{k}}(y)\ (\mbox{resp.}\ \pi^{-1}_{\boldsymbol{k}}(x)\preceq_{\boldsymbol{k}}\pi^{-1}_{\boldsymbol{k}}(y) ).
\end{equation*}
For $ x\in\Pi_{\boldsymbol{k}} $ and $ A\subset\Pi_{\boldsymbol{k}} $, we write $ x\prec^{\prime}_{\boldsymbol{k}}A $ ($ x\preceq^{\prime}_{\boldsymbol{k}}A $) if and only if $ x\prec^{\prime}_{\boldsymbol{k}}y $ ($ x\preceq^{\prime}_{\boldsymbol{k}}y $) for all $ y\in A $.
We also treat the vertices in $ \Pi_{\boldsymbol{k}} $ as $ L $-boxes to determine whether they are nice for brevity.
By straightforward calculations, any two adjacent vertices $ x,x^{\prime}\in\Pi_{\boldsymbol{k}} $ satisfy $ x\not\preceq_{\boldsymbol{k}}x^{\prime} $, and
\begin{equation}
    \label{eq:sublattice_prec}
    \min_{1\leq i\leq d} |(x-x^{\prime})_i| \geq L\cdot\min_{1\leq i,j\leq d} |\boldsymbol{l}_i(\boldsymbol{k})\cdot\boldsymbol{e}_j| \geq 3L\quad{\rm and}\quad|x-x^{\prime}|_{\infty}\leq L\cdot\max_{1\leq i\leq d} |\boldsymbol{l}_i(\boldsymbol{k})|_{\infty} \leq12dL.
\end{equation}

We are now ready to define the concatenation of the oriented shields. We will first define the oriented surfaces as specific collections of nice boxes for concatenation in Definition~\ref{def:oriented_surface}, and define $ \mathbf{S}_n $ in~\eqref{eq:def-S_n} to be the union of the oriented shields from the oriented surfaces. Then, in Lemma~\ref{lem:oriented_surface_disconnect_origin_from_infinity} we prove that $ \mathbf{S}_n $ disconnects $ B(\mathbf{0}, 3L_{n-1}) $ from $ \partial_{i} B(\mathbf{0}, \sqrt{K}L_{n-1}) $, and in Lemma~\ref{lem:oriented_surface_prevents_growth} we show that $ \mathbf{S}_n $  prevents the evolution starting from $ B(\mathbf{0}, 3L_{n-1}) $.

In order to describe the aforementioned disconnecting property, we now introduce some auxiliary subsets of $ \Pi_{\boldsymbol{k}} $ for an orientation $ \boldsymbol{k} $, such that we can reduce the disconnecting property in $ \mathbb{Z}^d $ to the disconnecting property in $ \Pi_{\boldsymbol{k}} $. Define
\begin{equation}
    \label{eq:def-A^0_n(k)}
    A^0_n(\boldsymbol{k}):=\big\{x\in\mathbb{Z}^d:\sqrt{K}L_{n-1}-6d^2L\leq|x|_1\leq \sqrt{K}L_{n-1}\mbox{ and } 0\preceq_{\boldsymbol{k}}x \big\}
\end{equation}
and further define
\begin{equation}
    \label{eq:def-A^pm_n(k)}
    \begin{aligned}
        &A_n^+(\boldsymbol{k}):=\Pi_{\boldsymbol{k}}\cap\big(A_n^0(\boldsymbol{k})- \frac{\sqrt{K}}{2d} L_{n-1}\boldsymbol{k}\big),\quad A_n^-(\boldsymbol{k})=\Pi_{\boldsymbol{k}}\cap\big(A_n^0(\boldsymbol{k})-\big( \frac{\sqrt{K}}{d} -10d\big)L_{n-1} \boldsymbol{k}\big),\\
        &A_n(\boldsymbol{k})=\big\{x\in\Pi_{\boldsymbol{k}}:\exists\,a,b\geq0,s.t.\ x-a\boldsymbol{k}\in A_n^-(\boldsymbol{k}){\rm\ and\ }x+b\boldsymbol{k}\in A_n^+(\boldsymbol{k})\big\}.
    \end{aligned}
\end{equation}
See Figure~\ref{fig:oriented_sublattice} for an illustration (in the case of $ d=2 $). A remark is that $ A^0_n(\boldsymbol{k}) $ (also $ A^+_n(\boldsymbol{k}) $ and $ A^-_n(\boldsymbol{k}) $) is not necessarily a hyper-plane (with dimension $d-1$) in $ \mathbb{Z}^d $. However, to make the pictures more visual friendly, we draw them as line segments in Figures~\ref{fig:oriented_sublattice},~\ref{fig:lem4.15-step1} and~\ref{fig:dual_path}.
\begin{figure}[h]
    \centering
    \includegraphics[width=0.5\textwidth]{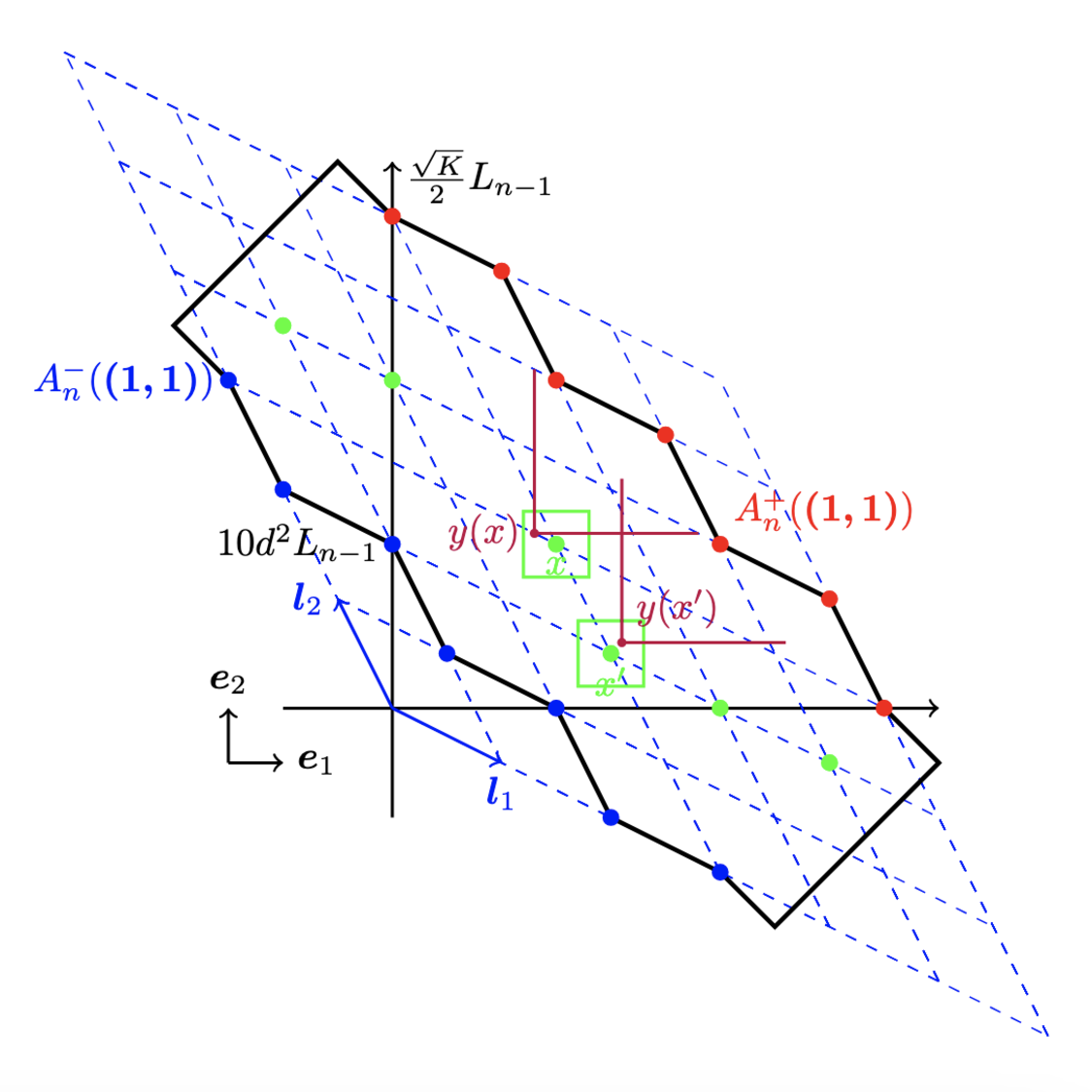}
    \caption{An illustration for sublattice $ \Pi_{(1,1)} $ (the dashed blue lattice). Vertices in $ A_n^{+}((1,1)) $ are colored in red and vertices in $ A_n^{-}((1,1)) $ are colored in blue. The boundary of $ A_n((1,1)) $ is colored in black. The oriented surface $ S_n((1,1)) $ is colored in green. Two nice boxes in $ S_n((1,1)) $ and their shields are drawn as well and colored in green and purple respectively.}
    \label{fig:oriented_sublattice}
\end{figure}
By~\eqref{eq:def-A^0_n(k)} we know that $ |x|_{\infty}\leq\sqrt{K} L_{n-1} $ for all $ x\in A^0_n(\boldsymbol{k}) $. Also note that for all $ z\in A^+_n(\boldsymbol{k}) $ there exists a vertex $ x(z)\in A^0_n(\boldsymbol{k}) $ (hence $ x(z)_i\boldsymbol{k}_i\geq0 $ for all $ 1\leq i\leq d $) such that $ x(z)=z+\frac{\sqrt{K}}{2d} L_{n-1}\boldsymbol{k} $. Therefore we get that each $ z\in A^+_n(\boldsymbol{k}) $ satisfies
\begin{equation*}
    |z|_{\infty} \leq\max_{1\leq i\leq d}\big|x(z)_i-\frac{\sqrt{K}}{2d}L_{n-1}\big| \leq \max\{\sqrt{K} L_{n-1}-\frac{\sqrt{K}}{2d}L_{n-1}, \frac{\sqrt{K}}{2d}L_{n-1} \} \leq \sqrt{K} L_{n-1},
\end{equation*}
and similarly we get that each $ w\in A^-_n(\boldsymbol{k}) $ satisfies
\begin{equation}
    \label{eq:A^-_n(k)}
    |w|_{\infty}\geq\frac{1}{d}|w|_1\geq \frac{1}{d}\big(\sqrt{K} L_{n-1}-6d^2L-(\frac{\sqrt{K}}{d}-10d)L_{n-1}|\boldsymbol{k}|_1\big)\geq 10L_{n-1}.
\end{equation}
As a result, we have
\begin{equation}
    \label{eq:A_n(k)}
    B(\mathbf{0},3L_{n-1})\cap A_n(\boldsymbol{k}) = \varnothing \quad{\rm and}\quad A_n(\boldsymbol{k})\subset B(\mathbf{0},\sqrt{K} L_{n-1}).
\end{equation}
Then we define the oriented surfaces as follows. For $ A\subset\Pi_{\boldsymbol{k}} $, we define its $ \boldsymbol{k} $-boundary ($ \boldsymbol{k} $-neighbor when $ A $ is a vertex) as
\begin{equation}
    \label{eq:def-oriented-neighbor}
    \partial^+_{\boldsymbol{k}}A:=\{x\notin A:\exists y\in A{\rm\ and\ }1\leq i\leq d,\,s.t.\,x=y+L\cdot\boldsymbol{l}_i(\boldsymbol{k})\}.
\end{equation}
Note that if $ y\in\Pi_{\boldsymbol{k}} $ is a $ \boldsymbol{k} $-neighbor of $ x\in\Pi_{\boldsymbol{k}} $, then there exists an index $ 1\leq i_0\leq d $ such that $ \pi^{-1}_{\boldsymbol{k}}(y)_{i_0}=\pi^{-1}_{\boldsymbol{k}}(x)_{i_0}+L $, and for all $ i\in\{i_0\}^c $ we have $ \pi^{-1}_{\boldsymbol{k}}(y)_i=\pi^{-1}_{\boldsymbol{k}}(x)_i $.
\begin{definition}
    \label{def:oriented_surface}
    Let $ n\geq1 $ and $ \boldsymbol{k}\in\{-1,1\}^d $. We say that $ S_n(\boldsymbol{k})\subset \Pi_{\boldsymbol{k}} $ is a $ \boldsymbol{k} $-oriented surface (of label $ n $) if the following hold: (i) for any $ x\in S_n(\boldsymbol{k}) $, $ (B(x,L),y(x)) $ is a marked box; (ii) there exists $ D_{\boldsymbol{k}}\subset\Pi_{\boldsymbol{k}} $ such that (ii.a) for any $ x,y\in\Pi_{\boldsymbol{k}} $ with $ y\preceq^{\prime}_{\boldsymbol{k}}x $, $ x\in D_{\boldsymbol{k}} $ implies $ y\in D_{\boldsymbol{k}} $; (ii.b) $ S_n(\boldsymbol{k})=A_n(\boldsymbol{k})\cap\partial^+_{\boldsymbol{k}}D_{\boldsymbol{k}} $, $ A^-_n(\boldsymbol{k}) $ is contained in $ D_{\boldsymbol{k}} $ and $ A^+_n(\boldsymbol{k}) $ does not intersect $ D_{\boldsymbol{k}} $ (here we allow $ S_n(\boldsymbol{k}) $ to intersect $ A^+_n(\boldsymbol{k}) $).
    Furthermore, we call $ S_n(\boldsymbol{k}) $ a nice oriented surface (for some fixed initial configuration) if and only if $ (B(x,L),y(x)) $ is nice for all $ x\in S_n(\boldsymbol{k}) $.
\end{definition}
An immediate corollary (due to (ii.b)) is that $ S_n(\boldsymbol{k}) $ disconnects $ A^+_n(\boldsymbol{k}) $ and $ A^-_n(\boldsymbol{k}) $ in $ A_n(\boldsymbol{k})\subset\Pi_{\boldsymbol{k}} $. We then define
\begin{equation}
    \label{eq:def-S_n}
    \mathbf{S}_n:=\bigcup_{\boldsymbol{k}\in\{-1,1\}^d}\bigcup_{x\in S_n(\boldsymbol{k})}{\rm Sh}(x,\boldsymbol{k}),
\end{equation}
where $ S_n(\boldsymbol{k}) $ is a $ \boldsymbol{k} $-oriented surface for all $ \boldsymbol{k}\in\{-1,1\}^d $.
We will prove in Lemma~\ref{lem:oriented_surface_disconnect_origin_from_infinity} that $ \mathbf{S}_n $ disconnects $ B(\mathbf{0},3L_{n-1}) $ and $ \partial_{i} B(\mathbf{0},\sqrt{K} L_{n-1}) $ in $ \mathbb{Z}^d $. (Note that this is a geometrical statement and hence does not need any information of niceness of oriented surfaces.)
We begin with the following auxiliary geometrical lemma that will be used in the proof of Lemma~\ref{lem:oriented_surface_disconnect_origin_from_infinity}. For a subset $ A\subset\Pi_{\boldsymbol{k}} $, define its inner boundary $ \partial_{\mathbf{i}}A $ as the collection of vertices in $A$ that have a neighbor in $ \Pi_{\boldsymbol{k}} $ which is not contained in $ A $.
Note that in $ \partial_{\mathbf{i}}A $ we use the bold subscript to emphasize the neighboring relation in $ \Pi_{\boldsymbol{k}} $ (for some $ A\subset\Pi_{\boldsymbol{k}} $), while $ \partial_{i}B $ indicates the ordinary neighboring relation in $ \mathbb{Z}^d $ (for some $ B\subset\mathbb{Z}^d $).
\begin{lemma}
    \label{lem:oriented-surface-ancillary}
    Fix an orientation $ \boldsymbol{k}\in\{-1,1\}^d $. Let $ D_{\boldsymbol{k}} $ be a subset of $ \Pi_{\boldsymbol{k}} $ satisfying (ii.a) of Definition~\ref{def:oriented_surface}, i.e., for all $ x,y\in\Pi_{\boldsymbol{k}} $ with $ y\preceq^{\prime}_{\boldsymbol{k}}x $, $ x\in D_{\boldsymbol{k}} $ implies $ y\in D_{\boldsymbol{k}} $. Then for any non-empty subset $ I\subsetneq\{1,\dots,d\} $ and any $ x\in \partial^{+}_{\boldsymbol{k}} D_{\boldsymbol{k}} $, there exist a vertex $ w^{\prime}\in \partial_{\mathbf{i}}D_{\boldsymbol{k}} $
    and a $ \boldsymbol{k} $-neighbor $ x^{\prime} $ of $ w^{\prime} $ such that: (1). $ x^{\prime}\in \partial^{+}_{\boldsymbol{k}} D_{\boldsymbol{k}} $; (2). $ (x^{\prime}_i-x_i)\boldsymbol{k}_i\leq 5^{5d}L $ for all $ i\in I $; (3). $ (x^{\prime}_j-x_j)\boldsymbol{k}_j\leq -3L $ for all $ j\in I^c $.\
\end{lemma}
\begin{proof}
    Let $ w\in\partial_{\mathbf{i}} D_{\boldsymbol{k}} $ be a $ \Pi_{\boldsymbol{k}} $-neighbor of $ x $.
    We will find a vertex $ w^{\prime} $ in $ \partial_{\mathbf{i}}B_{\Pi_{\boldsymbol{k}}}(w,2^{5d})\cap\partial_{\mathbf{i}}D_{\boldsymbol{k}} $ such that the desired property in the lemma-statement holds. First note that if $ w^{\prime}\in \partial_{\mathbf{i}}D_{\boldsymbol{k}} $, then the subset $ \partial^{+}_{\boldsymbol{k}} D_{\boldsymbol{k}}\cap\partial^{+}_{\boldsymbol{k}} w^{\prime} $ is non-empty, and thus (1) is satisfied by choosing $ x^{\prime}\in\partial^{+}_{\boldsymbol{k}} D_{\boldsymbol{k}}\cap\partial^{+}_{\boldsymbol{k}} w^{\prime} $.
    In addition, for any $ w^{\prime}\in\partial_{\mathbf{i}}B_{\Pi_{\boldsymbol{k}}}(w,2^{5d}) $, by~\eqref{eq:sublattice_prec} we have $ | w^{\prime} - w |_{\infty} \leq 2^{5d}\cdot12dL $. Therefore, if $ x^{\prime} $ is a $ \boldsymbol{k} $-neighbor of $ w^{\prime} $, then we have
    \begin{equation}
        \label{eq:lem4.14-a}
        (x^{\prime}_i - x_i)\boldsymbol{k}_i\leq| w^{\prime} - w |_{\infty}+2\cdot 12dL\leq 5^{5d}L
    \end{equation}
    for all $ 1\leq i\leq d $, which satisfies (2). Now we prove that (3) holds for some carefully chosen $ x^{\prime} $ and $ w^{\prime} $.
    Since $ x $ (resp., $ x^{\prime} $) is adjacent to $ w $ (resp., $ w^{\prime} $), by~\eqref{eq:sublattice_prec} and~\eqref{eq:lem4.14-a} it suffices to prove that there exists a vertex
    $ w^{\prime}\in\partial_{\mathbf{i}}D_{\boldsymbol{k}} \cap \mathsf{E}(w,I) $ where
    \begin{equation}
        \label{eq:lem4.14-000}
        \begin{aligned}
            \mathsf{E}(w,I) &:= \{ v\in\partial_{\mathbf{i}}B_{\Pi_{\boldsymbol{k}}}(w,2^{5d}): (v_j-w_j)\boldsymbol{k}_j\leq-30dL,\, \forall j\in I^c \} \\
            &= \{ w+\pi_{\boldsymbol{k}}(Ly): y\in\partial_{i}B(\mathbf{0},2^{5d}) \mbox{ and } 3dy_j-\sum^d_{i=1}y_i\leq-10d,\,\forall j\in I^c \},
        \end{aligned}
    \end{equation}
    where the equality follows from~\eqref{eq:def-pi_k}. In other words, we only need to prove that
    \begin{equation}
        \label{eq:lem4.14-b}
        \mathsf{E}(w,I)\cap \partial_{\mathbf{i}}D_{\boldsymbol{k}}\neq\varnothing.
    \end{equation}
    To this end, our idea is to find two vertices from a specific connected component of $ \mathsf{E}(w,I) $, and prove that one of these two vertices is contained in $ D_{\boldsymbol{k}} $ while the other one is contained in $ B_{\Pi_{\boldsymbol{k}}}(w,2^{5d})\setminus D_{\boldsymbol{k}} $.
    We next implement this proof idea. Note that
    \begin{equation*}
        z_w := w - \pi_{\boldsymbol{k}}(2^{5d}L\sum_{j\in I^c}\boldsymbol{e}_{j}) \in \mathsf{E}(w,I)
    \end{equation*}
    by a straightforward calculation.
    It is also easy to check that vertex
    \begin{equation}
        \label{eq:lem4.14-def-z_prime_w}
        z^{\prime}_w := w + \pi_{\boldsymbol{k}}(2^{5d}L\sum_{i\in I}\boldsymbol{e}_i)
    \end{equation}
    satisfies $ w\preceq^{\prime}_{\boldsymbol{k}}z^{\prime}_w $ and $ z^{\prime}_w\in\mathsf{E}(w,I) $.
    We claim the following: (i) $ z_w\in D_{\boldsymbol{k}} $, (ii) $ z_w $ and $ z^{\prime}_w $ are connected in $ \mathsf{E}(w,I) $, and (iii) $ z^{\prime}_w\notin D_{\boldsymbol{k}}\setminus\partial_{\mathbf{i}}D_{\boldsymbol{k}} $. Note that if Claim (iii) holds, then \textbf{either} $ z^{\prime}_w\in\partial_{\mathbf{i}}D_{\boldsymbol{k}} $ which already yields~\eqref{eq:lem4.14-b} since $ z^{\prime}_w\in\mathsf{E}(w,I) $, \textbf{or} $ z^{\prime}_w\in B_{\Pi_{\boldsymbol{k}}}(w,2^{5d})\setminus D_{\boldsymbol{k}} $ which combined with Claims (i) and (ii) implies~\eqref{eq:lem4.14-b}. Therefore, it remains to prove Claims (i), (ii) and (iii). See Figure~\ref{fig:oriented-surface-ancillary} for an illustration of the following argument in the case that $ d=3 $, $ \boldsymbol{k}=(1,1,1) $ and $ I=\{3\} $.
    \begin{figure}[h]
        \centering
        \includegraphics[width=0.5\textwidth]{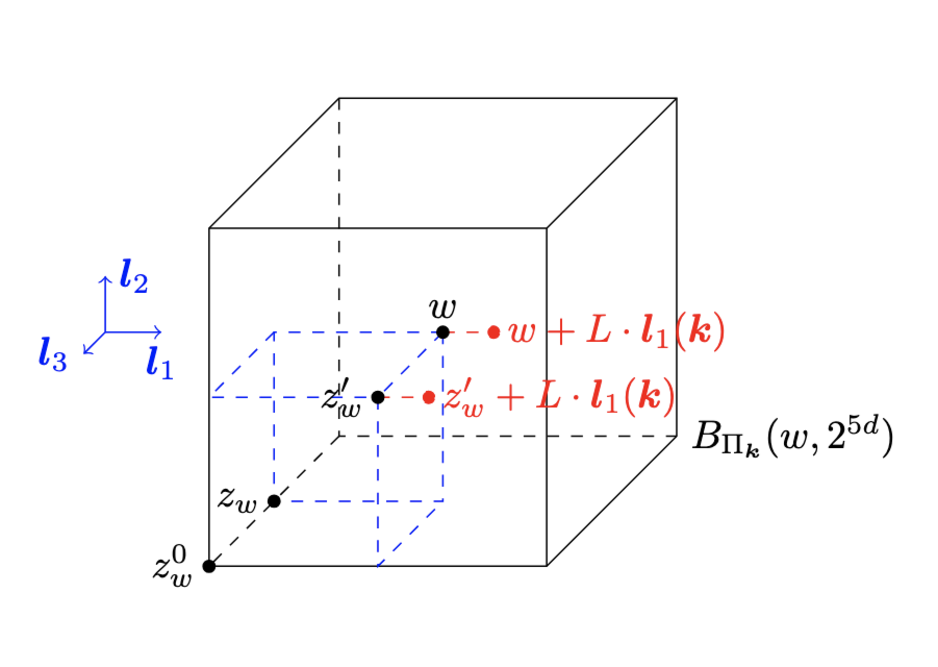}
        \caption{An illustration of Lemma~\ref{lem:oriented-surface-ancillary} for $ d=3 $, $ \boldsymbol{k}=(1,1,1) $ and $ I=\{3\} $. The sublattice $ \Pi_{\boldsymbol{k}} $ is drawn orthogonally for better clarity. The vertex $ \pi^{-1}_{\boldsymbol{k}}(z_w) $ has the same $ i $-th coordinate as $ \pi^{-1}_{\boldsymbol{k}}(w) $ for all $ i\in I=\{3\} $, and the vertex $ \pi^{-1}_{\boldsymbol{k}}(z^{\prime}_w) $ has the same $ j $-th coordinate as $ \pi^{-1}_{\boldsymbol{k}}(w) $ for all $ j\in I^c=\{1,2\} $. One of the neighbors of $ z^{\prime}_w $ is colored in red, and so is one of the neighbors of $ w $.}
        \label{fig:oriented-surface-ancillary}
    \end{figure}

    Claim (i) is easy to check since $ z_w\preceq^{\prime}_{\boldsymbol{k}}w $ and $ w\in D_{\boldsymbol{k}} $ (recall (ii.a) of Definition~\ref{def:oriented_surface}).
    As for Claim (ii), we first notice that
    \begin{equation*}
         \{ w+\pi_{\boldsymbol{k}}(Ly): |y|_{\infty}=2^{5d} \mbox{ and } y_j=-2^{5d},\,\forall j\in I^c \} \supset \{ z_w, z^0_w \}
    \end{equation*}
    is a connected subset of $ \mathsf{E}(w,I) $ where
    \begin{equation*}
        z^0_w:=w-\pi_{\boldsymbol{k}}(2^{5d}L\sum_{j\in I^c}\boldsymbol{e}_j)+\pi_{\boldsymbol{k}}(2^{5d}L\sum_{i\in I}\boldsymbol{e}_i).
    \end{equation*}
    Therefore it suffices to prove that $ z^0_w $ and $ z^{\prime}_w $ are connected in $ \mathsf{E}(w,I) $. To this end, let us consider the subset $ \{ w+\pi_{\boldsymbol{k}}(Ly): y\in \mathsf{E}^0 \} $ where
    \begin{equation*}
        \mathsf{E}^0:=\bigcup^{2^{5d}}_{t=0}\{ y\in\partial_{i}B(\mathbf{0},2^{5d}): \ |y_j+t|\leq1,\,\forall j\in I^c \mbox{ and } y_i=2^{5d},\,\forall i\in I \}.
    \end{equation*}
    By $ |I|\geq1 $, for all $ y\in \mathsf{E}^0 $ and all $ j\in I^c $, we have
    \begin{equation*}
        \response{\begin{aligned}
            3dy_j-\sum^d_{i=1}y_i
            &\leq 3d(-t+1) - \big(2^{5d}|I| + (-t-1)|I^c|\big) \\
            &\leq -2^{5d}|I| + 3d+|I^c| - t(3d-|I^c|) \leq -10d.
        \end{aligned}}
    \end{equation*}
    Thus, by~\eqref{eq:lem4.14-000}, $ \{ w+\pi_{\boldsymbol{k}}(Ly): y\in \mathsf{E}^0 \}\subset\mathsf{E}(w,I) $. Then Claim (ii) follows from observing that $ \{ w+\pi_{\boldsymbol{k}}(Ly): y\in \mathsf{E}^0 \} $ is a connected subset that contains both $ z^0_w $ and $ z^{\prime}_w $.
    Finally, we prove Claim (iii) by contradiction. If $ z^{\prime}_w\in D_{\boldsymbol{k}}\setminus\partial_{\mathbf{i}}D_{\boldsymbol{k}} $, then we have $ z^{\prime}_w+L\cdot\boldsymbol{l}_i(\boldsymbol{k})\in D_{\boldsymbol{k}} $ for all $ 1\leq i\leq d $. This then implies that $ w+L\cdot\boldsymbol{l}_i(\boldsymbol{k})\in D_{\boldsymbol{k}} $ for all $ 1\leq i\leq d $ (by $ w\preceq_{\boldsymbol{k}}z^{\prime}_w $ and (ii.a) of Definition~\ref{def:oriented_surface}), which contradicts $ w\in\partial_{\mathbf{i}}D_{\boldsymbol{k}} $, thus concludes the proof.
\end{proof}
\begin{lemma}
    \label{lem:oriented_surface_disconnect_origin_from_infinity}
    Define $ \mathbf{S}_n $ as in~\eqref{eq:def-S_n} where $ S_n(\boldsymbol{k}) $ is a $ \boldsymbol{k} $-oriented surface for all $ \boldsymbol{k}\in\{-1,1\}^d $.
    Then $ \mathbf{S}_n $ disconnects $ B(\mathbf{0},3L_{n-1}) $ from $ \partial_{i} B(\mathbf{0},\sqrt{K} L_{n-1}) $ on $\mathbb{Z}^d$.
\end{lemma}
\begin{proof}
    Define
    \begin{equation}
        \label{eq:def-S^prime_n}
        \mathbf{S}^{\prime}_n:=\big\{ z\in\mathbb{Z}^d: W(z)\geq1, W^{\prime}(z)=0 \big\},
    \end{equation}
    where
    \begin{equation}
        \label{eq:def_W(z)}
        W(z)=\sum_{\boldsymbol{k}\in\{-1,1\}^d}\sum_{x\in S_n(\boldsymbol{k})}\mathbf{1}_{y(x)\preceq_{\boldsymbol{k}}z}\quad{\rm and}\quad W^{\prime}(z)=\sum_{\boldsymbol{k}\in\{-1,1\}^d}\sum_{x\in S_n(\boldsymbol{k})}\mathbf{1}_{y(x)\prec_{\boldsymbol{k}}z}.
    \end{equation}
    Note that for all $ z\in B(\mathbf{0}, 3L_{n-1}) $, $ \boldsymbol{k}\in\{-1,1\}^d $ and all $ x\in S_n(\boldsymbol{k}) $, by~\eqref{eq:A^-_n(k)} and $ |x|_{\infty}\geq\min_{w\in A^-_n(\boldsymbol{k})}|w|_{\infty} $ we have that \response{(recall that $n\geq1$ and $ L_0\geq L $)}
    \begin{equation*}
        |z|_{\infty} \leq 3L_{n-1} \response{\leq 10L_{n-1}-L\leq\min_{w\in A^-_n(\boldsymbol{k})}|w|_{\infty}-L\leq |x|_{\infty}-L} \leq |y|_{\infty}
    \end{equation*}
    for all $ y\in B(x,L) $.
    This then implies that $ y\not\preceq_{\boldsymbol{k}}z $ for any $ y\in B(x,L) $ (which in particular holds for $y$ being the marked point $ y(x) $) and hence $ W(z)=0 $ for all $ z\in B(\mathbf{0},3L_{n-1}) $.
    Therefore $ \mathbf{S}^{\prime}_n\cap B(\mathbf{0},3L_{n-1})=\varnothing $. Then by~\eqref{eq:A_n(k)} it remains to prove that $ \mathbf{S}^{\prime}_n $ disconnects $ \mathbf{0} $ from infinity and $ \mathbf{S}^{\prime}_n\subset\mathbf{S}_n $.

    \textbf{Step 1}. We will show that $ \mathbf{S}^{\prime}_n $ disconnects $ \mathbf{0} $ from infinity. First we prove that $ W(z)\geq1 $ for all $ z\in\mathbb{Z}^d $ with $ |z|_{\infty}\geq dL_{n} $.
    For each $ 1\leq i\leq d $, by the second inequality of~\eqref{eq:sublattice_prec} we can find a vertex $ \mathsf{k}^{(i)}\in\Pi_{\boldsymbol{k}} $ such that $ \boldsymbol{k}_j\mathsf{k}^{(i)}_j \in[\sqrt{K} L_{n-1}-48d^2L, \sqrt{K} L_{n-1}-24d^2L] $ if $ j=i $, and $ \boldsymbol{k}_j\mathsf{k}^{(i)}_j\in[-48d^2L,-24d^2L] $ if $j \in \{i\}^c$. \response{This is because the domain where we want to find such vertex $ \mathsf{k}^{(i)} $ is a box in $ \mathbb{Z}^d $ with $|\cdot|_{\infty}$-diameter at least $ 24d^2L $, hence must contain a vertex of the sublattice $ \Pi_{\boldsymbol{k}} $, whose cell size can be upper-bounded by~\eqref{eq:sublattice_prec}.} Then $ \mathsf{k}^{(i)} $ has $ |\cdot|_{\infty} $-distance at most $ 48d^2L $ to $ \sqrt{K} L_{n-1}\boldsymbol{k}_i\boldsymbol{e}_i \in A^0_n(\boldsymbol{k}) $.
    We claim that for each $ 1\leq i\leq d $ and $\boldsymbol{k}\in\{-1,1\}^d$, there exists a vertex $ x\in S_n(\boldsymbol{k}) $ satisfying $ x\preceq_{\boldsymbol{k}} \mathsf{k}^{(i)} $, which will be proved later. For each $ |z|_{\infty}\geq dL_n $, there exists an orientation $ \boldsymbol{k} $ such that $ \mathbf{0}\preceq_{\boldsymbol{k}}z $. Under this condition, we can further find $ 1\leq i\leq d $ such that $ \mathsf{k}^{(i)}+L\boldsymbol{k} \preceq_{\boldsymbol{k}}z $. This, combined with $ y(x) \preceq_{\boldsymbol{k}} x+L\boldsymbol{k} $, yields $ y(x)\preceq_{\boldsymbol{k}}\mathsf{k}^{(i)}+L\boldsymbol{k}\preceq_{\boldsymbol{k}}z $, which implies that $ W(z)\geq1 $.

    \begin{figure}[h]
        \centering
        \includegraphics[width=0.5\textwidth]{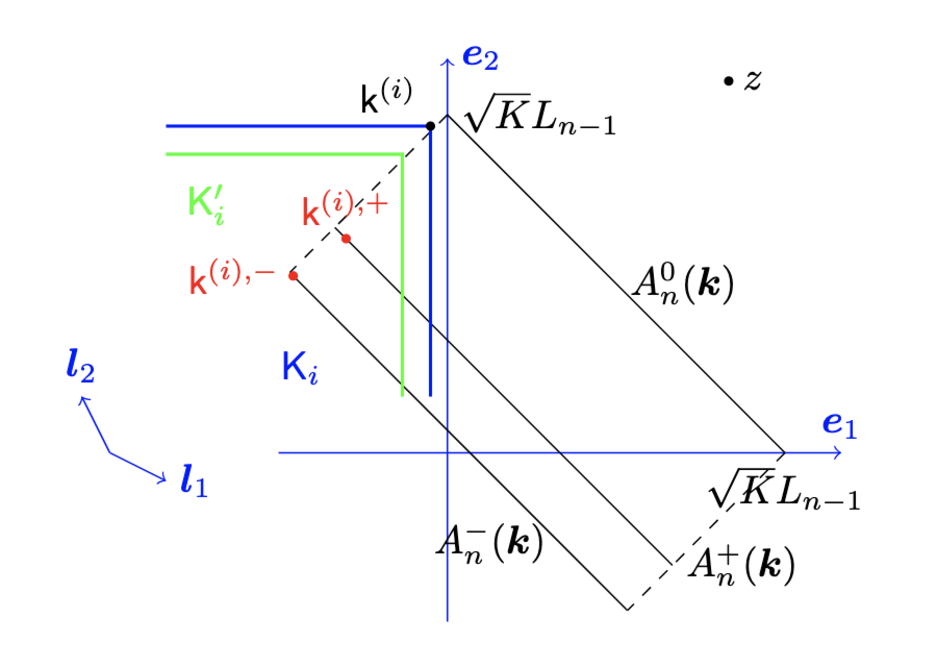}
        \caption{An illustration for $ d=2 $, $ \boldsymbol{k}=(1,1) $ and $ i=2 $. The boundary of $ \mathsf{K}_i $ is colored in blue, and the boundary of $ \mathsf{K}^{\prime}_i $ is colored in green. The vertices $ \mathsf{k}^{(i),\pm} $ are colored in red. Note that the vertex $ z\in\mathbb{Z}^d $ satisfies $ |z|_{\infty}\geq dL_{n} $ and $ \mathsf{k}^{(i)}\preceq_{\boldsymbol{k}} z $.}
        \label{fig:lem4.15-step1}
    \end{figure}
    We now prove the preceding claim via contradiction. See Figure~\ref{fig:lem4.15-step1} for an illustration of the following argument in the case that $ d=2 $, $ \boldsymbol{k}=(1,1) $ and $ i=2 $. Suppose there exist some $ \boldsymbol{k}\in\{-1,1\}^d $ and $ 1\leq i\leq d $ such that $ x\not\preceq_{\boldsymbol{k}}\mathsf{k}^{(i)} $ for all $ x\in S_n(\boldsymbol{k}) $. Then we have
    \begin{equation}
        \label{eq:lem4.15-1}
        \mathsf{K}_i \cap S_n(\boldsymbol{k})=\varnothing, \quad \mbox{ where $ \mathsf{K}_i := \{ w\in\Pi_{\boldsymbol{k}}: w\preceq_{\boldsymbol{k}}\mathsf{k}^{(i)} \} $.}
    \end{equation}
    We are going to find a vertex $ \mathsf{k}^{(i),+}\in\mathsf{K}_i \cap A^+_n(\boldsymbol{k}) $ and prove that $ \mathsf{k}^{(i),+}\in D_{\boldsymbol{k}} $ (here $ D_{\boldsymbol{k}} $ is some arbitrary subset that serves as in Definition~\ref{def:oriented_surface} for $ S_n(\boldsymbol{k}) $),
    which contradicts the assumption that $ A^+_n(\boldsymbol{k})\cap D_{\boldsymbol{k}}=\varnothing $ (recall (ii.b) of Definition~\ref{def:oriented_surface}). To this end, let us first consider two arbitrarily chosen vertices (by~\eqref{eq:sublattice_prec},~\eqref{eq:def-A^pm_n(k)} and the fact that $ \mathsf{k}^{(i)} $ has $ |\cdot|_{\infty} $-distance at most $ 48d^2L $ to $ A^0_n(\boldsymbol{k}) $ we can indeed find such vertices)
    \begin{equation*}
        \mathsf{k}^{(i),-}\in A^-_n(\boldsymbol{k}) \cap B(\mathsf{k}^{(i)}-\big( \frac{\sqrt{K}}{d} - 10d \big)L_{n-1}\boldsymbol{k}, 100d^2L)\subset A^{-}_n(\boldsymbol{k})\subset D_{\boldsymbol{k}},
    \end{equation*}
    and
    \begin{equation*}
        \mathsf{k}^{(i),+} \in A^+_n(\boldsymbol{k})\cap B(\mathsf{k}^{(i)}-\frac{\sqrt{K}}{2d}L_{n-1}\boldsymbol{k}, 100d^2L).
    \end{equation*}
    With a straightforward calculation we know that $ \mathsf{k}^{(i),\pm}\in\mathsf{K}_i $.
    Then consider the subset
    \begin{equation*}
        \mathsf{K}^{\prime}_i:=\{ w\in \Pi_{\boldsymbol{k}}: \mbox{the } \Pi_{\boldsymbol{k}} \mbox{-graph distance between }w\mbox{ and }\Pi_{\boldsymbol{k}}\setminus\mathsf{K}_i\mbox{ is at least }2 \}\subset\mathsf{K}_i.
    \end{equation*}
    Since the $ \Pi_{\boldsymbol{k}} $-graph distance between $ S_n(\boldsymbol{k}) $ and $ \partial_{\mathbf{i}} D_{\boldsymbol{k}} $ is 1 and $ S_n(\boldsymbol{k})\cap \mathsf{K}_i=\varnothing $ (recall~\eqref{eq:lem4.15-1}), we have $ \partial_{\mathbf{i}} D_{\boldsymbol{k}} \cap \mathsf{K}^{\prime}_i=\varnothing $.
    Again we can check through straightforward calculations that $ \mathsf{k}^{(i),-} $ and $ \mathsf{k}^{(i),+} $ are connected in $ \mathsf{K}^{\prime}_{i} $. Combining this with the facts that $ \mathsf{k}^{(i),-}\in\mathsf{K}^{\prime}_i\cap D_{\boldsymbol{k}} $ and $ \partial_{\mathbf{i}} D_{\boldsymbol{k}} \cap \mathsf{K}^{\prime}_i=\varnothing $, we see that $ \mathsf{k}^{(i),+}\in D_{\boldsymbol{k}} $. This contradicts $ A^+_n(\boldsymbol{k})\cap D_{\boldsymbol{k}}\neq\varnothing $, thereby proving the claim.
   
    Now we have $ W(z)\geq1 $ for all $ z\in\mathbb{Z}^d $ with $ |z|_{\infty}\geq dL_n $. Then by $ W(\mathbf{0})=0 $ we know that for any nearest neighbor path (abbreviated as n.n.path in the rest of this subsection) $ \eta $ that starts from $ z $ and ends at $\mathbf{0}$, it contains a vertex $ z^{\prime} $ which is the last (under the parameterization such that $\eta(0)=z$) vertex satisfying $ W(z^{\prime}) \geq 1 $. Now we claim that $ W^{\prime}(z^{\prime})=0 $, which implies $ z^{\prime}\in\mathbf{S}^{\prime}_n $ and hence concludes Step 1. In fact, if $ W^{\prime}(z^{\prime})\geq1 $, then there must exist some $ \boldsymbol{k} $ and $ x\in S_n(\boldsymbol{k}) $ such that $ y(x)\prec_{\boldsymbol{k}}z^{\prime} $. Let $ z^{\prime\prime} $ be the next vertex after $ z^{\prime} $ in $ \eta $, we have $ |z^{\prime\prime}-z^{\prime}|_1=1 $ and hence $ y(x)\preceq_{\boldsymbol{k}}z^{\prime\prime} $ and $ W(z^{\prime\prime})\geq 1 $, which contradicts the assumption of ``last'' for $ z^{\prime} $.

    \textbf{Step 2}. We now prove that $ \mathbf{S}^{\prime}_n\subset\mathbf{S}_n $.
    To this end, we only need to show that for any $ z\in\mathbf{S}^{\prime}_n $ with $ y(x)\preceq_{\boldsymbol{k}}z $ for some $ \boldsymbol{k}\in\{-1,1\}^d $ and $ x\in S_n(\boldsymbol{k}) $, we have $ z\in{\rm Sh}(x,\boldsymbol{k})\subset\mathbf{S}_n $. Our proof proceeds by showing that if $ z\notin{\rm Sh}(x,\boldsymbol{k}) $, then we can find $ x^{\prime}\in S_n(\boldsymbol{k}) $ such that $ y(x^{\prime})\preceq_{\boldsymbol{k}}x^{\prime}+L\boldsymbol{k}\prec_{\boldsymbol{k}}z $, leading to $ W^{\prime}(z)\geq1 $ and $ z\notin\mathbf{S}^{\prime}_n $ (a contradiction).
    See Figure~\ref{fig:oriented-surface-disconnects} for an illustration for $ d=2 $ of the following arguments.
    \begin{figure}[h]
        \centering
        \includegraphics[width=0.5\textwidth]{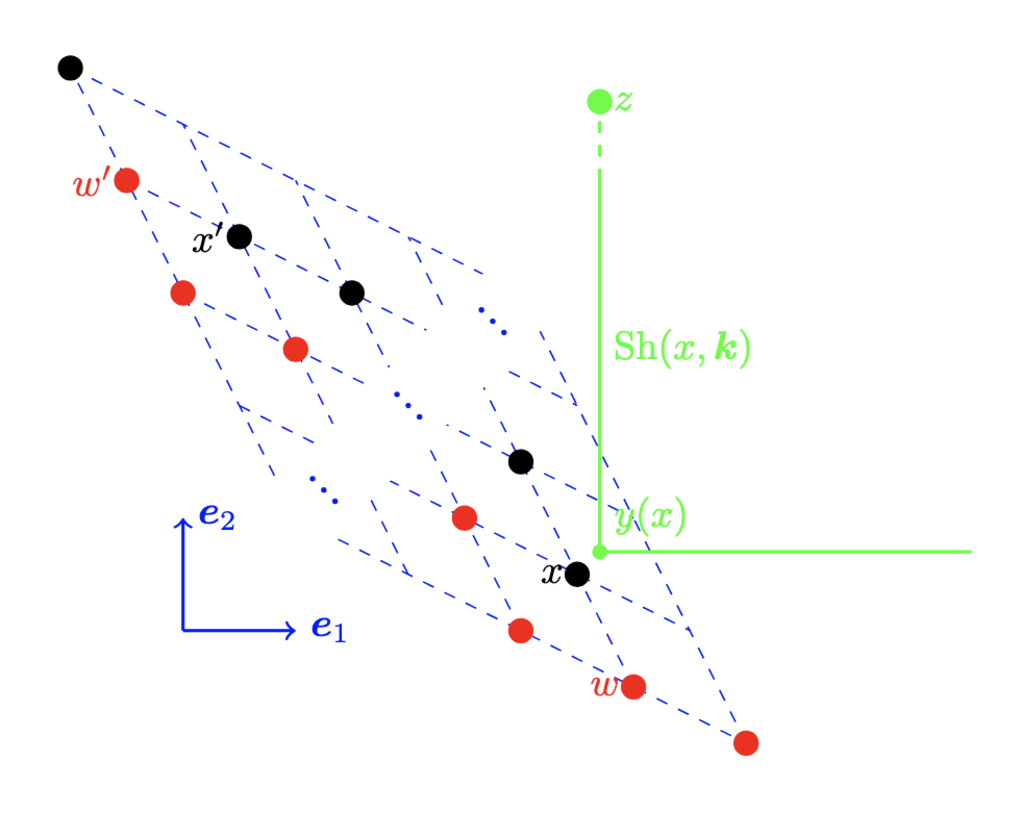}
        \caption{An illustration for proof of Lemma~\ref{lem:oriented_surface_disconnect_origin_from_infinity} for $d=2$ and $ \boldsymbol{k}=(1,1) $. The inner boundary of $D_{\boldsymbol{k}}$ is colored in red, and $ S_n(\boldsymbol{k}) $ is colored in black. The sublattice $ \Pi_{\boldsymbol{k}} $ is in dashed blue. The shield $ {\rm Sh}(x,\boldsymbol{k}) $ and the vertex $z$ are colored in green. In the shown case, $ I_z=\{2\} $.}
        \label{fig:oriented-surface-disconnects}
    \end{figure}

    Now let $ z $ (and the associated $ x $, $ \boldsymbol{k} $) be described as in the previous paragraph. Suppose $ z\notin{\rm Sh}(x,\boldsymbol{k}) $.
    Then there exists a non-empty subset $ I_z\subsetneq\{1,\dots,d\} $ such that
    \begin{equation}
        \label{eq:lem4.13-0}
        (z_i-y(x)_i)\cdot\boldsymbol{k}_i>10^{10d}L
    \end{equation}
    for all $ i\in I_z $. The subset $ I_z $ is non-empty due to the fact that $ z\notin{\rm Sh}(x,\boldsymbol{k}) $ and the fact that $ (z_i-y(x)_i)\boldsymbol{k}_i=0 $ for some $i$ (since $ W^{\prime}(z)=0 $); the fact $ I_z\neq\{1,\dots,d\} $ is due to the fact that $ y(x)\not\prec_{\boldsymbol{k}}z $.
    Recall $ D_{\boldsymbol{k}} $ from Step 1 as an arbitrary subset in $ \Pi_{\boldsymbol{k}} $ that serves as in Definition~\ref{def:oriented_surface}. Suppose $ w\in \partial_{\mathbf{i}} D_{\boldsymbol{k}} $ is adjacent (in $ \Pi_{\boldsymbol{k}} $) to $ x $.
    Then by Lemma~\ref{lem:oriented-surface-ancillary} (where we take $ I=I_z $) there exists a vertex $ w^{\prime}\in \partial_{\mathbf{i}}D_{\boldsymbol{k}} $ and a $ \boldsymbol{k} $-neighbor $ x^{\prime}\in S_n(\boldsymbol{k}) $ of $ w^{\prime} $ such that
    \begin{equation}
        \label{eq:lem4.13-1}
        (x^{\prime}_i-x_i)\boldsymbol{k}_i\leq 5^{5d}L,\forall i\in I_z\quad{\rm and}\quad(x^{\prime}_j-x_j)\boldsymbol{k}_j\leq-3L,\forall j\in I^c_z.
    \end{equation}
    Combining $ y(x)\preceq_{\boldsymbol{k}}z $ with~\eqref{eq:lem4.13-0} and~\eqref{eq:lem4.13-1}, we know that
    \begin{equation}
        \label{eq:lem4.14-0}
        z_i\boldsymbol{k}_i > y(x)_i\boldsymbol{k}_i+10^{10d}L \geq x_i\boldsymbol{k}_i+(10^{10d}-2)L \geq x^{\prime}_i\boldsymbol{k}_i+(10^{10d}-5^{5d}-2)L \geq x^{\prime}_i\boldsymbol{k}_i + 2L
    \end{equation}
    for all $ i\in I_z $ and
    \begin{equation}
        \label{eq:lem4.14-1}
        z_j\boldsymbol{k}_j \geq y(x)_j\boldsymbol{k}_j \geq x_j\boldsymbol{k}_j-L \geq x^{\prime}_j\boldsymbol{k}_j+2L
    \end{equation}
    for all $ j\in I^c_z $. Then~\eqref{eq:lem4.14-0} and~\eqref{eq:lem4.14-1} imply that $ y(x^{\prime}) \prec_{\boldsymbol{k}} x^{\prime}+2L\boldsymbol{k} \preceq_{\boldsymbol{k}} z $ which then concludes the proof of Step 2.
\end{proof}
Recall $ S_n(\boldsymbol{k}) $ from Definition~\ref{def:oriented_surface}, and let $ S_n(\boldsymbol{k}) $ be an arbitrary $ \boldsymbol{k} $-oriented surface for all $ \boldsymbol{k}\in\{-1,1\}^d $. Recall from~\eqref{eq:def-S^prime_n} and~\eqref{eq:def_W(z)} the definitions of $ \mathbf{S}^{\prime}_n $, $ W(z) $ and $ W^{\prime}(z) $ for $ z\in\mathbb{Z}^d $.
By Step 1 in the proof of Lemma~\ref{lem:oriented_surface_disconnect_origin_from_infinity}, we know that $ \mathbf{S}^{\prime}_n $ separates $ \mathbb{Z}^d $ into two components, and by $ W(\mathbf{0})=0 $ we have
\begin{equation}
    \label{eq:def-So_n}
     \mathbf{S}^{\mathrm o}_n:=\{z\in\mathbb{Z}^d: W(z)=0\}
\end{equation}
is the connected component of $\mathbf{0}$ in $ \mathbb{Z}^d\setminus\mathbf{S}^{\prime}_n $. By Step 2 in the proof of Lemma~\ref{lem:oriented_surface_disconnect_origin_from_infinity} we have $ \mathbf{S}^{\prime}_n\subset\mathbf{S}_n $, and by definition any vertex $ z\in\mathbf{S}_n $ satisfies $ W(z)\geq1 $ and hence is not contained in $ \mathbf{S}^{\mathrm o}_n $. Therefore $ \mathbf{S}^{\mathrm o}_n $ is also the connected component of $ \mathbf{0} $ in $ \mathbb{Z}^d\setminus\mathbf{S}_n $, or in other words, is the collection of vertices that are enclosed by $ \mathbf{S}_n $. The following lemma states that if $ \mathbf{S}_n $ is defined with nice oriented surfaces, then it is the desired contour (such that the evolution starting from $ B(\mathbf{0},3L_{n-1}) $ will stop before reaching it).
\begin{lemma}
    \label{lem:oriented_surface_prevents_growth}
    Define $ \mathbf{S}_n $ and $ \mathbf{S}^{\mathrm o}_n $ as in~\eqref{eq:def-S_n} and~\eqref{eq:def-So_n}.
    Fix an initial configuration and suppose all the oriented surfaces in defining $ \mathbf{S}_n $ are nice.
    Consider the evolution grown from $ \mathbf{S}_n $ (recall from the statement of Lemma~\ref{lem:control-growth} that we only allow the evolution grown from $B(\mathbf{0},3L_{n-1})$ which is enclosed by $ \mathbf{S}_n $) with the following initial configuration: all the vertices in $ \mathbf{S}^{\mathrm o}_n $ are open and the other vertices are not changed.
    Then in the final configuration the open cluster containing $ \mathbf{S}^{\mathrm o}_n $ is within $ |\cdot|_{\infty} $-distance $10d$ from $ \mathbf{S}^{\mathrm o}_n $.
\end{lemma}
\begin{proof}
    From Definition~\ref{def:nice-box}, we see that the initial configuration in the lemma-statement satisfies the following: all the corner vertices of some shields remain closed; each $ z\in \mathsf{R}(y)\setminus\mathbf{S}^{\mathrm o}_n $ is not initially open for any corner vertex $ y\in\mathbf{S}_n $ of some shields; for all vertex
    $ w\in \cup_{\boldsymbol{k}}\cup_{x\in S_n(\boldsymbol{k})} B(x, 10^{10d}L)\setminus\mathbf{S}^{\mathrm o}_n $ there are at most $4d$ initially open vertices in $ B(w,10d)\setminus\mathbf{S}^{\mathrm o}_n $.
    
    Our proof is based on a similar argument as in Lemma~\ref{lem:shield-local}. For a vertex $ z\in\mathbb{Z}^d\setminus\mathbf{S}^{\mathrm o}_n $, we say that $ z $ is of type-1 if it has at least $ d-1 $ neighbors in $ \mathbf{S}^{\mathrm o}_n $, and we say $ z $ is of type-2 if it has at most $d-2$ neighbors in $ \mathbf{S}^{\mathrm o}_n $. Since we only allow the evolution to grow from $ \mathbf{S}_n $, it suffices to control the open clusters in $ \mathbb{Z}^d\setminus\mathbf{S}^{\mathrm o}_n $ that are grown from the outer boundary of $ \mathbf{S}^{\mathrm o}_n $, i.e. $ \mathbf{S}^{\prime}_n $. Note that if $ z $ is of type-1, then there exists some orientation $\boldsymbol{k}\in\{-1,1\}^d$ and $ x\in S_n(\boldsymbol{k}) $ such that there exists at most 1 index $ i\in\{1,\dots,d\} $ satisfying that $ z_i\neq y(x)_i $, where $ y(x) $ is the corner vertex of the shield $ {\rm Sh}(x,\boldsymbol{k}) $. Then by (2) of Definition~\ref{def:nice-box} all the type-1 vertices are not within $ |\cdot|_{\infty} $-distance $ 100d $ to any initially open vertex. Therefore if a type-1 vertex is open in the final configuration, there must exist an open cluster grown from $ \mathbb{Z}^d\setminus\big( \mathbf{S}^{\mathrm o}_n\cup (\cup_{\boldsymbol{k}}\cup_{x\in S_n(\boldsymbol{k})}\mathsf{R}(y(x))) \big) $ (since by the lemma-statement we only need to consider the open clusters that grown from $B(\mathbf{0},3L_{n-1})$, which is enclosed by $ \mathbf{S}^{\mathrm{o}} $) with $ |\cdot|_{\infty} $-diameter at least $100d$ even when the growth is restricted to type-2 vertices.
    This is not true by Lemma~\ref{lem:initial-open-number-lower-bound} (applied to the clusters of type-2 vertices) and (3) of Definition~\ref{def:nice-box}, and thus any type-1 vertex is not open in the final configuration. Then by Lemma~\ref{lem:initial-open-number-lower-bound} again we know that all the open clusters grown from $ \mathbf{S}^{\prime}_n $ (after removing $ \mathbf{S}^{\prime}_n $) have $ |\cdot|_{\infty} $-diameters at most $10d$, thereby concluding the proof.
\end{proof}

Now we only need to prove that with probability at least $ 1-e^{-K\cdot2^n}/(100K)^d $ there exists an $ \mathbf{S}_n $ defined as in~\eqref{eq:def-S_n} where each $ S_n(\boldsymbol{k}) $ is a nice oriented surface.
Recall the setting of Lemma~\ref{lem:control-growth} where we initially sample open vertices (and open cubes) and closed vertices on $\mathbb{Z}^d$ according to the law $\mathbb{Q}_{n-1}$.
Then by (2) of Definition~\ref{def:nice-box}, for all $ \boldsymbol{k}\in\{-1,1\}^d $, each $ x\in\Pi_{\boldsymbol{k}} $ within $ |\cdot|_{\infty} $-distance $ 10^{10d}L $ to an initially sampled open $ 3\sqrt{K} L_k $-box for $ 0\leq k\leq n-2 $ (as defined in \response{Definition~\ref{def:Qn}}) cannot be nice.
This motivates us to define a probability measure on $ \Pi_{\boldsymbol{k}} $ to facilitate the analysis of nice boxes.
\begin{definition}
    \label{def:overwhelming}
    Let $ \delta\in(0,1) $ and fix some $ \boldsymbol{k}\in\{-1,1\}^d $. We say a distribution $ \mathbb{Q}_{n-1}^{\prime} $ on $ \{0,1\}^{\Pi_{\boldsymbol{k}}} $ is $ (\delta,n) $-overwhelming, if there exist $ \{\tilde{p}_j\}_{-1\leq j\leq n-2}\subset[0,1] $ and $ \{a_j\}_{-1\leq j\leq n-2}\subset\mathbb{Z}^+ $ such that the following hold.
    \begin{enumerate}[(1)]
        \item $ \tilde{p}_{-1}=\delta $, $ a_{-1}=1 $ and $ \tilde{p}_j= p_j^{\alpha_d} \leq p^{2d\alpha_d}e^{-\alpha_d K\cdot2^j} $, $ a_j= K^{j+1.5} $ for all $ 0\leq j\leq n-2 $, where $ \alpha_d=10^{-10d} $ is a constant depending only on $d$.
        \item Under $ \mathbb{Q}_{n-1}^{\prime} $ there exist independent Bernoulli variables $ \{X_{x,j}\} $ with expectation $ \tilde{p}_j $ for all $ x\in\Pi_{\boldsymbol{k}} $ and $ -1\leq j\leq n-2 $.
    \end{enumerate}
    With respect to a realization of $ \mathbb{Q}^{\prime}_{n-1} $, we call a vertex $ x\in\Pi_{\boldsymbol{k}} $ poor, if there exist $ -1\leq j\leq n-2 $ and $ y\in B_{\Pi_{\boldsymbol{k}}}(x,a_j) $ such that $ X_{y,j}=1 $, and in this case we say that $x$ is poor due to $(y,j)$. Furthermore if $ x\in\Pi_{\boldsymbol{k}} $ is poor, let $ \mathsf{N}(x) $ denote the collection of all pairs due to which $ x $ is poor.
\end{definition}
\begin{lemma}
    \label{lem:overwhelming_dominance}
    For any fixed $ \delta\in(0,1) $, there exists a constant $ K_0=K_0(\delta,d) $ such that for all $ K\geq K_0 $, there exists a constant $ C=C(K,d) $ such that the following holds: for all $ p\leq C^{-1}, q\geq Cp^d, n\geq1 $ and any fixed orientation $ \boldsymbol{k}\in\{-1,1\}^d $, we can couple $ \mathbb{Q}_{n-1} $ with some $ (\delta,n) $-overwhelming measure $ \mathbb{Q}_{n-1}^{\prime} $ such that
    \begin{equation*}
        \{ x\in\Pi_{\boldsymbol{k}}: x{\rm\ is\ not\ nice} \}\subset\{ x\in\Pi_{\boldsymbol{k}}: x{\rm\ is\ poor} \}.
    \end{equation*}
\end{lemma}
\begin{proof}
    By symmetry we fix the orientation $ \boldsymbol{k} $ in what follows. Also fix $ n\geq1 $ and $ \delta\in(0,1) $.
    We will deal with $ 0\leq j\leq n-2 $ and $ j=-1 $ in two steps.
    Recall that under $ \mathbb{Q}_{n-1} $ we independently sample two types of open boxes which we will call them type-I open boxes and type-II open boxes respectively. A type-I open box is of the form $ B(z,3\sqrt{K} L_{j}) $ for $ 0\leq j\leq n-2 $ and $ z\in L_j\mathbb{Z}^d $, sampled independently with probability $ p_j $; a type-II open box is of the form $ B(z,k) $ for $ 0\leq k\leq 100^d $, sampled independently with probability $ 3^dp^{1+k/3} $.
    Also recall from Definition~\ref{def:nice-box} the definition of nice $L$-boxes, and note that there are several circumstances (according to the initial configuration of open boxes) that make an $ L $-box not nice.
    In Step 1 we rule out the $ L $-boxes that are not nice because of type-I open boxes by the domination of poor vertices due to $ (y,j) $ for some $ y\in\Pi_{\boldsymbol{k}} $ and $ 0\leq j\leq n-2 $; in Step 2 we rule out the $L$-boxes that are not nice because of type-II open boxes by the domination of poor vertices due to $ (y,-1) $ for some $ y\in\Pi_{\boldsymbol{k}} $.

    \textbf{Step 1}: $ 0\leq j\leq n-2 $. For a vertex $ x\in\Pi_{\boldsymbol{k}} $, we say that $x$ is not type-I nice if there exists $ 0\leq j\leq n-2 $ and $ z\in L_j\mathbb{Z}^d $ satisfying that $ B(x,20^{10d}L) $ intersects $ B(z,3\sqrt{K} L_{j}) $ while the latter is a type-I open box sampled under $ \mathbb{Q}_{n-1} $. Observe that
    \begin{equation*}
        \begin{aligned}
            &\quad\max_{I,J\subset\{1,\dots,d\}}\big|\sum_{i\in I}\boldsymbol{l}_i(\boldsymbol{k})-\sum_{j\in J}\boldsymbol{l}_j(\boldsymbol{k})\big|_{\infty} \\
            &=\max_{\stackrel{I,J\subset\{1,\dots,d\}}{I\cap J=\varnothing}}\big|9d\big(\sum_{i\in I}\response{\boldsymbol{k}_i}\boldsymbol{e}_i-\sum_{j\in J}\response{\boldsymbol{k}_j}\boldsymbol{e}_j\big)-3(|I|-|J|)\boldsymbol{k}\big|_{\infty}\leq12d,
        \end{aligned}
    \end{equation*}
    which implies $ \cup_{x\in\Pi_{\boldsymbol{k}}}B(x,10dL)\supset\mathbb{Z}^d $. Also observe that $ \min_{w\in B_{\Pi_{\boldsymbol{k}}}(x,1)\setminus\{x\}}|w-x|_{\infty}\geq3dL $.
    The preceding two observations motivate us to choose $ a_j=K^{j+1.5} $, and hence by $ a_j\cdot3dL>L+20^{10d}L+3\sqrt{K}L_j+10dL $ we know that for any $ x, x^{\prime}\in\Pi_{\boldsymbol{k}} $ with $ x\notin B_{\Pi_{\boldsymbol{k}}}(x^{\prime},a_j) $ and any $ z\in B(x,10dL),y\in B(x^{\prime},L) $, we have that $ B(z,3\sqrt{K} L_{j})\cap\mathsf{R}(y)=\varnothing $. In light of this, we let $ \{Z_{z,j}\}_{z\in\mathbb{Z}^d,0\leq j\leq n-2} $ be independent Bernoulli variables with expectation $ p_j $ and define
    \begin{equation*}
        X^{\prime}_{x,j}:=\mathbf{1}\big[\exists z\in \cap B(x,10dL){\rm\ such\ that\ }Z_{z,j}=1\big]
    \end{equation*}
    for $ x\in\Pi_{\boldsymbol{k}} $, $ 0\leq j\leq n-2 $ where $ \mathbf{1}[\cdot] $ is the indicator function. Then by~\cite[~Theorem 1.3]{Liggett1997DominationBP} there exists a positive constant $ \alpha_d=10^{-10d} $ depending only on $ d $ such that for all $ 0\leq j\leq n-2 $, $ \{X^{\prime}_{x,j}\}_{x\in\Pi_{\boldsymbol{k}}} $ can be stochastically dominated (because they are 1-dependent for all $ 0\leq j\leq n-2 $ since the dependence comes from the overlapping of $ 10dL $-boxes) by some independent Bernoulli variables $ \{X_{x,j}\}_{x\in\mathbb{Z}^d} $ such that
    \begin{equation*}
        \mathbb{E}X_{x,j}=\tilde{p}_j:= p^{\alpha_d}_j\leq p^{2d\alpha_d}e^{-\alpha_d K\cdot2^j}
    \end{equation*}
    for sufficiently large $K$.
    Then the preceding discussion provides a coupling such that
    \begin{equation}
        \label{eq:coupling_type-I}
        \{ x\in\Pi_{\boldsymbol{k}}: x \mbox{ is not type-I nice} \} \subset \{ x\in\Pi_{\boldsymbol{k}}: x\mbox{ is type-I poor} \}\,,
    \end{equation}
    where $ x $ is type-I poor if $ X_{y,j}=1 $ for some $ y,j $ with $ 0\leq j\leq n-2 $ and $ y\in B_{\Pi_{\boldsymbol{k}}}(x,a_j) $.

    \textbf{Step 2}: $ j=-1 $. Suppose the coupling is already constructed for $ 0\leq j\leq n-2 $ such that~\eqref{eq:coupling_type-I} holds. Then we could fix some $ u\in\Pi_{\boldsymbol{k}} $ which is not type-I poor. In what follows, by Conditions (1), (2), (3) we mean Conditions (1), (2), (3) in Definition~\ref{def:nice-box} with $ x=u $. Since the niceness of $ u $ is measurable with statuses of vertices in $ B(u, 20^{10d}L) $ and since $ u $ is not type-I poor, in order to lower-bound the probability of $ u $ being nice, it then suffices to consider type-II open boxes and closed vertices sampled in $ \mathbb{Q}_{n-1} $ (recall \response{Definition~\ref{def:Qn}}). We will use this without further notice in what follows. Thus, we have
    \begin{equation}
        \label{eq:nice_condition-1}
        \begin{aligned}
            \mathbb{P}[{\rm Condition\ }(1)]&=\mathbb{P}[\exists{\rm\ closed\ }y\in B(u,L)]\geq1-(1-q)^{(2L)^d}\\
            &\geq1-\exp\big[-Cp^d\cdot(2L)^d\big]\geq1-\delta^{\prime}/3
        \end{aligned}
    \end{equation}
    if $ q\geq Cp^d $ when $ C/K^d $ is sufficiently large, where $ \delta^{\prime}\in(0,1) $ is a constant to be chosen. Now condition on this closed $ y\in B(u,L) $ (as the marked vertex).
    Recalling the definition of type-II open boxes (and recalling \response{Definition~\ref{def:Qn}}), we let $ \{U_{z,k}\}_{z\in\mathbb{Z}^d,0\leq k\leq100^d} $ be independent Bernoulli variables with expectation $ 3^dp^{1+k/3} $.
    Since for any $ y $ the volume of $ \mathsf{R}(y) $ is at most $ d\cdot20^{10d}L(200d+1)^{d-1} $, we get
    \begin{equation}
        \label{eq:nice_condition-2}
        \begin{aligned}
            \mathbb{P}[{\rm Condition\ }(2)]&\geq1-d\cdot20^{10d}L(200d+1)^{d-1}\cdot \sum^{100^d}_{k=0}(2k+1)^d\cdot3^dp^{1+k/3}\\
            &\geq1-\frac{C_3(d)}{K}\geq1-\delta^{\prime}/3.
        \end{aligned}
    \end{equation}
    In addition, we get that
    \begin{equation}
        \label{eq:nice_condition-3}
        \begin{aligned}
            &\quad\mathbb{P}[{\rm Condition\ }(3)]\\
            &\geq1-\sum_{z\in B(u,10^{10d}L)}\mathbb{Q}_{n-1}\big[\sum_{x\in B(z,10d)}\sum^{100^d}_{k=0}\sum_{w\in B(x,k)}U_{w,k}\geq4d\big]\\
            &\geq1-(20^{10d}L)^d\cdot e^{-4dt}\cdot\prod^{100^d}_{k=0}\big[ 1+(e^{(20d+1)^d t}-1)\cdot3^dp^{1+k/3} \big]^{(3k+30d)^d}\\
            &\geq 1-\frac{C_4(d)}{p^dK^d}\cdot e^{-4dt}\cdot \exp\big( C_4(d)\cdot e^{C_4(d)t}\cdot p \big)\\
            &\geq1-\frac{C_4(d)}{p^dK^d}e^{-\frac{2d\log(1/p)}{C_4(d)}+C_4(d)\sqrt{p}}\geq1-\delta^{\prime}/3
        \end{aligned}
    \end{equation}
    for sufficiently small $ p $ and some constants $ C_3(d), C_4(d) $ depending only on $d$, where in the second inequality we applied Markov's inequality with exponential moments with the choice of $ t=\log(1/p)/(2C_4(d)) $ (note that $ U_{w,k} $ is counted at most $ (3k+30d)^d $ times for each fixed $ w\in\mathbb{Z}^d $ and $ 0\leq k\leq 100^d $, and there are at most $ (20d+1)^d $ choices for $ w $).
    Therefore, combining~\eqref{eq:nice_condition-1} with \eqref{eq:nice_condition-2} and~\eqref{eq:nice_condition-3} yields
    \begin{equation*}
        \mathbb{P}[u{\rm\ is\ nice}]\geq1-\exp\big[-Cp^d\cdot(2L)^d\big]-\frac{C_3(d)}{K}-\frac{C_4(d)}{p^dK^d}e^{-\frac{2d\log(1/p)}{C_4(d)}+C_4(d)\sqrt{p}}\geq1-\delta^{\prime}
    \end{equation*}
    for sufficiently large $ K_0=K_0(\delta,d) $, $ C=C(K,d) $ and any $ K\geq K_0 $.
    Now choosing $ \delta^{\prime} $ sufficiently small and using the stochastic dominance in~\cite{Liggett1997DominationBP} again, we derive that there exist independent Bernoulli variables $ \{ X_{x,-1} \}_{x\in\Pi_{\boldsymbol{k}}} $ with $ \mathbb{E}X_{x,-1}=\delta $ such that
    \begin{align*}
        &\quad\ \{ x\in\Pi_{\boldsymbol{k}}: x\mbox{ is not type-I poor} \} \cap \{ x\in\Pi_{\boldsymbol{k}}: \mbox{conditions (1), (2), (3) hold} \}^c \\
        &\subset \{ x\in\Pi_{\boldsymbol{k}}: X_{x,-1}=1 \}.
    \end{align*}
    This combined with~\eqref{eq:coupling_type-I} proves the lemma.
\end{proof}

Now we are ready to prove Lemma~\ref{lem:control-growth}.

\begin{proof}[Proof of Lemma~\ref{lem:control-growth}]
    Recall the $ \boldsymbol{k} $-boundary from~\eqref{eq:def-oriented-neighbor} and Definition~\ref{def:oriented_surface} of the oriented surface. By Lemmas~\ref{lem:oriented_surface_disconnect_origin_from_infinity} and~\ref{lem:oriented_surface_prevents_growth} and a union bound over all orientations, it suffices to prove that
    \begin{equation}
        \label{eq:4.9-0}
        \mathbb{Q}_{n-1}[\exists\,{\rm nice\ oriented\ surface\ }S_n(\boldsymbol{k})]\geq1-e^{-K\cdot2^{n}}/(200K)^d
    \end{equation}
    for any $ n\geq1 $ and $ \boldsymbol{k}\in\{-1,1\}^d $. In what follows we fix such $ n $ and $ \boldsymbol{k} $.

    By Lemma~\ref{lem:overwhelming_dominance}, for any $ \delta $ we have
    \begin{equation}
        \label{eq:count-path-00}
        \begin{aligned}
            &\quad\ \mathbb{Q}_{n-1}[\exists\,{\rm nice\ } S_n(\boldsymbol{k})]\\
            &\geq\mathbb{Q}_{n-1}^{\prime}[\exists\,\mbox{oriented surface } S_n(\boldsymbol{k}) \mbox{ such that each } x\in S_n(\boldsymbol{k}) \mbox{ is not poor}]
        \end{aligned}
    \end{equation}
    for sufficiently large $ K $ and sufficiently small $ p $, where $ \mathbb{Q}_{n-1}^{\prime} $ is a $ (\delta,n) $-overwhelming measure on $ \Pi_{\boldsymbol{k}} $.
    We will give a strategy (under $ \mathbb{Q}_{n-1}^{\prime} $ with some $ \delta\in(0,1) $) of finding $ S_n(\boldsymbol{k}) $ consisting of not-poor vertices and prove that with probability at least $ 1-e^{-K\cdot2^{n}}/(200K)^d $ such strategy succeeds by a duality argument.
    Recall $ \alpha_d $, $a_j$ and $ \tilde{p}_j $ for $ -1\leq j\leq n-2 $ from Definition~\ref{def:overwhelming}.

    \textbf{The design of the strategy.} The strategy is as follows.
    We define inductively a sequence of subsets $ \{D_i\subset A_n(\boldsymbol{k})\}_{i\geq0} $ starting from $ D_0:=\big\{x\in\Pi_{\boldsymbol{k}}:\exists y\in A_n^-(\boldsymbol{k}),x\preceq^{\prime}_{\boldsymbol{k}}y\big\} $, and aim to stop at some $ i\geq0 $ when the oriented boundary $ \partial^+_{\boldsymbol{k}}D_i $ is contained in $ A_n(\boldsymbol{k}) $ and does not contain any poor vertex. In this case we say that the strategy succeeds at $ S:=\partial^+_{\boldsymbol{k}}D_i $. Otherwise suppose $ E_i $ is the collection of poor vertices in $ \partial^+_{\boldsymbol{k}}D_i $, and we define
    \begin{equation}
        \label{eq:oriented-surface-2}
        D_{i+1}:=D_i\cup\big(\bigcup_{(y,j)\in I_i}\mathsf{D}_{n,\boldsymbol{k}}(y,j)\big)
    \end{equation}
    where (recall $ \mathsf{N}(x) $ from Definition~\ref{def:overwhelming} as the collection of pairs due to which $x$ is poor)
    \begin{equation*}
        I_i:=\cup_{x\in E_i}\mathsf{N}(x)\quad{\rm and}\quad\mathsf{D}_{n,\boldsymbol{k}}(y,j):=\{z\in\Pi_{\boldsymbol{k}}:z\preceq^{\prime}_{\boldsymbol{k}} w(y,j):=y+a_j\cdot6dL\boldsymbol{k}\}\cap A_n(\boldsymbol{k}).
    \end{equation*}
    Note that $ w(y,j):=y+a_j\cdot6dL\boldsymbol{k}=y+\pi_{\boldsymbol{k}}(La_j\cdot\sum^{d}_{i=1}\boldsymbol{e}_i)\in B_{\Pi_{\boldsymbol{k}}}(y,a_j) $ satisfies that $ z\preceq^{\prime}_{\boldsymbol{k}}w(y,j) $ for all vertices $ z\in B_{\Pi_{\boldsymbol{k}}}(y,a_j) $.
    If there does not exist any $ 0\leq i\leq\tau-1 $ such that the strategy succeeds at $ \partial^+_{\boldsymbol{k}}D_i $ before $ D_{\tau}\cap A^+_n(\boldsymbol{k})\neq\varnothing $ for some $ \tau\geq1 $, we say that the strategy fails.
    \begin{figure}[h]
        \centering
        \includegraphics[width=0.5\textwidth]{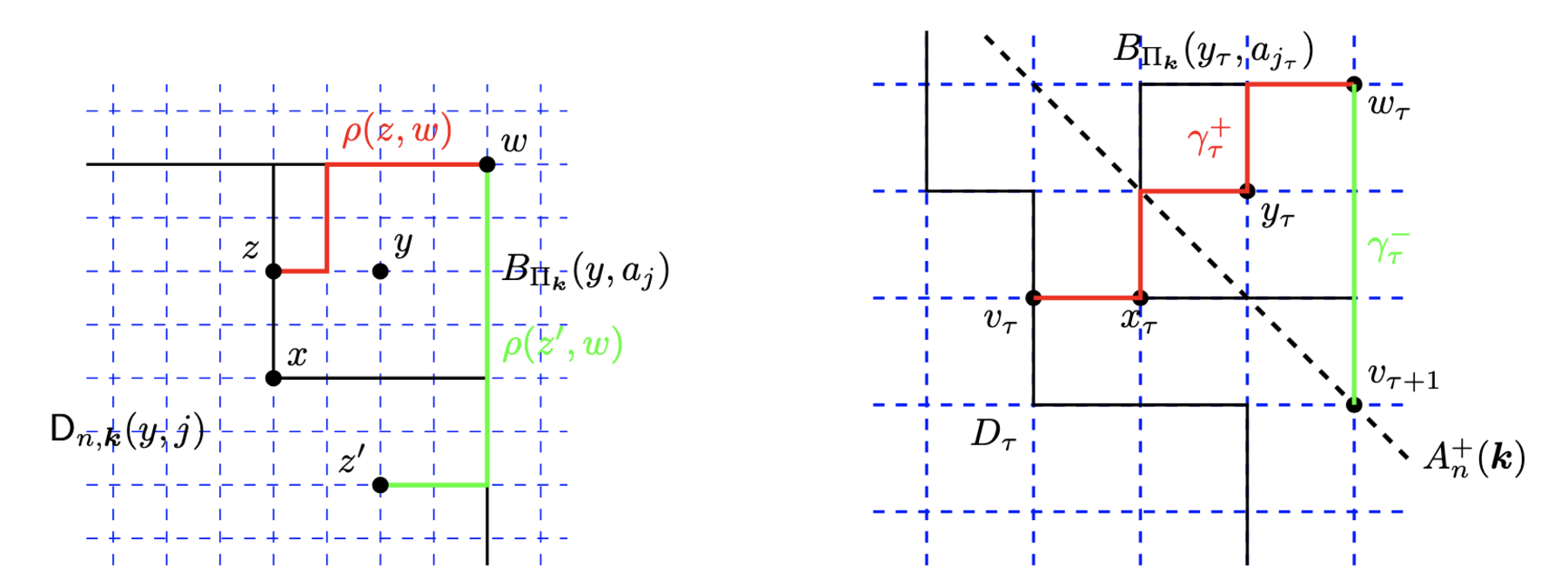}
        \caption{Illustrations for the duality analysis in the proof of Lemma~\ref{lem:control-growth} for $ d=2 $ and $ \boldsymbol{k}=(1,1) $. In both subfigures, the sublattice $ \Pi_{\boldsymbol{k}} $ (the dashed blue) is drawn orthogonally for better clarity. \emph{Left:} An illustration for the n.n.paths $ \rho(z, w) $ and $ \rho(z^{\prime}, w) $ where $ z\in B_{\Pi_{\boldsymbol{k}}}(y,a_j) $, $ z^{\prime}\in \mathsf{D}_{n,\boldsymbol{k}}(y,j) $ and $ w=w(y,j) $ for some poor vertex $ x\in\Pi_{\boldsymbol{k}} $ and some $ (y,j)\in \mathsf{N}(x) $. The path $ \rho(z,w) $ is colored in red and the path $ \rho(z^{\prime},w) $ is colored in green. \emph{Right:} An illustration for the construction of $ \gamma_{\tau} $. In the case shown here, $ a_{j_{\tau}}=1 $. Note that function $ |\pi^{-1}_{\boldsymbol{k}}(\cdot)|_1 $ is  increasing along $ \gamma^+_{\tau} $ (colored in red) from $ v_{\tau} $ to $ w_{\tau} $ and decreasing along $ \gamma^-_{\tau} $ (colored in green) from $ w_{\tau} $ to $ v_{\tau+1} $. The metric box $ B_{\Pi_{\boldsymbol{k}}}(y_{\tau},a_{j_{\tau}}) $ and the subset $ D_{\tau} $ are colored in black.}
        \label{fig:dual_path}
    \end{figure}

    \textbf{The duality analysis.} We next carry out a duality analysis on our search strategy above.
    See Figure~\ref{fig:dual_path} for an illustration of the following argument. An observation is that by definition of $ \preceq^{\prime}_{\boldsymbol{k}} $, for any $ z\in\mathsf{D}_{n,\boldsymbol{k}}(y,j) $ and for some poor $ x $ and $ (y,j)\in\mathsf{N}(x) $, there exists a n.n.path $ \rho(z,w(y,j))\subset\mathsf{D}_{n,\boldsymbol{k}}(y,j) $ from $ w(y,j) $ to $ z $ such that the function $ |\pi^{-1}_{\boldsymbol{k}}(\cdot)|_1 $ is decreasing along this path. In addition, if $ z\in B_{\Pi_{\boldsymbol{k}}}(y,a_j) $ we can (and will always) choose $ \rho(z,w(y,j)) $ to be contained in $ B_{\Pi_{\boldsymbol{k}}}(y,a_j) $, and in this case we have
    \begin{equation}
        \label{eq:count-path_-9}
        |\rho(z,w(y,j))|=L^{-1}(|\pi^{-1}_{\boldsymbol{k}}(w(y,j))|_1-|\pi^{-1}_{\boldsymbol{k}}(z)|_1)\leq 2da_j.
    \end{equation}
    See Figure~\ref{fig:dual_path} (left) for an illustration.
    If the strategy fails, then there exists $ \tau\geq0 $ such that $ I_i\neq\varnothing $ for all $ 0\leq i\leq\tau $ and $ D_{\tau+1} $ intersects $ A_n(\boldsymbol{k}) $.
    Choose some $ v_{\tau+1}\in D_{\tau+1}\cap A_n(\boldsymbol{k}) $.
    By~\eqref{eq:oriented-surface-2} there exists $ x_{\tau}\in E_{\tau} $ and $ (y_{\tau},j_{\tau})\in\mathsf{N}(x_{\tau}) $ such that $ v_{\tau+1}\in\mathsf{D}_{n,\boldsymbol{k}}(y_{\tau},j_{\tau}) $. Since $ x_{\tau} \in \partial^+_{\boldsymbol{k}}D_{\tau} $, we know that there exists $ v_{\tau}\in D_{\tau} $ adjacent to $ x_{\tau} $.
    Let $ w_{\tau}:= w(y_{\tau},j_{\tau}) $ and let $ \gamma^{-}_{\tau} $ denote the aforementioned n.n.path $ \rho(v_{\tau+1},w_{\tau}) $ connecting $ w_{\tau} $ to $ v_{\tau+1} $, and let $ \gamma^+_{\tau} $ denote the concatenated path of the edge $ (v_{\tau},x_{\tau}) $ and the aforementioned n.n.path $ \rho(x_{\tau}, w_{\tau}) $ from $ x_{\tau} $ to $ w_{\tau} $. Note that since $ x_{\tau}\in B_{\Pi_{\boldsymbol{k}}}(y_{\tau},a_{j_{\tau}}) $ and by definition of $ \rho(x_{\tau}, w_{\tau}) $ we know that all vertices on $ \rho(x_{\tau}, w_{\tau}) $ except $ v_{\tau} $ is contained in $ B_{\Pi_{\boldsymbol{k}}}(y_{\tau},a_{j_{\tau}}) $.
    Then the function $ |\pi^{-1}_{\boldsymbol{k}}(\cdot)|_1 $ is  increasing along $ \gamma^+_{\tau} $ (in which each edge has an end contained in the metric box $ B_{\Pi_{\boldsymbol{k}}}(y_{\tau},a_{j_{\tau}}) $) from $ v_{\tau} $ to $ w_{\tau} $ and decreasing along $ \gamma^-_{\tau} $ from $ w_{\tau} $ to $ v_{\tau+1} $. Concatenating $ \gamma^{+}_{\tau} $ and $ \gamma^{-}_{\tau} $ at $ w_{\tau} $, we get a new n.n.path $ \gamma_{\tau} $.
    See Figure~\ref{fig:dual_path} (right) for an illustration.
    Now inductively for any $ 0\leq i\leq\tau $, we can define $ \gamma_{i} $ similarly. Concatenate these n.n.paths into a single path (note that this is possible since the ending point $ v_{i} $ of $ \gamma_{i-1} $ is the same as the starting point of $ \gamma_{i} $ for $0\leq i\leq\tau$) denoted as $ \gamma $ and parameterize it in the sense that $ \gamma(0)=x_0\in D_0 $ and $ \gamma(|\gamma|)=v_{\tau+1} $ (and we write $ \mathcal{E}_{\gamma} $ as the event that such $ \gamma $ exists). The aforementioned observation implies that for each edge $ (\gamma(k-1), \gamma(k)) $ with $ 1\leq k\leq |\gamma| $, we have $ |\pi^{-1}_{\boldsymbol{k}}(\gamma(k))|_1>|\pi^{-1}_{\boldsymbol{k}}(\gamma(k-1))|_1 $ if and only if $ \gamma(k)\in\cup_{0\leq i\leq\tau}\gamma^{+}_{i} $. Define $ \mathrm{Poor}(\gamma) $ as the collection of vertices $ \gamma(k) $ satisfying $ |\pi^{-1}_{\boldsymbol{k}}(\gamma(k))|_1>|\pi^{-1}_{\boldsymbol{k}}(\gamma(k-1))|_1 $. Then we have
    \begin{equation}
        \label{eq:count-path_-4}
        \mathrm{Poor}(\gamma)\subset\cup_{0\leq i\leq\tau}B_{\Pi_{\boldsymbol{k}}}(y_i,a_{j_i}),
    \end{equation}
    and for each $ 0\leq i\leq\tau $, by~\eqref{eq:count-path_-9} and the definition of $ \gamma^+_i $ we get that
    \begin{equation}
        \label{eq:count-path_-7}
        |{\rm Poor}(\gamma)\cap B_{\Pi_{\boldsymbol{k}}}(y_i, a_{j_i})|=1+|\gamma^+_i|\leq 1+2da_{j_i}\leq 3da_{j_i}.
    \end{equation}
    Since
    \begin{equation*}
        |v_{\tau+1}|_1-|\gamma(0)|_1\geq(\frac{\sqrt{K}}{2}-10d^2)L_{n-1}\geq\sqrt{K}L_{n-1}/5>0,
    \end{equation*}
    by~\eqref{eq:L_1-of-pi^-1} we know that $ \gamma $ has length at least $ \frac{1}{18dL}\cdot\sqrt{K}L_{n-1}/5\geq K^{n+0.5}/(100d) $.
    By the definition of $ \pi_{\boldsymbol{k}} $ and the definition of $ {\rm Poor}(\gamma) $, the preceding display also implies
    \begin{equation}
        \label{eq:count-path_-5}
        \begin{aligned}
            0&<|\pi^{-1}_{\boldsymbol{k}}(v_{\tau+1})|_1 - |\pi^{-1}_{\boldsymbol{k}}(\gamma(0))|_1 = \sum_{k=1}^{|\gamma|}\big( |\pi^{-1}_{\boldsymbol{k}}(\gamma(k))|_1 - |\pi^{-1}_{\boldsymbol{k}}(\gamma(k-1))|_1 \big) \\
            &= L \big(|{\rm Poor}(\gamma)| - (|\gamma|-|{\rm Poor}(\gamma)|) \big) = L(2|{\rm Poor}(\gamma)|-|\gamma|).
        \end{aligned}
    \end{equation}
    Combining~\eqref{eq:count-path_-4},~\eqref{eq:count-path_-7} and~\eqref{eq:count-path_-5} we have
    \begin{equation}
        \label{eq:count-path_-2}
        \sum_{0\leq i\leq \tau}3da_{j_i} \geq |\mathrm{Poor}(\gamma)| \geq |\gamma|/2.
    \end{equation}
    This then calls for a union bound over all these boxes intersecting with $ \gamma $.

    Let $ \lambda_{-1}=\frac{1}{10} $ and $ \lambda_j:=\frac{10^{10^{2d}}}{\alpha_d}\cdot\max\{\frac{1}{K\cdot2^{j-1}},\frac{2^{n-j}}{K^{n-j-1}}\} $ for $ 0\leq j\leq n-2 $. Then we have
    \begin{equation}
        \label{eq:count-path-0}
        \sum_{j=-1}^{n-2}3da_j\cdot\frac{\lambda_j}{3da_j}\leq\frac{1}{10}+\frac{10^{10^{2d}}}{\alpha_d}\cdot\sum^{n-2}_{j=0}\big(\frac{1}{K\cdot2^{j-1}}+\frac{2^{n-j}}{K^{n-j-1}}\big)\leq\frac{1}{2}
    \end{equation}
    for sufficiently large $ K $. Recall that on $ \mathcal{E}_{\gamma} $ we have $ \mathrm{Poor}(\gamma)\subset\cup_{0\leq i\leq\tau-1}B_{\Pi_{\boldsymbol{k}}}(y_i,a_{j_i}) $. Let $ G_j $ be the event that $ |\{ 0\leq i\leq\tau-1: j_i=j \}| \geq \lceil\frac{\lambda_{j}}{3da_j}|\gamma| \rceil $. Then by~\eqref{eq:count-path_-2} and~\eqref{eq:count-path-0} we get that $ \cap^{n-2}_{j=-1}G^{\mathrm{c}}_j\subset\mathcal{E}^{\mathrm{c}}_{\gamma} $ and hence
    \begin{equation}
        \label{eq:count-path-1}
        \mathbb{Q}_{n-1}^{\prime}[\mathcal{E}_{\gamma}]\leq\sum_{j=-1}^{n-2}\mathbb{Q}_{n-1}^{\prime}[G_j].
    \end{equation}

    \textbf{The probability analysis.}
    Let $ \mathcal{P} $ be the collection of all possible realizations of $ \gamma $ (which is a random n.n.path on $\Pi_{\boldsymbol{k}}$), and we will next count $ P\in \mathcal{P} $ with some coarse-graining. Recall that for each $ P\in \mathcal{P} $ we have $ |P|\geq K^{n+0.5}/(100d) $.
    For each $ -1\leq j\leq n-2 $ and for $ P\in\mathcal{P} $, let $ P^j $ denote the collection of $ x\in3da_j\Pi_{\boldsymbol{k}} $ such that metric box $ B_{\Pi_{\boldsymbol{k}}}(x,3da_j) $ intersects $ P $. For a fixed $ j $ and $ x $, the probability that $ B_{\Pi_{\boldsymbol{k}}}(x,3da_j) $ contains a poor vertex due to some $ (y,j) $ for some $ y\in B_{\Pi_{\boldsymbol{k}}}(x, 6da_j) $ is at most $ (6da_j)^d\tilde{p}_j $. Let $ \{Y_{j,i}\}_{j\geq-1,i\geq1} $ be independent Bernoulli variables with expectation $ (6da_j)^d\tilde{p}_j $. Note that the cardinality of $ P^j $ is at most $ 10^d\cdot|P|/(3da_j) $ and there are at most $ K^{(n+2)(d-1)} $ choices of $ P(0) $ (among $ A^-_n(\boldsymbol{k}) $).
    In addition $ \frac{P^j}{3da_j} $ is a connected subset on $ \Pi_{\boldsymbol{k}} $, \response{and} thus can be encoded by a depth-first search, i.e., a nearest neighbor contour on $ \Pi_{\boldsymbol{k}} $ with length at most $ 2\cdot 10^d |P|/(3da_j) $. Therefore, the enumeration of $ P^j $ can be upper-bounded by the enumeration of such depth-first search contours, which in turn is bounded by $ (2d)^{2\cdot 10^d |P|/(3da_j)} $ (recall Lemma~\ref{lem:count_surrounding_contour} for an analogous result in the regular lattice). This implies that
    \begin{equation}
        \label{eq:lemma4.9-probability-bound}
        \begin{split}
            &\quad\ \mathbb{Q}_{n-1}^{\prime}[G_j]=\sum_{P\in\mathcal{P}}\mathbb{Q}_{n-1}^{\prime}[\gamma=P]\\
            &\leq\sum_{|P|\geq \frac{K^{n+0.5}}{100d}}K^{(n+2)(d-1)}(2d)^{2\cdot10^d\frac{|P|}{3da_j}}\mathbb{Q}_{n-1}^{\prime}\big[\sum_{i=1}^{2\cdot10^d\frac{|P|}{3da_j}}Y_{j,i}\geq \frac{\lambda_j}{3da_j}|P|\big]\\
            &\leq K^{(n+2)(d-1)}\sum_{|P|\geq \frac{K^{n+0.5}}{100d}}(2d)^{2\cdot10^d\frac{|P|}{3da_j}}e^{-t_j\lambda_j\cdot\frac{|P|}{3da_j}}\big(\mathbb{E}[\exp(t_jY_{j,1})]\big)^{2\cdot10^d\frac{|P|}{3da_j}}\\
            &=K^{(n+2)(d-1)}\sum_{|P|\geq \frac{K^{n+0.5}}{100d}}\exp\Big[ \frac{|P|}{3da_j}\big(2\cdot10^d\log(2d)-t_j\lambda_j+2\cdot10^d\log(1+(e^{t_j}-1)\cdot(6da_j)^d\tilde{p}_j)\big) \Big]\\
            &\leq K^{(n+2)(d-1)}\sum_{|P|\geq \frac{K^{n+0.5}}{100d}}\exp\Big[-\frac{|P|}{K^{n+0.5}}\cdot\frac{K^{n+0.5}}{3da_j}\cdot\big( t_j\lambda_j-10^{10^d}-10^{10^d}e^{t_j}\cdot a^d_j\tilde{p}_j \big)\Big].
        \end{split}
    \end{equation}
    Recall from Definition~\ref{def:overwhelming} that $ a_{-1}=1 $, $ \tilde{p}_{-1}=\delta $ and \response{$ a_j = K^{j+1.5} $}, $ \tilde{p}_j\leq p^{2d\alpha_d}e^{-\alpha_d K\cdot2^j} $ for $ 0\leq j\leq n-2 $. Let $ t_{-1}=10^{10^{2d}} $, we have
    \begin{equation*}
        t_{-1}\lambda_{-1}-10^{10^d} -10^{10^d}e^{t_{-1}}\cdot a^d_{-1}\tilde{p}_{-1}=10^{10^{2d}-1}-10^{10^d}-10^{10^d}e^{10^{10^{2d}}}\cdot\delta\geq10^{2d}
    \end{equation*}
    and hence
    \begin{equation}
        \label{eq:count_path_j=-1}
        \begin{aligned}
            \mathbb{Q}_{n-1}^{\prime}[G_{-1}]&\leq K^{(n+2)(d-1)}\sum_{|P|\geq K^{n+0.5}/(100d)}\exp\Big[-\frac{|P|}{K^{n+0.5}}\cdot\frac{K^{n+0.5}}{3d}\cdot10^{2d}\Big]\\
            &\leq \exp(-100K\cdot2^n)
        \end{aligned}
    \end{equation}
    for sufficiently small $ \delta $ and sufficiently large $ K $. As for $ 0\leq j\leq n-2 $, taking $ t_j=\alpha_d K\cdot2^{j-1} $ yields
    \begin{align*}
        &\quad\ t_j\lambda_j-10^{10^d}-10^{10^d}e^{t_j}\cdot a^d_j\tilde{p}_j\\
        &=\max\big\{\frac{10^{10^{2d}}K2^{j-1}}{K\cdot2^{j-1}},\frac{10^{10^{2d}}K2^{j-1}\cdot2^{n-j}}{K^{n-j-1}}\big\}-10^{10^d}-10^{10^d}e^{t_j}a^d_j\tilde{p}_j\\
        &\geq\frac{10^{10^{2d}}}{2}+\frac{10^{10^{2d}}}{2}\cdot\frac{K2^{n-1}}{K^{n-j-1}} -10^{10^d} -10^{10^d}K^{d(j+1.5)}p^{2d\alpha_d}e^{\alpha_d(K\cdot2^{j-1}-K\cdot2^j)}\\
        &\geq 10^{10^d} + 10^{10^d}\cdot\Big(\frac{2^{n}}{K^{n-j-2}}-K^{jd}p^{2d\alpha_d}e^{-\alpha_d K\cdot2^{j-1}}\Big)
            \geq10^{2d}\frac{2^n}{K^{n-j-2}}.
    \end{align*}
    Combined with~\eqref{eq:lemma4.9-probability-bound}, this implies that
    \begin{equation}
        \label{eq:count_path_j>=0}
        \begin{aligned}
            \mathbb{Q}_{n-1}^{\prime}[G_j]
            &\leq K^{(n+2)(d-1)}\sum_{|P|\geq K^{n+0.5}/(100d)}\exp\Big[-\frac{|P|}{K^{n+0.5}}\cdot\frac{K^{n+0.5}}{3da_j}\cdot10^{2d}\frac{2^n}{K^{n-j-2}}\Big]\\
            &\leq \exp(-100K\cdot2^n)
        \end{aligned}
    \end{equation}
    for sufficiently large $ K $ as well (where in the last step we used that $ |P|\geq K^{n+0.5}/(100d) $). Combing~\eqref{eq:count-path-1},~\eqref{eq:count_path_j=-1} and~\eqref{eq:count_path_j>=0} yields
    \begin{equation*}
        \begin{aligned}
            &\quad\,\mathbb{Q}_{n-1}^{\prime}[\exists\, S_n(\boldsymbol{k})\mbox{ such that each }x\in S_n(\boldsymbol{k}){\rm\ is\ not\ poor}]\\
            &\geq1-\mathbb{Q}_{n-1}^{\prime}[{\rm the\ aforementioned\ strategy\ fails}]\\
            &\geq1-\mathbb{Q}_{n-1}^{\prime}[\mathcal{E}_{\gamma}]\geq1-\sum^{n-2}_{j=-1}\mathbb{Q}_{n-1}^{\prime}[G_j]\\
            &\geq1-\sum^{n-2}_{j=-1}p^{2d}\exp(-100K\cdot2^n)\geq1-e^{-K\cdot2^{n}}/(200K)^d.
        \end{aligned}
    \end{equation*}
    Combined with~\eqref{eq:count-path-00} and the duality analysis, this yields~\eqref{eq:4.9-0} and hence the lemma.
\end{proof}


\section{Glauber evolution}
\label{sec:glauber}

In this section, we study the Glauber evolution of the RFIM in dimensions $d \geq 2$. In Sections~\ref{subsec:claim1} and~\ref{subsec:claim2}, we consider the zero-temperature case and prove Theorem~\ref{thm:1.3}. In Section~\ref{subsec:positive-glauber}, we introduce the positive-temperature version and show how to incorporate both the ground state evolution and the zero-temperature Glauber evolution into this model. In Section~\ref{subsec:open-prob}, we list some open problems.

Recall~\eqref{eq:def-hamiltonian} and recall that $M \in (-\infty, \infty)$ is the mean of the external field, $\epsilon > 0$ is the noise intensity, and $(h_v)_{v \in \mathbb{Z}^d}$ are i.i.d.\ standard Gaussian variables. Recall the zero-temperature Glauber evolution defined in Section~\ref{subsec:glauber-evolution}. An observation is that the Glauber evolution exhibits a nesting property stated as follows. For each time $ M\in\mathbb{R} $, define $ \Sigma(M) $ as the collection of spin configurations $ (\sigma_v(M))_{v\in\mathbb{Z}^d} $ such that for every $ v\in\mathbb{Z}^d $, $ \sigma_v(M)\cdot(M+\epsilon h_v+\sum_{u\sim v}\sigma_u(M))\geq0 $, \response{and moreover, when $M+\epsilon h_v+\sum_{u\sim v}\sigma_u(M) = 0$, we require that $\sigma_v(M) = 1$.} We will say that $ \Sigma(M) $ is a collection of local minimizers.
Then we can define the Glauber evolution starting from some $ \sigma(M):=(\sigma_v(M))_{v\in\mathbb{Z}^d}\in\Sigma(M) $ for all $ M\in\mathbb{R} $, and denote the corresponding increasing process by $ \sigma^{\rm Gl}_v(M^{\prime}; \sigma(M))_{v\in\mathbb{Z}^d, M^{\prime}\geq M} $. Note that by definition we have $ \sigma^{\rm Gl}_v(M^{\prime}; \sigma(M))_{v\in\mathbb{Z}^d}\in\Sigma(M^{\prime}) $ for all $ M^{\prime}\geq M $.
\begin{lemma}
    \label{lem:Glauber-nesting}
    For any $ M\leq M_1\leq M_2 $ and any $ \sigma(M):=(\sigma_v(M))_{v\in\mathbb{Z}^d}\in\Sigma(M) $, we have
    \begin{equation*}
        \sigma^{\rm Gl}_v(M_2; \sigma(M))=\sigma^{\rm Gl}_v(M_2; \sigma^{\rm Gl}(M_1; \sigma(M))) \quad \mbox{for all $ v\in\mathbb{Z}^d $},
    \end{equation*}
    where $ \sigma^{\rm Gl}(M_1; \sigma(M)) $ is the Glauber evolution configuration at $ M_1 $ with the starting configuration $ \sigma(M) $.
\end{lemma}
\begin{proof}
    Our proof is based on monotonicity of the evolution.
    Note that the configurations $ \sigma^{\rm Gl}_v(M_2; \sigma(M)) $ and $ \sigma^{\rm Gl}_v(M_2; \sigma^{\rm Gl}(M_1; \sigma(M))) $ are both local minimizers contained in $ \Sigma(M_2) $. Then on the one hand, by monotonicity of the Glauber evolution with respect to time (with the same starting configuration) we have $ \sigma^{\rm Gl}(M_1; \sigma(M))\geq \sigma(M) $, and hence by monotonicity with respect to the starting configuration we get that
    \begin{equation*}
        \sigma^{\rm Gl}_v(M_2; \sigma^{\rm Gl}(M_1; \sigma(M)))\geq\sigma^{\rm Gl}_v(M_2; \sigma(M)) \quad \mbox{for all $ v\in\mathbb{Z}^d $}.
    \end{equation*}
    On the other hand, by monotonicity of the Glauber evolution with respect to time (with the same starting configuration) we have $ \sigma^{\rm Gl}(M_1; \sigma(M))\leq\sigma^{\rm Gl}(M_2; \sigma(M)) $, and hence by monotonicity with respect to the starting configuration we get that $ \sigma^{\rm Gl}_v(M_2; \sigma^{\rm Gl}(M_1; \sigma(M)))\leq\sigma^{\rm Gl}_v(M_2; \sigma(M)) $ for all $ v\in\mathbb{Z}^d $, which proves the lemma.
\end{proof}
In light of Lemma~\ref{lem:Glauber-nesting}, for a fixed time $M$, its spin configuration can also be given as follows: set $\sigma_v = -1$ for all $v \in \mathbb{Z}^d$, and check the status of each minus spin $\sigma_v$ to see whether flipping it from minus to plus will (weakly) decrease the Hamiltonian with external field $M$. If the answer is yes (equivalently, if $\sum_{u \sim v} \sigma_u + M + \epsilon h_v \geq 0$), then we flip the spin and continue to check other spins until no more spins can be flipped. By comparing the Glauber evolution with polluted bootstrap percolation (they are very similar but still slightly different), we will prove the two claims in Theorem~\ref{thm:1.3} separately in Sections~\ref{subsec:claim1} and \ref{subsec:claim2}.

\subsection{Proof of Claim~\eqref{thm1.3-claim1} in Theorem~\ref{thm:1.3}}\label{subsec:claim1}

This part essentially follows from Theorem~\ref{thm:1.4}. Fix $d \geq 2$ and $M < c_d$. Recall the setup of polluted bootstrap percolation. We say a vertex $v \in \mathbb{Z}^d$ is closed if $\epsilon h_v + M < 0 $, open if $\epsilon h_v + M - 2 \geq 0$, and empty otherwise. In particular, at time $M$, a closed vertex becomes plus only if it has at least $d + 1$ plus neighbors, and an empty vertex becomes plus only if it has at least $d$ plus neighbors.
\response{We consider a variant of polluted bootstrap percolation where the rule is adjusted such that a closed vertex becomes open if and only if it has at least $d+1$ open neighbors, and an empty vertex becomes open if and only if it has at least $d$ open neighbors.}
Then we see that a vertex is plus in the Glauber evolution only if it becomes open in this variant of polluted bootstrap percolation. Let $q = \mathbb{P}[\epsilon h_v + M < 0]$ and $p = \mathbb{P}[\epsilon h_v + M - 2 \geq 0]$. Then, \response{by $ M<c_d $ and a standard Gaussian tail estimate}, we have $\lim_{\epsilon \rightarrow 0} \frac{q}{p^d} = +\infty$ and $\lim_{\epsilon \rightarrow 0} p = 0$. Applying Theorem~\ref{thm:1.4} (for this variant, see Remark~\ref{rmk:polluted-bootstrap}) yields Claim~\eqref{thm1.3-claim1} in Theorem~\ref{thm:1.3}.

\subsection{Proof of Claim~\eqref{thm1.3-claim2} in Theorem~\ref{thm:1.3}}
\label{subsec:claim2}

Fix $d \in \{2,3\}$ and $M > c_d$. Let $p = \mathbb{P}[\epsilon h_v + M - 2 \geq 0]$, $q = \mathbb{P}[\epsilon h_v + M < 0]$, and $L = \lfloor \frac{- d \log p}{p } \rfloor$. We say a vertex $v$ is bad if $\epsilon h_v + M < 0$ and good if $\epsilon h_v + M - 2 \geq 0$. In particular, a non-bad vertex becomes plus if $d$ of its neighbors are plus, and a good vertex becomes plus if $(d-1)$ of its neighbors are plus. Consider the cubes $B_x := x + [0,L)^d \cap \mathbb{Z}^d$ for $x \in L \mathbb{Z}^d$. Now we define a modified polluted bootstrap percolation\footnote{\response{Recall that in polluted bootstrap percolation with threshold $r$, an empty vertex becomes open if and only if it has at least $r$ open neighbors. In modified polluted bootstrap percolation, the threshold satisfies $1 \leq r \leq d$, and an empty vertex becomes open if and only if it has open neighbors in at least $r$ different directions. In particular, it is more difficult for an empty vertex to become open than in the non-modified version with the same threshold.}} with threshold $r = d - 1$ (as discussed in~\cite{GH-polluted-all}) on these boxes. We say such a box $B_x$ is empty if it does not contain any bad vertices and there is at least one good vertex in each row or column (namely, a one-dimensional axis-aligned sector of $B_x$ containing $L$ vertices). Then, we have
\begin{align*}
    \mathbb{P}[\mbox{$B_x$ is empty}] &\geq \mathbb{P}[\mbox{$B_x$ has no bad vertices}] \times \prod_{J \subset B_x} \mathbb{P}[\mbox{at least one good vertex in $J$}] \\
    &= (1 - q)^{L^d} \times (1 - (1-p)^L)^{d L^{d-1}}
\end{align*}
where $J$ ranges over all the rows and columns of $B_x$, and the first inequality follows from the fact that all these events are increasing. Since $d(2-M)^2 < M^2$, we have $\lim_{\epsilon \rightarrow 0} qL^d = 0$. By simple calculations, we obtain that $ \lim_{\epsilon \rightarrow 0} \mathbb{P}[B_x\mbox{ is empty}] = 1 $. We say a box is open if each vertex $ v $ in this box satisfies $\epsilon h_v + M - 2d \geq 0$. The probability of a box being open is small but positive. By definition, at time $M$, all the vertices in an open box are plus. Moreover, if an empty box $B_x$ has $(d-1)$ neighboring boxes that are all plus and reside in different directions with respect to $B_x$, then all the vertices in $B_x$ will become plus. The reason is that there will be a one-dimensional sector of $B_x$ such that each vertex in this sector is neighboring to all these $(d-1)$ boxes, and thus, neighboring to $(d-1)$ plus spins. Since there is one good vertex in this sector, it will become plus and trigger other vertices in this sector to become plus (recall that a vertex in an empty box becomes plus if $d$ of its neighbors are plus). By the same mechanism, all the vertices in this box will become plus. If a box is neither open nor empty (in fact an open box must be empty), we say it is closed.

Consider the evolution of boxes that have all plus signs. This is reminiscent of the modified polluted bootstrap percolation. Namely, we have open and closed vertices (corresponding to open and closed boxes) on $\mathbb{Z}^d$ initially. The open and closed vertices remain the same, and an empty vertex (corresponding to a box that is empty but not open) becomes open if it has at least $(d-1)$ neighbors in different directions 
that are open. From the above argument, we see that the spins are all plus in the boxes corresponding to open vertices in the evolution of modified polluted bootstrap percolation. When $d \in \{2,3\}$, the origin will become open with a probability close to 1 as the closed probability tends to 0 in the modified polluted bootstrap percolation (no matter how small the open probability is). Specifically, the two-dimensional case is trivial as there would be an infinite connected component of empty or open vertices containing the origin with probability close to 1, and this infinite connected component must be open since it almost surely contains an open vertex. The three-dimensional case follows from~\cite[Theorem 1]{GH-polluted-all}. This proves Claim~\eqref{thm1.3-claim2} in Theorem~\ref{thm:1.3}.

\subsection{Glauber evolution at positive temperature}\label{subsec:positive-glauber}

In this subsection, we define the Glauber evolution at positive temperature and explain how the ground state evolution and the zero-temperature Glauber evolution discussed in Sections~\ref{subsec:gs-evolution} and \ref{subsec:glauber-evolution} can be incorporated into this single model. Consider the finite graph $G = \TT_N^d$, and let $T \geq 0$ be the temperature, $\alpha > 0$ be the update rate, and $\epsilon \geq 0$ be the noise intensity.

We first define the model. Let $L > 0$ be a constant that will tend to infinity. Fix a realization of the noise $(h_v)_{v \in \TT_N^d}$. At time $-L$, set $\sigma_v = -1$ for all $v \in \TT_N^d$. Define a random process evolving from $-L$ to $L$ as follows: at each vertex $v$, there is a time clock ringing at rate $\alpha$. When it rings, the new spin configuration agrees with the old one except at $v$, and the spin at $v$ is chosen to be\footnote{We choose this rate function such that the invariant distribution with respect to it at time $M$ is the RFIM measure with external field $M$. One can also consider other rates, such as the Metropolis transition rate.}
\begin{equation}\label{eq:rate-positive-glauber}
    \mbox{plus with probability $\frac{\exp(-\frac{1}{T} H_{G, M, \epsilon h}(\sigma^+))}{\exp(-\frac{1}{T} H_{G, M, \epsilon h}(\sigma^+)) + \exp(-\frac{1}{T} H_{G, M, \epsilon h}(\sigma^-))}$, and otherwise minus.}
\end{equation}Here, $\sigma^+$ (resp.\ $\sigma^-$) denotes the spin configuration obtained from $ \sigma $ by putting $ + $ (resp. $ - $) at $v$. Let $(\sigma_M^L)_{-L \leq M \leq L}$ denote the resulting random process on spin configurations. Then, for any fixed time $R>0$, the law of $(\sigma_M^L)_{-R \leq M \leq R}$ converges in distribution as $L$ tends to infinity.\footnote{The convergence in distribution holds because for each realization $(h_v)$, $\sigma_{-M}^L$ is equal to the all-minus spin configuration with high probability as $M$ and $L$ both tend to infinity. Therefore, for different $L' > L$, we can find a coupling such that $\sigma_{-L}^{L'} = \sigma_{-L}^L$ with probability $1 - o_L(1)$ and the evolutions are the same afterward. This shows that for any fixed $R$, the law of $(\sigma_M^L)_{-R \leq M \leq R}$ converges as $L$ tends to infinity.} Moreover, it is easy to see that the limiting random process evolves according to the previous rule and is consistent for different $R$. Therefore, we can define a random process indexed by $ \mathbb{R}$, denoted by $(\sigma_M)_{M \in \mathbb{R}}$, which evolves according to~\eqref{eq:rate-positive-glauber}. This random process depends on the realization of $(h_v)$ as well as $T, \epsilon$, and $\alpha$, and will be referred to as the Glauber evolution at positive temperature.

For each fixed $M$, as $\alpha$ tends to infinity, the law of $\sigma_M$ converges in distribution to the invariant distribution for the rate function~\eqref{eq:rate-positive-glauber} which is the RFIM measure at time $M$. Notice that in general, the law of $\sigma_M$ is not the same as the RFIM measure at $M$, and is strictly stochastically dominated by the latter. Next we show that the ground state evolution defined in Section~\ref{subsec:gs-evolution} and the zero-temperature Glauber evolution defined in Section~\ref{subsec:glauber-evolution} (while we defined the case of $\mathbb{Z}^d$ there, the case of $\TT_N^d$ can be defined analogously) can be obtained from this model by appropriately sending $\alpha$ and $\frac{1}{T}$ to infinity. 

\begin{lemma}\label{lem:positive-glauber}
    For any fixed $M_1 < M_2 < \ldots < M_k$, the following hold for almost every realization of $(h_v)$:
    \begin{itemize}
        \item By first sending $\alpha$ to infinity and then sending $T$ to 0, the law of $(\sigma_{M_1}, \ldots, \sigma_{M_k})$ converges in distribution to that of the ground state evolution.

        \item By first sending $T$ to 0 and then sending $\alpha$ to $\infty$, the law of $(\sigma_{M_1}, \ldots, \sigma_{M_k})$ converges in distribution to that of the zero-temperature Glauber evolution.
    \end{itemize}
\end{lemma}

\begin{proof}
    We first send $\alpha$ to infinity. Then, the law of $(\sigma_{M_1}, \ldots, \sigma_{M_k})$ converges to the law of independent samples from the RFIM measures at times $M_1,\ldots, M_k$. If we further send $T$ to 0, they become the ground state configurations at these times.
    
    Instead, if we first send $T$ to 0, then the rate function in~\eqref{eq:rate-positive-glauber} degenerates to choosing the spin that minimizes the corresponding Hamiltonian. We can define a new process using this transition rule via the same approximation procedure as that of $(\sigma_M)_{M \in \mathbb{R}}$, denoted by $(\widetilde \sigma_M)_{M \in \mathbb{R}}$. Since there are only countably many rings, there is no tie in the transitions of $(\widetilde \sigma_M)$ for almost every realization of $(h_v)$. We claim that as $T$ tends to 0, the law of $(\sigma_{M_1}, \ldots, \sigma_{M_k})$ converges to the law of $(\widetilde \sigma_{M_1}, \ldots, \widetilde \sigma_{M_k})$. This is because for any fixed realization of $(h_v)$ and sufficiently large $M$, $\sigma_{-M}$ and $\widetilde \sigma_{-M}$ equal the all-minus spin configuration with probability $1 - o_M(1)$ (this probability is independent of $T$ for small $T$). Starting from the all-minus spin configuration at $-M$, we can find a coupling such that these two processes are the same before time $M_k$ with probability $1 - o_T(1)$ (this may depend on $M$) as $T$ tends to 0. By first sending $T$ to 0, and then sending $M$ to infinity, we get the claim. If we further send $\alpha$ to infinity, the distribution of $(\widetilde \sigma_{M_1}, \ldots, \widetilde \sigma_{M_k})$ converges to that of the zero-temperature Glauber evolution.
\end{proof}

\subsection{Open problems}\label{subsec:open-prob}

We conclude with some open problems related to the Glauber evolution. 

\begin{prob}\label{prob:threshold}
Prove Claim~\eqref{thm1.3-claim2} in Theorem~\ref{thm:1.3} for $d \geq 4$.
\end{prob}

From the proof in Section~\ref{subsec:claim2}, we see that the condition $d \in \{2,3\}$ is only used in the last step, where we need the fact that the final density of modified polluted bootstrap percolation with threshold $r = d-1$ tends to 1 as the closed probability tends to 0 (no matter how small the open probability is). This claim was first stated in~\cite[Conjecture 4.6]{Morris-bootstrap} although it is for the non-modified version (where a vertex becomes open if and only if it has at least $d-1$ open neighbors). Claim~\eqref{thm1.3-claim2} for $d \geq 4$ would follow from this conjecture for the modified version.

Recall that in the ground state evolution, we proved that there is no global avalanche for $d = 2$, while for $d \geq 3$, a global avalanche occurs (on the finite box) when the noise intensity $\epsilon$ is small. For the zero-temperature Glauber evolution, similar to Proposition~\ref{prop:2d-large-m}, we know that there is no global avalanche for sufficiently large $\epsilon$. However, the phenomenon for small $\epsilon$ remains an open question, and we feel that this is a challenging problem.

\begin{prob}
    Prove or disprove the existence of a global avalanche in the zero-temperature Glauber evolution for small $\epsilon$. Is there a transition in dimension?
\end{prob}

In this paper, we studied the zero-temperature Glauber evolution, but we also proposed a positive-temperature version, which would be interesting to investigate further. Compared to the zero-temperature version, one main difficulty is that spins will flip between plus and minus, making the analysis more subtle.

\begin{prob}
    Study the properties of the positive-temperature Glauber evolution, including the existence of a global avalanche, spin correlation with respect to times, mixing properties, etc.
\end{prob}

\bibliographystyle{alpha}
\bibliography{ref_arXiv}

\end{document}